\documentclass[10pt,a4paper]{amsart}

\usepackage{amsmath,amssymb}
\usepackage{ulem}
\usepackage[dvipdfmx]{graphicx,xcolor} 
\usepackage{amscd} 
\usepackage{array} 

\newtheorem{theorem}{Theorem}[section]

\newtheorem{lemma}[theorem]{Lemma}
\newtheorem{proposition}[theorem]{Proposition}
\newtheorem{corollary}[theorem]{Corollary}
\newtheorem{question}[theorem]{Question}

\theoremstyle{definition}

\newtheorem{example}[theorem]{Example}

\theoremstyle{remark}
\newtheorem{remark}[theorem]{Remark}

\numberwithin{equation}{section}

\makeatletter
\@namedef{subjclassname@2020}{\textup{2020} Mathematics Subject Classification}
\makeatother

\begin{document}

\title[Dynamics of superattracting skew products]
{Dynamics of superattracting skew products on the attracting basins: 
B\"{o}ttcher coordinates and plurisubharmonic functions}

\author[K. Ueno]{Kohei Ueno}
\address{Daido University, Nagoya 457-8530, Japan}
\curraddr{}
\email{k-ueno@daido-it.ac.jp}
\thanks{This work was supported by the Research Institute for Mathematical Sciences, 
           a Joint Usage/Research Center located in Kyoto University.}


\date{}
\dedicatory{}

\begin{abstract}
We study the dynamics of a superattracting skew product $f$ on the attracting basin. 
As the first strategy,
we find out forward $f$-invariant wedge-shaped regions in the basin,
on some of which $f$ is conjugate to monomial maps,
and consider whether the unions of all the preimages of the regions 
coincide with the basin.
As the second strategy,
we show the existence and properties of several kinds of plurisubharmonic functions of $f$,
the main functions of which are induced from the B\"{o}ttcher coordinates,
and investigate the asymptotic behavior of the functions toward the boundaries of the unions.
Consequently,
we obtain a plurisubharmonic function
on the complement of specific fibers in the basin,
which is continuous and pluriharmonic on open and dense subsets of the complement and
describes an certain weighted vertical dynamics well.
\end{abstract}

\maketitle

\section{Introduction}

\subsection{Background}

The dynamics around superattracting fixed points of holomorphic functions is 
completely understood in dimension one.
Let $p(z) = a z^{\delta} + O(z^{\delta + 1})$ be a holomorphic function
defined on $\mathbb{C}$, where $a \neq 0$ and $\delta \geq 2$.
Then it has a superattracting fixed point at the origin.
B\"{o}ttcher's theorem \cite{b} indicates that the limit
\[
\varphi_p (z) = 
\lim_{n \to \infty} p_0^{-n} \circ p^n (z)
\]
exists on a small neighborhood of the origin,
where $p^n$ is the $n$-th iterate of $p$,
$p_0 (z) = a z^{\delta}$ and the branch is taken 
such that $p_0^{-n} \circ p_0^n = id$.
This limit is called the B\"{o}ttcher coordinate for $p$ at the origin,
and it conjugates $p$ to $p_0$. 
We refer \cite{m} for details. 
We remark that
the order of $p^n$ at the origin is clearly $\delta^n$,
which is also called the attraction rate of $p^n$.
Let $A_p$ be the attracting basin of $p$ at the origin,
and $E_p$ the union of all the preimages of the origin.
We can extend $\varphi_p$ from the neighborhood to $A_p$
until it meets the other critical points than the origin.
On the other hand,
the limit
\[
G_p (z) =
\lim_{n \to \infty} \dfrac{1}{\delta^n} \log | p^n (z) |
\]
exists on $A_p$.
It is subharmonic and negative on $A_p$
and harmonic on $A_p - E_p$. 
Note that $G_p = \log | \varphi_p |$ on the neighborhood of the origin.
We remark that similar statements hold even if $p$ is a holomorphic germ. 

Several studies have been made
toward the generalization of B\"{o}ttcher's theorem to higher dimensions.
For example,
Ushiki~\cite{ushiki}, Ueda~\cite{ueda} and
Buff {\it et al}~\cite{bek} studied the case in which holomorphic germs,
with superattracting fixed points, have the B\"{o}ttcher coordinates
on neighborhoods of the points.
The germs in \cite{ushiki} are conjugate to monomial maps, 
whereas the germs in \cite{ueda} and \cite{bek} are conjugate to 
homogeneous and quasihomogeneous maps,
respectively.

However, 
B\"{o}ttcher's theorem does not extend to higher dimensions entirely
as pointed out by Hubbard and Papadopol \cite{hp}.
Let $f$ be a two dimensional holomorphic germ 
with a superattracting fixed point at the origin.
Unlike the one dimensional case,
in which a superattracting fixed point of a holomorphic germ is a isolated critical point
and forward invariant under the germ,
the origin is contained in the critical set of $f$, 
which may have a singularity at the origin and is not forward invariant under $f$ in general.
Hence it is difficult or impossible to take a nice normal form of $f$
on a neighborhood of the origin. 

Rigidity is a keyword 
for the study of the local dynamics of superattracting germs.
We say that $f$ is rigid if 
the critical set is contained in a totally invariant set 
with normal crossing singularities.
Favre \cite{f} introduced and classified attracting rigid germs.
Favre and Jonsson \cite{fj} proved that
any two dimensional superattracting germ
can be blown up to a rigid germ,
and hence the original germ is rigid
on an open set whose closure contains the fixed point.
Using the normal form on the open set,
they proved that the asymptotic attraction rate is a quadratic integer,
gave an inequality on the attraction rates, and
constructed a dynamically nice plurisubharmonic function 
defined on the neighborhood of the fixed point.
Ruggiero \cite{r} extended their results to more general cases.

We restrict our attention to superattracting skew products in this paper.
A skew product is a map of the form $f(z,w) = (p(z), q(z,w))$.
The fundamental properties of the dynamics of polynomial skew products is
well studied and summarized in \cite{j} and \cite{fg},
in which the main topics are the Green functions, currents and measures.
Lilov \cite{l} studied the local and semi-local dynamics of 
holomorphic skew products near a superattracting invariant fiber.
As natural extensions of the one dimensional results,
he obtained nice normal forms on neighborhoods of periodic points
which are geometrically attracting, parabolic and Siegel on fiber direction,
except the superattracting case.
See also \cite{ps} and \cite{pr}
for the dynamics of skew products near an invariant fiber
of different types.
Whereas the dynamics of skew products is mild 
in the sense that it can be regarded as 
the intermediate between the one and two dimensional dynamics,
it has very complicated aspects and
gives new phenomena of the dynamics in dimension two. 
For example,
Astorg {\it et al} \cite{abdpr} 
found polynomial skew products with wandering domains,
which is the first example of non-invertible
polynomial maps with wandering domains. 
See also \cite{abp} and \cite{ab}
for other types of polynomial skew products with wandering domains.
Furthermore,
polynomial skew products are used in \cite{d} and \cite{t}
to construct robust bifurcations, i.e.
open sets contained in a bifurcation locus.

\subsection{Previous results: B\"{o}ttcher coordinates}

We have studied the dynamics of superattracting skew products
and succeeded to construct B\"{o}ttcher coordinates 
on invariant wedge-shaped regions in \cite{u1}, \cite{u2}, \cite{u3} and \cite{ueno}.
Let us briefly review main results in \cite{ueno},
which is a consequence of the studies.
Let $f(z,w) = (p(z), q(z,w))$ be a skew product 
with a superattracting fixed point at the origin.
More precisely,
we assume that the eigenvalues of $Df(0)$ are both zero;
this assumption is called nilpotent in \cite{r},
whereas the assumption $Df(0) = 0$ is imposed in \cite{fj}.
Hence we may write 
$p(z) = a z^{\delta} + O(z^{\delta + 1})$ and
$q(z,w) = b z + \sum_{i, j \geq 0, i+j \geq 2} b_{ij} z^{i} w^{j}$,
where $a \neq 0$ and $\delta \geq 2$.
Let $bz = b_{10} z^1 w^0$ and $q(z,w) = \sum b_{ij} z^{i} w^{j}$ for short.

It is clear that the dominant term of $p$ is $a z^{\delta}$.
We found in \cite{ueno} that
the order $\delta$ of $p$ and the Newton polygon of $q$ assign
a term $b_{\gamma d} z^{\gamma} w^d$ of $q$ 
that is dominant on a wedge-shaped region $U$, where
$U = \{ |z|^{l_1 + l_2} < r^{l_2} |w|, |w| < r|z|^{l_1} \}$
for some rational numbers $0 \leq l_1 < \infty$ and $0 < l_2 \leq \infty$,
and that preserves $U$.
Let $f_0(z,w) = (a z^{\delta}, b_{\gamma d} z^{\gamma} w^d)$.

\begin{lemma}[\cite{ueno}] \label{main lemma for previous reselts}
If $d \geq 2$ or if $d = 1$ and $\delta \neq T_k$ for any $k$, then
$f \sim f_0$ on $U$ as $r \to 0$, and
$f(U) \subset U$ for small $r >0$.
\end{lemma}

The notation $f \sim f_0$ on $U$ as $r \to 0$ means that
the ratios of the first and second components of $f$ and $f_0$
tends to $1$ on $U$ as $r \to 0$, and
the definition of $T_k$ is given below.
Moreover,
we proved that the limit
\[
\phi (z,w) = 
\lim_{n \to \infty} f_0^{-n} \circ f^n (z,w)
\]
exists on $U$, 
where the branch is taken such that $f_0^{-n} \circ f_0^n = id$.

\begin{theorem}[\cite{ueno}] \label{main thm for previous reselts}
If $d \geq 2$ or if $d = 1$ and $\delta \neq T_k$ for any $k$, then
$\phi$ is biholomorphic on $U$ for small $r > 0$,
which conjugates $f$ to $f_0$.
Moreover,
$\phi \sim id$ on $U$ as $r \to 0$.
\end{theorem}

We call $\phi$ the B\"{o}ttcher coordinate for $f$ on $U$,
although the second component of $\phi$ 
looks like the K\oe nigs coordinate if $d = 1$.

Let us denote the precise definitions in the results.
We define the Newton polygon $N(q)$ of $q$ as 
the convex hull of the union of $D(i,j)$ with $b_{ij} \neq 0$,
where $D(i,j) = \{ (x,y) \ | \ x \geq i, y \geq j \}$. 
Let $(n_1, m_1)$, $(n_2, m_2)$, $\cdots, (n_s, m_s)$ be the vertices of $N(q)$,
where $n_1 < n_2 < \cdots < n_s$ and $m_1 > m_2 > \cdots > m_s$.
Let $T_k$
be the $y$-intercept of the line $L_k$ passing through 
the vertices $(n_k, m_k)$ and $(n_{k+1}, m_{k+1})$
for each $1 \leq k \leq s-1$.

\vspace{1mm}
\begin{itemize}
\item[\underline{Case 1}] 
  If $s = 1$, then $N(q)$ has the only one vertex, which is denoted by $(\gamma, d)$. \\
  We define $l_1 =  l_2^{-1} = 0$ and so $U = \{ |z| < r, |w| < r \}$.
\end{itemize}

The results for Case 1 are classical and 
the proofs are simpler than the other cases.
We remark that 
$f$ in Case 1 is rigid and 
belongs to class 6 or class 4 in \cite{f} if $d \geq 2$ or $d = 1$,
respectively,
and a description on B\"{o}ttcher coordinates is given in pp.498-499 in \cite{f}.

Therefore,
we only deal with the case $s > 1$ hereafter,
which is divided into the following three cases:

\vspace{1mm}
\begin{itemize}
\item[\underline{Case 2}] 
  If $s > 1$ and $\delta \leq T_{s-1}$, then  we define
  \[
  (\gamma, d) = (n_s, m_s), \ 
  l_1 = \frac{n_s - n_{s-1}}{m_{s-1} - m_s} 
  \text{ and } l_2^{-1} = 0.
  \] 
  Hence $U = \{ |z| < r, |w| < r |z|^{l_1} \}$. 
  \vspace{1mm}
\item[\underline{Case 3}] 
  If $s > 1$ and $T_1 \leq \delta$, then  we define
  \[
  (\gamma, d) = (n_1, m_1), \
  l_1 = 0 
  \text{ and } l_2 = \frac{n_2 - n_1}{m_1 - m_2}.
  \] 
  Hence $U = \{ |z|^{l_2} < r^{l_2} |w|, |w| < r \} = \{ r^{-l_2} |z|^{l_2} < |w| < r \}$.  
  \vspace{1mm}
\item[\underline{Case 4}]
  If $s > 2$ and $T_k \leq \delta \leq T_{k-1}$ for some $2 \leq k \leq s-1$, then  we define
  \[
  (\gamma, d) = (n_k, m_k), \ 
  l_1 = \frac{n_k - n_{k-1}}{m_{k-1} - m_k} 
  \text{ and } l_1 + l_2= \frac{n_{k+1} - n_k}{m_k - m_{k+1}}.
  \] 
  Hence  $U = \{ |z|^{l_1 + l_2} < r^{l_2} |w|, |w| < r|z|^{l_1} \} = \{ r^{-l_2} |z|^{l_1 + l_2} < |w| < r|z|^{l_1} \}$. 
\end{itemize}
\vspace{1mm}

We remark that
if $\delta > T_1$ and $d = 1$, then
$f$ is rigid and belongs to class 4 in \cite{f}
and hence $f$ is conjugate to $f_0$ not only on $U$
but also on a neighborhood of the origin;
see \cite[pp.44-45]{ueno} for a detail.


Note that
if $\delta = T_k$ for some $k$, 
then there are two different dominant terms of $q$, 
whereas there is the only one dominant term
if $\delta \neq T_k$ for any $k$.
Moreover,
if the two dominant terms satisfy the degree condition,
then we have two disjoint invariant regions 
on which $f$ is conjugate to each of the two different monomial maps.

Using the same idea and arguments as in \cite{ueno},
we gave inequalities on the attraction rates  in \cite{u4},
which imply that 
the asymptotic attraction rate is $\delta$ if $\gamma > 0$ or $\delta \leq d$,
and is $d$ if $\gamma = 0$ and $\delta > d$.

\subsection{New results: weighted plurisubharmonic functions}

To consider the global dynamics,
we assume that $f$ is defined on $\mathbb{C}^2$ or
on $\{ |z| < R \} \times \mathbb{C}$ for large $R$
such that $A_p$ is relatively compact in $\{ |z| < R \}$.
We give a summary of main results for Cases 2, 3 and 4, respectively,
and Tables \ref{table for case 2},
\ref{table for case 3} and \ref{table for case 4}
as comparison charts of the results for Cases 2, 3 and 4, respectively,
in the next section.
Although those results are important,
they are so much and require to introduce the interval or rectangle $\mathcal{I}_f$
instead of $l_1$ and $l_2$.
Therefore,
we only give three consequences of the main results in this subsection,
which are separated into the case $\delta < T_{s-1}$ and the case $\delta \geq T_{s-1}$.

Let us first state two consequences of the main results for the case $\delta < T_{s-1}$. 
Let $A_f$ be the union of all the preimages of $U$ and $A_0$ the attracting basin.
If $\delta < T_{s-1}$, 
then $A_f$ almost coincides with $A_0$.
Let $E_z$ and $E_w$ be the unions of all the preimages of $\{ z = 0 \}$ and $\{ w = 0 \}$,
respectively,
and $E = E_z \cup E_w$. 

\begin{theorem} \label{main thm on A_f when delta < T}
If $\delta < T_{s-1}$,
then $A_f^{} = A_0 - E_z$.
\end{theorem}

To describe the vertical dynamics,
we introduce the following limits in Case 2: 
\begin{align*}
G_z^{\alpha} (w) 
&= \lim_{n \to \infty} \dfrac{1}{d^n} \log \left| \dfrac{w_n}{z_n^{\alpha}} \right|
\text{ if } \delta \neq d,
\text{ and} \\
G_z^{\infty} (w) 
&= \lim_{n \to \infty} \dfrac{1}{d^n} \log \dfrac{|w_n|}{|z_n|^{\frac{\gamma}{d} n}}
\text{ if } \delta = d,
\end{align*}
where $\alpha = \gamma/(\delta - d)$
and $(z_n, w_n) = f^n (z,w)$.
Note that
\[
G_z^{\alpha} (w) = \log \left| \dfrac{\phi_2 (z,w)}{\phi_1 (z)^{\alpha}} \right| \text{ and }
G_z^{\infty} (w) = \log | \phi_2 (z,w) | 
\]
if the B\"{o}ttcher coordinate $\phi$ exists,
where $\phi = (\phi_1, \phi_2)$,
and that $G_z^{\alpha} \circ f = d G_z^{\alpha}$ and
$G_z^{\infty} \circ f = d^n G_z^{\infty} + n \gamma d^{n-1} G_p$.
Hence the existence of $\phi$ on $U$ induces the existence of the limits above on $A_f$,
whose properties depend on the magnitude relation between $\delta$ and $d$.

\begin{theorem}\label{main thm on G_z^a when delta < T}
Let $\delta < T_{s-1}$.
If $\delta < d$, then
$G_z^{\alpha}$ is plurisubharmonic and negative on $A_0$ and
$G_z^{\alpha} (w) = \log \left| z^{- \alpha} w \right| + o(1)$
on $U^{}$ as $r \to 0$.
If $\delta = d$, then
$G_z^{\infty}$ is plurisubharmonic on $A_0 - E_z$ and
$G_z^{\infty} (w)  = \log \left| w \right| + o(1)$
on $U^{}$ as $r \to 0$.
If $\delta > d \geq 1$, then
$G_z^{\alpha}$ is plurisubharmonic on $A_0 - E_z$ and
$G_z^{\alpha} (w) = \log \left| w/z^{\alpha} \right| + o(1)$
on $U^{}$ as $r \to 0$.
For all the cases,
the functions are pluriharmonic on $A_0 - E$. 
\end{theorem}

Detailed versions of the two theorems above are given as 
Theorems \ref{main thm on attr sets for Case 2} and 
\ref{main thm on G_z^a for Case 2},
respectively, in the next section.
We remark that the notation $\alpha_0$ is used instead of $\alpha$
in \cite{ueno}.

Let us next state a consequence of the main results for the case $\delta \geq T_{s-1}$. 
Note that $d \geq 1$ for this case.
The weight $\alpha$ is not defined 
for the special case $\gamma = 0$ and $\delta = d$,
whereas $\delta > d$ and so $\alpha \geq 0$ 
except the special case if $\delta \geq T_{s-1}$.
By redefining $\alpha$ for the case $\delta \geq T_{s-1}$,
which coincides with $\gamma/(\delta - d)$ except the special case,
we can always define the limit $G_z^{\alpha}$ for the case $\delta \geq T_{s-1}$,
which exists on $A_f - E_z$ if $\phi$ exists on $U$.
Moreover,
the limit
\[
G_z^{\alpha, +} (w) 
= \lim_{n \to \infty} \dfrac{1}{d^n} \log^{+} \left| \dfrac{w_n}{z_n^{\alpha}} \right|
\]
exists and
behaves well on the complement of specific fibers in the basin.
Let $E_{deg}$ be the union of all the preimages of 
the set of fibers in $A_p \times \mathbb{C}$ 
that are degenerated to points by $f$.
In the following theorem
we set $(\gamma, d)$ as $(n_k, m_k)$ 
if $f$ has two dominant terms $(n_k, m_k)$ and $(n_{k+1}, m_{k+1})$;
we remark that $\alpha$ depends only on $f$. 

\begin{theorem} \label{main thm on G_z^a,+ when delta >= T}
Let $\delta \geq T_{s-1}$ and
let $d \geq 2$ or $\delta \neq T_k$ for any $k$.
If $\gamma > 0$ or $\delta = d$,
then $\alpha > 0$, and 
$G_z^{\alpha, +}$ is well-defined on $A_0 - E_z$,
plurisubharmonic and continuous on $A_0 - E_z \cup E_{deg}$,
pluriharmonic outside the boundary of 
$\{ G_z^{\alpha, +} > 0 \} \cap (A_0 - E_z \cup E_{deg})$ and
$G_z^{\alpha, +} (w) = \log \left| w/z^{\alpha} \right| + o(1)$ 
on $U - \{ z = 0 \}$ as $r \to 0$.
If $\gamma = 0$ and $\delta > d$,
then $\alpha = 0$, and $G_z^{\alpha, +}$ is well-defined on $A_0$,
plurisubharmonic and continuous on $A_0 - E_{deg}$,
pluriharmonic outside the boundary of 
$\{ G_z^{\alpha, +} > 0 \} \cap (A_0 - E_{deg})$ and
$G_z^{\alpha, +} (w) = \log \left| w \right| + o(1)$ 
on $U$ as $r \to 0$.
\end{theorem}

If $\alpha > 0$ and $d \geq 2$,
then the set $\{ G_z^{\alpha, +} > 0 \} \cap (A_0 - E_z)$ is expressed by
the union of all the preimages of another invariant wedge-shaped region,
which is included in $U$ and whose slope relates to $\alpha$.
See Theorems \ref{main thm on G_z^a,+ for case 3} and
\ref{main thm on G_z^a,+ for case 4} for details.

The limits $G_z^{\alpha}$ and $G_z^{\alpha, +}$ are good tools
for the study of the dynamics of $w_n/z_n^{\alpha}$
on $A_f - E_z$ and $A_0 - E_z$, 
respectively.
In particular,
the set where $G_z^{\alpha, +}$ is not pluriharmonic
seems to be a Julia set of the dynamics of $w_n/z_n^{\alpha}$
and we expect that this Julia set is fractal when $\delta = T_k$ for some $k$,
which is true at least for Examples
\ref{example of deg type} and \ref{example of non-deg type}.


\subsection{New results: typical plurisubharmonic functions}

There are other limits such as
\[
G_z (w) 
= \lim_{n \to \infty} \dfrac{1}{\lambda^n} \log \left| w_n \right|
\text{ and }
G_f (z,w) 
= \lim_{n \to \infty} \dfrac{1}{\lambda^n} \log \left| (z_n, w_n) \right|,
\]
where $\lambda = \max \{ \delta, d \}$ and
$|(z,w)| = \max \{ |z|, |w| \}$.
We remark that $\lambda$ depends only on $f$. 
Although these limits are typical 
to study the dynamics of a holomorphic map $f$,
they are less useful than the previous limits 
to 
discover the weighted vertical dynamics; 
roughly speaking,
if $\delta \geq d$, then
either these limits are constant on the intersection of 
each fiber and $A_f$, or
$\gamma = 0$ and so $f_0$ is just a product.
Tables \ref{table for typical psh funs} and \ref{table for G_z} 
in the next section
help us to understand the statements below. 

First,
the existence of the B\"{o}ttcher coordinate induces the following.

\begin{theorem} \label{main thm on G_z and G_f}
If $\delta < d$,
then $f$ is in Case 2,
$G_z = G_z^{\alpha}$ on $A_0$ and
$G_f = 0$ on $A_0 - E_z$.
If $\gamma > 0$ and $\delta = d$,
then $f$ is in Case 2,
$G_z = - \infty$ and $G_f = G_p$ on $A_0$.
If $\gamma = 0$ and $\delta = d$,
then $f$ is in Case 3,
$G_z$ is pluriharmonic on $A_f$ and 
$G_z (w) = \log \left| w \right| + o(1)$
on $U^{}$ as $r \to 0$.
If $\delta > d \geq 2$ or 
if $\delta > d = 1$ and $\delta \neq T_k$ for any $k$,
then $G_z = \alpha G_p$ on $A_0 - E_w$, $A_f - E_w$ or $A_f$. 
More precisely,
$G_z = \alpha G_p$ on $A_0 - E_w$ if $\delta < T_{s-1}$ and $d \geq 1$,
on $A_f - E_w$ if $\delta = T_{s-1}$ and $d \geq 2$, and
on $A_f$ if $\delta > T_{s-1}$, and 
if $d \geq 2$ or if $d = 1$ and $\delta \neq T_k$.
\end{theorem}

Although $G_z$ is useful for the special case $\gamma = 0$ and $\delta = d$,
the monomial map $f_0$ is just a product and
the dynamics is rather easy.
We remark that
if $\gamma = 0$ and $\delta > d$,
then $f$ is in Case 3 and $G_z = 0$ on $A_f$.

Even if $d = 0$, 
we can prove that $G_z = \alpha G_p$ on $A_0$
unless $\delta = T_{s-1}$,
although the B\"{o}ttcher coordinate does not exist on $U$.

\begin{theorem}  \label{main thm on G_z and G_f when d = 0}
If $T_{s-1} > \delta > d = 0$,
then $f$ is in Case 2 and $G_z = \alpha G_p$ on $A_0$.
\end{theorem}

If $\delta = T_{s-1}$, then
the dominant term $(m_{s-1}, n_{s-1})$ belongs to Case 3 or Case 4, and
the other dominant term $(m_{s}, n_{s})$ belongs to Case 2.
Since $\delta > m_{s-1}$ and $m_s \geq 2$,
the equality $G_z = \alpha G_p$ holds
on the two disjoint invariant regions
as stated in Theorem \ref{main thm on G_z and G_f}.
Hence it follows from
Corollary \ref{main cor on attr sets for Case 4} that
$G_z = \alpha G_p$ on $A_0 - E_w$.

\begin{corollary} \label{cor on G_z in main results}
If $T_{s-1} = \delta > m_{s-1} > m_s \geq 2$,
then $G_z = \alpha G_p$ on $A_0 - E_w$.
\end{corollary}

The limit
\[
G_f^{\alpha} (z,w) 
= \lim_{n \to \infty} \dfrac{1}{\lambda^n} \log \left| (z_n^{\alpha}, w_n) \right|
\]
is useful to hide the set where we do not know whether the limit $G_z$ exists,
and it summarizes the statements for some cases in the theorems above as follows.

\begin{theorem} \label{main thm on G_f^a}
If $\gamma = 0$ and $\delta = d$,
then $G_f^{\alpha}$ is plurisubharmonic on $A_0$,
$G_f^{\alpha} > \alpha G_p$ on $A_f$ and
$G_f^{\alpha} = \alpha G_p$ on $A_0 - A_f$.
If $\delta > d$,
then $G_f^{\alpha} = \alpha G_p$ on $A_0$
unless $d \leq 1$ and $\delta = T_k$ for some $k$.
\end{theorem}

We remark that
if $\gamma = 0$ and $\delta > d$,
then $G_f^{\alpha} = 0$ on $A_0$
by defining that $z^0 = 1$.
The three theorems above are consequences of 
Theorems \ref{main thm on G_z, G_f and G_f^a for Case 2},
\ref{main thm on G_z and G_f^a for case 3} and
\ref{main thm on G_z and G_f^a for Case 4}.

\begin{corollary}
At least one of the limits $G_z$, $G_f$ and $G_f^{\alpha}$ is
plurisubharmonic on $A_0$
for any case unless $d \leq 1$ and $\delta = T_k$ for some $k$.
\end{corollary}

The plurisubharmonic function $G_{FJ}$ constructed in \cite{fj}
has to satisfy the equation $G_{FJ} \circ f = c_{\infty} G_{FJ}$,
where $c_{\infty}$ is the asymptotic attraction rate for $f$.
If $\gamma > 0$ or $\delta \leq d$,
then $G_p$ is a candidate of $G_{FJ}$ since $c_{\infty} = \delta$
and, moreover, 
$G_f^{\alpha}$ is another candidate when $\alpha > 0$.
However,
$G_z^{\alpha}$, $G_z^{\infty}$ and $G_z^{\alpha, +}$ do not satisfy the equation above.
On the other hand,
if $\gamma = 0$ and $\delta > d$,
then $G_z^{\alpha}$ is a candidate of $G_{FJ}$ since $c_{\infty} = d$,
whereas $G_p$ is not.
However,
$\alpha = 0$ and $f_0$ is just a product for this case.

\subsection{Organization and comments}

We give main results for each case in Section 2.
To state our results more precisely, 
we introduce the interval or rectangle $\mathcal{I}_f$: 
our previous results,
Lemma \ref{main lemma for previous reselts} and
Theorem \ref{main thm for previous reselts}, hold
not only for $l_1$ and $l_2$ but also for
any real number in the interval or numbers in the rectangle,
which numbers are called weights in the previous papers. 
In this sense, 
there are infinitely many invariant wedges
if $\delta \neq T_k$ for any $k$.
For Cases 2 and 4,
we find out new invariant wedges, 
on which $b_{\gamma d} z^{\gamma} w^d$ may not be the dominant term in $q$, and
claim that the unions of all the preimages of the wedges 
coincide with $A_0 - E_z$ for almost all the cases.
These results imply 
Theorems \ref{main thm on A_f when delta < T} and
are important to induce
Theorems \ref{main thm on G_z^a,+ when delta >= T} and 
\ref{main thm on G_z and G_f when d = 0} and
Corollary \ref{cor on G_z in main results}.
Tables are given in each subsection 
as comparison charts of the main results.

We review our previous study in \cite{ueno} more precisely
in Section 3.
We give the definition of $\mathcal{I}_f$, calculate it, and
illustrate detailed versions of 
Lemma \ref{main lemma for previous reselts} and
Theorem \ref{main thm for previous reselts}
in terms of blow-ups.

The proofs of the results for Case 2 are given in Sections 4 and 5.
With an illustration in terms of blow-ups,
we first prove the results on invariant wedges and the unions of 
all the preimages of the wedges in Section 4.
We next show the existence and properties of 
plurisubharmonic functions, and 
investigate the asymptotic behavior of the functions 
toward the boundaries of the unions in Section 5,
which makes clear the magnitude relation of the unions.
The proofs of the results for Cases 3 and 4 are given in Sections 6 and 7,
respectively.
Moreover,
we provide examples of polynomial skew products that are semiconjugate
to polynomial products in Section 6,
which demonstrate that $G_z^{\alpha}$ and $G_z^{\alpha, +}$
are good tools for the study of 
the weighted vertical dynamics.

Similar studies have been made for polynomial skew products in 
\cite{u-weight1}, \cite{u-weight2} and \cite{u1}
to understand the dynamics near infinity.
Although the idea and the statements are similar,
and we apply the same arguments to investigate
the asymptotic behavior of plurisubharmonic functions,
the study in this paper is more valuable 
in the sense that we only dealt with the case when 
the specific term in $q$ is dominant on a wedge-shaped region,
and with the specific weight
in the previous papers above.

\section{Summary of main results for each case}

Main results for each case are given in this section.
We deal with Cases 2, 3 and 4 in Sections 2.1, 2.2 and 2.3,
respectively.
Moreover,
we provide tables for each case,
which help us to recognize the differences of
properties of plurisubharmonic functions
depending on the magnitude relation of $\delta$ and $d$,
of $\delta$ and $T_j$ and of $\gamma$ and $0$.
Furthermore,
we provide tables of properties of 
the typical  plurisubharmonic functions
in Section 2.4.

\subsection{Main results for Case 2}

Let $s > 1$, $\delta \leq T_{s-1}$ and $(\gamma, d) = (n_s, m_s)$ in this subsection.
Then $\gamma > 0$.
Table \ref{table for case 2} is given below as a comparison chart of main results for Case 2.
We remark that
the descriptions on the asymptotic behavior of $G_z^{\alpha}$ and $G_z^{\infty}$
in the table are not precise.
Actually,
we have to remove $E_z$ and $E_{deg}$, the sets of fibers, from
the boundary $\partial A_f^l$ of $A_f^l$;
see Propositions 
\ref{main prop on asy behavior when delta < d},
\ref{main prop on asy behavior when delta = d},
\ref{main prop on asy behavior when T > delta > d} and
\ref{main prop on asy behavior when delta = T foe case 2}
in Section 5 for complete descriptions.

\begin{table}[htb]
\caption{Comparison chart of main results for Case 2} \label{table for case 2}
\begin{tabular}{|c|c|c|c|c|} \hline
                       & $\mathcal{I}_f$  & $G_z^{\alpha}$ or $G_z^{\infty}$                  & limsup                    & $G_z$                  \\ \hline
$T > d > \delta$  & $[ l_1, \infty )$  & $G_z^{\alpha}$ psh on $A_0$                     & $G_z^{\alpha} \to 0$   &   $G_z= G_z^{\alpha}$   \\ 
                       & $\alpha < 0$    & ph on $A_0 - E$                          & as $\to \partial A_f^l$    & on $A_0$     \\ 
                       &                     &  $\fallingdotseq \log |z^{- \alpha} w|$ on $U$ & for any $l$ in $\mathcal{I}_f$ & \\ \hline
$T > \delta = d$  & $[ l_1, \infty )$ & $G_z^{\infty}$ psh on $A_0 - E_z$                & $G_z^{\infty} \to \infty$ & $G_z=- \infty$  \\ 
                       &                     & ph on $A_0 - E$                          & as $\to \partial A_f^l$   &  on $A_0$  \\   
                       &                      & $\fallingdotseq \log |w|$ on $U$                & for any $l$ in $\mathcal{I}_f$ &  \\ \hline
$T > \delta > d$  & $[ l_1, \alpha]$  & $G_z^{\alpha}$ psh on $A_0 - E_z$             & $G_z^{\alpha} \to \infty$ if $l < \alpha$ &  $G_z = \alpha G_p$  \\ 
$d \geq 1$         & $\alpha > 0$    & ph on $A_0 - E$  & $G_z^{\alpha} \to 0$ or   & on $A_0 - E_w$    \\ 
                       &                     & $\fallingdotseq \log \left| w/z^{\alpha} \right|$ on $U$  & bdd if $l = \alpha$  &   \\ \hline
$T > \delta > d$  & $[ l_1, \alpha]$  & $\nexists \phi$                                    &  & $G_z = \alpha G_p$    \\  
$d = 0$             & $\alpha > 0$     &                                                           & & on $A_0$     \\  \hline
$T = \delta > d$  & $\{ l_1 \}$        & $G_z^{\alpha}$ psh on $A_f$                      & $G_z^{\alpha} \to 0$                          & $G_z= \alpha G_p$    \\ 
$d \geq 2$         & $\alpha = l_1$   &  ph on $A_f - E$                         & as $\to \partial A_f^{}$ & on $A_f - E_w$  \\  \hline
\end{tabular}
\end{table}

We first exhibit a result on invariant wedges, and
give an affirmative answer to the question
whether the unions of all the preimages of the wedges 
coincide with the basin.
Let $U^l = \{ |z| < r, |w| < r |z|^{l} \}$
and $U_{r_1, r_2}^l = \{ |z| < r_1, |w| < r_2 |z|^{l} \}$.
Note that $U = U^{l_1}$.
The following is a detailed version of 
Lemma \ref{main lemma for previous reselts},
which is stated also in the next section as 
Lemma \ref{detailed lemma for case 2}.

\begin{lemma}[\cite{ueno}] \label{}
The following holds for any $l$ in $\mathcal{I}_f$ if $d \geq 2$
and for any $l$ in $\mathcal{I}_f- \{ \alpha \}$ if $d = 1$ and $\delta < T_{s-1}$,
where
$\mathcal{I}_f = [ l_1, \infty )$ if $\delta \leq d$ and
$\mathcal{I}_f = [ l_1, \alpha ]$ if $\delta > d$:
$f \sim f_0$ on $U^l$ as $r \to 0$, and
$f(U^l) \subset U^l$ for small $r >0$.
\end{lemma}

The interval $\mathcal{I}_f$ plays an important role
to study the dynamics of $f$
as explained in the next section.
Note that $\min \mathcal{I}_f = l_1$ and
$U^{l_1}$ is the largest wedge among $U^{l}$ for any $l$ in $\mathcal{I}_f$.
We have succeeded to extend the second statement of 
this lemma to the following.

\begin{theorem}\label{main thm on inv wedges for Case 2}
The following holds for any $l > 0$ if $\delta \leq d$ and 
for any $0< l < \alpha$ if $\delta > d$:
$f(U_{r_1, r_2}^l) \subset U_{r_1, r_2}^l$
for suitable small $r_1$ and $r_2$.
\end{theorem}

Whereas we used the property $f \sim f_0$ on $U^l$ as $r \to 0$
to show that $f(U^l) \subset U^l$ in the lemma above, 
we apply different arguments to prove this theorem;
the property above does not hold 
since $l$ does not belong to $\mathcal{I}_f$ if $0< l < l_1$.
Note that $U_{r_1, r_2}^l$ is larger than $U_{r_1, r_2}^{l_1}$ for any $0 < l < l_1$
and that there is no condition on the degree $d$ in the theorem.

Let $A_f^{l}$ 
be the union of all the preimages of $U^l$.
Note that $A_f = A_f^{l_1}$.
We have also succeeded to give the following affirmative answer to the question,
which implies Theorem
\ref{main thm on A_f when delta < T}
since $l_1 < \alpha$ if $\delta  < T_{s-1}$.

\begin{theorem} \label{main thm on attr sets for Case 2}
It follows that $A_f^l = A_0 - E_z$ for any $l > 0$  if $\delta \leq d$
and for any $0 < l < \alpha$ if $\delta > d$
and $(n_1, m_1) \neq (0, \delta)$.
\end{theorem}


We next exhibit main results on plurisubharmonic functions.
If the B\"{o}ttcher coordinate exists on $U$, 
then one can show that the limits $G_z^{\alpha}$ and $G_z$ exist 
and are plurisubharmonic on $A_f - E_z$ and $A_f$,
respectively.
Moreover,
thanks to Theorem \ref{main thm on A_f when delta < T},
we obtain the following theorems. 

\begin{theorem} \label{main thm on G_z^a for Case 2}
Let $\delta < T_{s-1}$.
If $\delta < d$, then
$G_z^{\alpha} = G_z$ on $A_0$,
which is plurisubharmonic and negative on $A_0$ and
$G_z^{\alpha} (w) = \log \left| z^{- \alpha} w \right| + o(1)$
on $U^{}$ as $r \to 0$.
If $\delta = d$, then
$G_z^{\infty}$ is plurisubharmonic on $A_0 - E_z$ and
$G_z^{\infty} (w) = \log \left| w \right| + o(1)$
on $U^{}$ as $r \to 0$.
If $\delta > d \geq 1$, then
$G_z^{\alpha}$ is plurisubharmonic on $A_0 - E_z$ and
$G_z^{\alpha} (w) = \log \left| w/z^{\alpha} \right| + o(1)$
on $U^{}$ as $r \to 0$.
For all the cases,
the functions are pluriharmonic on $A_0 - E$. 

On the other hand,
if $T_{s-1} = \delta > d \geq 2$, then
$G_z^{\alpha}$ is plurisubharmonic on $A_f$,
pluriharmonic on $A_f - E_w$ and 
$G_z^{\alpha} (w) = \log \left| w/z^{\alpha} \right| + o(1)$
on $U^{}$ as $r \to 0$.
\end{theorem}

\begin{theorem} \label{main thm on G_z, G_f and G_f^a for Case 2}
Let $\delta < T_{s-1}$.
If $\delta < d$, then $G_z = G_z^{\alpha}$ and $G_f = 0$ on $A_0$.
If $\delta = d$, then $G_z = - \infty$ and $G_f = G_p$ on $A_0$.
If $\delta > d \geq 1$, then
$G_z = \alpha G_p$ on $A_0 - E_w$
and $- \infty$ on $E_w$.
If $d = 0$, then
$G_z = \alpha G_p$ on $A_0$.
In particular,
if $\delta > d$, 
then $G_f^{\alpha} = \alpha G_p$ on $A_0$.

On the other hand,
if $T_{s-1} = \delta > d \geq 2$, then
$G_z = \alpha G_p$ on $A_f - E_w$
and $- \infty$ on $E_w$.
Hence $G_f^{\alpha} = \alpha G_p$ on $A_f$.
Moreover,
if $n_{s-1} > 0$,
then $G_z = \alpha G_p$ on $A_0 - E_w$
and $- \infty$ on $E_w$.
Hence $G_f^{\alpha} = \alpha G_p$ on $A_0$.
\end{theorem}

By investigating the asymptotic behavior of $G_z^{\alpha}$ 
toward the boundary $\partial A_f^l$,
we can show that
$A_f^{\alpha}$ is smaller than $A_f^{l_1}$.

\begin{proposition}\label{main propo on A_f^l for case 2}
If $T_{s-1} > \delta > d \geq 1$, then
$A_f^{\alpha} \subsetneq A_f^{l_1}$.
\end{proposition}

Moreover, 
we give a sufficient condition for $A_f^{\alpha}$ and $A_f^{l_1}$
being unbounded.

\begin{proposition}\label{main prop on the unboundness for case 2}
Let $T_{s-1} > \delta > d \geq 1$.
If $q$ is polynomial and 
$f^{-1} \{ w=0 \} \supsetneq \{ w=0 \}$, 
then both $A_f^{\alpha}$ and $A_f^{l_1}$ are unbounded.
\end{proposition}

\subsection{Main results for Case 3}

Let $s > 1$, $\delta \geq T_{1}$ and $(\gamma, d) = (n_1, m_1)$
in this subsection.
Then $\delta \geq T_{1} \geq d \geq 1$, and $\delta > d$ if $\gamma > 0$.
Table \ref{table for case 3} is given below as a comparison chart of main results for Case 3;
see Proposition
\ref{Case 3: asy of G_z^alpha on boundary^in}
for a complete description on the asymptotic behavior of $G_z^{\alpha}$
toward $\partial A_f^l \cap A_0$.

\begin{table}[htb]
\caption{Comparison chart of main results for Case 3} \label{table for case 3}
{\small
\begin{tabular}{|c|c|c|c|c|} \hline
 & $\gamma > 0$ \& $\delta > T_1$ & $\gamma > 0$ \& $\delta = T_1$ & $\gamma = 0$ \& $\delta > T_1$ & $\gamma = 0$ \& $\delta = T_1$ \\ \hline
$\mathcal{I}_f$ & $[\alpha, l_2]$ & $\{ \alpha \} = \{ l_2 \}$ & $(0, l_2]$ & $(0, l_2]$ \\ 
$\alpha$        & $\alpha > 0$ & $\alpha > 0$ & $\alpha = 0$ & $\alpha = l_2 > 0$ \\ \hline
$d$               & $\delta > d \geq 1$ & $\delta > d \geq 1$ & $\delta > d \geq 2$ & $\delta = d  \geq 2$ \\ 
$f_0$             & $(z^{\delta}, z^{\gamma} w^d)$ & $(z^{\delta}, z^{\gamma} w^d)$ & $(z^{\delta}, w^d)$ & $(z^{\delta}, w^{\delta})$ \\ \hline
$G_z^{\alpha}$  & ph on $A_f - E_z$ & ph on $A_f - E_z$ & ph on $A_f$ & ph on $A_f - E_z$ \\ 
 & $\fallingdotseq \log |w/z^{\alpha}|$ on $U$ & $\fallingdotseq \log |w/z^{\alpha}|$ on $U$ & $\fallingdotseq \log |w|$ on $U$ & $\fallingdotseq \log |w/z^{\alpha}|$ on $U$ \\ 
 &  & if $d \geq 2$ &  &  \\ \hline
lim & $G_z^{\alpha} \to 0$ or & $G_z^{\alpha} \to 0$ & $G_z^{\alpha} \to - \infty$ & $G_z^{\alpha} \to (l - \alpha) G_p$ \\ 
 & bdd if $l = \alpha$ & as $\to \partial A_f^{\alpha} \cap A_0$ & as $\to \partial A_f^l \cap A_0$ & for any $l$ in $\mathcal{I}_f$ \\ 
 & $G_z^{\alpha} \to - \infty$ &  & for any $l $ in $\mathcal{I}_f$ & In particular, \\ 
 & for any $l > \alpha$ &  &  & $G_z^{\alpha} \to 0$ if $l = \alpha$ \\ \hline
$G_z^{\alpha, +}$ & psh \& conti on & psh \& conti on & psh \& conti on & psh \& conti on \\ 
 & $A_0 - E_z \cup E_{deg}$ & $A_0 - E_z \cup E_{deg}$ & $A_0 - E_{deg}$ & $A_0 - E_z \cup E_{deg}$ \\ 
 & $G_z^{\alpha, +} > 0$ on $A_f^{\alpha}$ & $G_z^{\alpha, +} > 0$ on $A_f^{\alpha}$ &  & $G_z^{\alpha, +} > 0$ on $A_f^{\alpha}$ \\ 
 & $= 0$ on $A_0 - A_f^{\alpha}$ & $= 0$ on $A_0 - A_f^{\alpha}$ &  & $= 0$ on $A_0 - A_f^{\alpha}$ \\
 &  if $d \geq 2$ & &  &  \\ \hline
$A_f^l$ & $A_f^{\alpha} \subsetneq A_f^l = A_f^{l_2}$ &  & $A_f^l = A_f^{l_2}$ & $A_f^{l'} \subsetneq A_f^{l''}$ \\ 
 & for any $l > \alpha$ &  & for any $l$ in $\mathcal{I}_f$ & for any $l' < l''$ \\ \hline
$G_z$ & $G_z = \alpha G_p$ on $A_f$ & $G_z = \alpha G_p$ on $A_f$ & $G_z = 0$ on $A_f$ & $\fallingdotseq \log |w|$ on $U$ \\ 
 &  &  &  & ph on $A_f$ \\ \hline
$G_f^{\alpha}$ & $G_f^{\alpha} = \alpha G_p$ on $A_0$ & $G_f^{\alpha} = \alpha G_p$ on $A_0$ & $G_f^{\alpha} = 0$ on $A_0$ & psh on $A_0 - E_{deg}$ \\ 
 & (psh on $A_0$) & (psh on $A_0$) & (psh on $A_0$) & $G_f^{\alpha} > \alpha G_p$ on $A_f^{\alpha}$ \\
 &  &  &  & $= \alpha G_p$ on $A_0 - A_f^{\alpha}$ \\ \hline
\end{tabular}
}
\end{table}

Although we want to define $\alpha$ as $\gamma/(\delta - d)$,
it is not defined if $\gamma = 0$ and $\delta = d$,
which corresponds to the exceptional case $(n_1, m_1) = (0, \delta)$
in Theorem \ref{main thm on attr sets for Case 2}.
To overcome this problem,
we redefine  $\alpha$ as the minimum of $l \geq 0$ such that
a line $x + ly - l \delta = 0$ intersects $N(q)$;
this definition works not only when $\delta \geq T_{1}$
but also when $\delta > m_s$.
Then the line that takes the minimum passes through
$(0, \delta)$ and $(\gamma, d)$, and
$\alpha$ coincides with $\gamma/(\delta - d)$
unless $\gamma = 0$ and $\delta = d$.
Moreover,
$\alpha$ is well-defined even if $\gamma = 0$ and $\delta = d$,
which coincides with $l_2 = n_2/(\delta - m_2)$.

We exhibit results only on plurisubharmonic functions;
we have no results on invariant wedges and
do not know whether the unions of all the preimages of the wedges 
coincide with the basin.

The existence of the B\"{o}ttcher coordinate induces the following. 

\begin{proposition} \label{main result for case 3: G_z^a}
Let $d \geq 2$ or $\delta > T_{1}$.  
If $\gamma > 0$ or $\delta = d$,
then $\alpha > 0$ and $G_z^{\alpha}$ is pluriharmonic on $A_f - E_z$. 
Moreover,
$G_z^{\alpha} (w) = \log \left| w/z^{\alpha} \right| + o(1)$
on $U - \{ z = 0 \}$ as $r \to 0$.
If $\gamma = 0$ and $\delta > d$,
then $\alpha = 0$ and $G_z^{\alpha}$ is pluriharmonic on $A_f$. 
Moreover,
$G_z^{\alpha} (w) = \log \left| w \right| + o(1)$
on $U$ as $r \to 0$.
\end{proposition}

Let $U^l = \{ |z|^{l} < r^{l} |w|, |w| < r \}$ and $A_f^l$ be 
the union of all the preimages of $U^l$.
Note that $U = U^{l_2}$ and $A_f = A_f^{l_2}$.
Although the interval $\mathcal{I}_f$ plays an important role
as the same as Case 2,
where $\mathcal{I}_f = [\alpha, l_2]$ if $\gamma > 0$ and
$\mathcal{I}_f = (0, l_2]$ if $\gamma = 0$,
we avoid to use this notation in this subsection.
By investigating the asymptotic behavior of $G_z^{\alpha}$ and $G_z$
toward the boundary $\partial A_f^l$,
we obtain the theorems and proposition below.

\begin{theorem} \label{main thm on G_z^a,+ for case 3} 
Let $d \geq 2$ or $\delta > T_{1}$.  
If $\gamma > 0$ or $\delta = d$, then 
$G_z^{\alpha, +}$ is well-defined on $A_0 - E_z$,
plurisubharmonic and continuous 
on $A_0 - E_z \cup E_{deg}$ and
pluriharmonic outside the boundary of 
$\{ G_z^{\alpha, +} > 0 \} \cap (A_0 - E_z \cup E_{deg})$.
More precisely,
if $d \geq 2$, then
$G_z^{\alpha, +} > 0$ on $A_f^{\alpha}$ and 
$G_z^{\alpha, +} = 0$ on $A_0 - A_f^{\alpha}$,
which is pluriharmonic on $A_0 - \partial A_f^{\alpha} \cup E_z \cup E_{deg}$.
If $\gamma = 0$ and $\delta > d$, then 
$G_z^{\alpha, +}$ is  well-defined on $A_0$,
plurisubharmonic and continuous on $A_0 - E_{deg}$ and
pluriharmonic outside the boundary of 
$\{ G_z^{\alpha, +} > 0 \} \cap (A_0 - E_{deg})$. 
\end{theorem}

\begin{theorem} \label{main thm on G_z and G_f^a for case 3} 
If $\delta > d \geq 2$ or if $\delta > d = 1$ and $\delta > T_{1}$,
then $G_z = \alpha G_p$ on $A_f$ and $G_f^{\alpha} = \alpha G_p$ on $A_0$.
If $\gamma = 0$ and $\delta = d$, then 
$G_z$ is pluriharmonic and continuous on $A_f$,
$G_z (w) = \log \left| w \right| + o(1)$ on $U$ as $r \to 0$, 
$G_f^{\alpha}$ is plurisubharmonic and continuous on $A_0 -  E_{deg}$,
$G_f^{\alpha} > \alpha G_p$ on $A_f$ and $G_f^{\alpha} = \alpha G_p$ on $A_0 - A_f$.
\end{theorem}

We remark that
if $\gamma = 0$ and $\delta > d$, 
then $G_z = 0$ on $A_f$ and $G_f^{\alpha} = 0$ on $A_0$
and that the equality $G_f^{\alpha} = \alpha G_p$ on $A_0$ holds 
even if $d = 1$ and $\delta = T_1$ 
as stated in Remark 
\ref{remark on G_f^a for case 3}.

\begin{proposition} \label{main result for case 3: A_f^l}
If $\delta > T_1$, then $A_f^{l} = A_f^{l_2}$ 
for any $\alpha < l \leq l_2$. 
Moreover,
if $\gamma > 0$, then 
$A_f^{\alpha} \subsetneq A_f^{l_2}$.
If $\gamma = 0$ and $\delta = d$, 
then $\delta = T_1$ and 
$A_f^{l'} \subsetneq A_f^{l''}$
for any $0 < l' < l'' \leq l_2$. 
\end{proposition}

\subsection{Main results for Case 4}

Let $s > 2$, $T_k \leq \delta \leq T_{k-1}$ and $(\gamma, d) = (n_k, m_k)$
for some $2 \leq k \leq s-1$ in this subsection.
Then $\gamma > 0$ and $\delta > d \geq 1$.
Table \ref{table for case 4} is given below as a comparison chart of main results for Case 4;
see Propositions
\ref{asy behavior when T_k < delta  T_k-1},
\ref{asy behavior when delta = T_k} and
\ref{asy behavior when delta = T_k-1}
for complete descriptions on the asymptotic behavior of $G_z^{\alpha}$.
In the table
the notation $A$ means the attracting set corresponding to
the other dominant term and
we denote $A_f^{l_1, \alpha - l_1}$ as $A_f^{\cdots}$ for short.

\begin{table}[htb]
\caption{Comparison chart of main results for Case 4} \label{table for case 4}
{\small 
\begin{tabular}{|c|c|c|c|c|} \hline
 & $T_k < \delta < T_{k-1}$ & $T_k = \delta < T_{k-1}$ & $T_k < \delta = T_{k-1}$ & $T_k < \delta = T_{k-1}$ \\ 
 &  &  & \& $n_{k-1} > 0$ & \& $n_{k-1} = 0$ \\ \hline
$\mathcal{I}_f$ & $[l_1, \alpha] \times [\alpha, l_1+l_2]$ & $[l_1, \alpha) \times \{ \alpha \}$ & $\{ \alpha \} \times (\alpha, l_1+l_2]$ & $\{ \alpha \} \times (\alpha, l_1+l_2]$ \\ 
 & $- \{ (\alpha, \alpha) \}$ &  &  &  \\ \hline
$U_{r_1, r_2}^{l,+}$ & $f(U_{r_1, r_2}^{l,+}) \subset U_{r_1, r_2}^{l,+}$ & $f(U_{r_1, r_2}^{l,+}) \subset U_{r_1, r_2}^{l,+}$ & $f(U_{r_1, r_2}^{l,+}) \subset U_{r_1, r_2}^{l,+}$ & $f(U_{r_1, r_2}^{l,+}) \subset U_{r_1, r_2}^{l,+}$ \\ 
 & for any $0 < l < \alpha$ & for any $0 < l < \alpha$ & for any $0 < l < \alpha$ & for any $0 < l < \alpha$ \\ \hline
$A_f^{l,+}$ & $A_f^{l,+} = A_0 - E_z$ & $A_f^{l,+} = A_0 - E_z$ & $A_f^{l,+} = A_0 - E_z$ &  \\ 
 & for any $0 < l < \alpha$ & for any $0 < l < \alpha$ & for any $0 < l < \alpha$ &  \\ \hline

$\phi$ & $\exists \phi$ on $U$ & $\exists \phi$ on $U$ if $d \geq 2$ & $\exists \phi$ on $U$ if $d \geq 2$ & $\exists \phi$ on $U$ if $d \geq 2$ \\ \hline
$G_z^{\alpha}$ & ph on $A_f$ & ph on $A_f$ & ph on $A_f$ & ph on $A_f$ \\ 
 & $\fallingdotseq \log |w/z^{\alpha}|$ on $U$ & $\fallingdotseq \log |w/z^{\alpha}|$ on $U$ & $\fallingdotseq \log |w/z^{\alpha}|$ on $U$ & $\fallingdotseq \log |w/z^{\alpha}|$ on $U$ \\ \hline

lim & $G_z^{\alpha} \to \infty$ if $l < \alpha$ & $G_z^{\alpha} \to \infty$ & $G_z^{\alpha} \to 0$ & $G_z^{\alpha} \to 0$  \\ 
sup & $G_z^{\alpha} \to 0$ or bdd if & as $\to \partial A_f^{l,+}$ & as $\to \partial A_f^{\alpha,+}$ & as $\to \partial A_f^{\alpha,+}$  \\
 & $l = \alpha$ as $\to \partial A_f^{l,+}$ & for any $l$ in $\mathcal{I}_f^1$ &   &   \\ \hline
lim & $G_z^{\alpha} \to 0$ or bdd if  & $G_z^{\alpha} \to 0$ as  & $G_z^{\alpha} \to - \infty$ as  & $G_z^{\alpha} \to - \infty$ as \\
 & $l = \alpha$, $G_z^{\alpha} \to - \infty$  & $\to \partial A_f^{l_1, \alpha -l_1} \cap A_0$  & $\to \partial A_f^{\alpha,l- \alpha} \cap A_0$  & $\to \partial A_f^{\alpha,l- \alpha} \cap A_0$ \\
 & if $l > \alpha$ as $\to$ &   & for any $l$ in $\mathcal{I}_f^2$  &  for any $l$ in $\mathcal{I}_f^2$ \\
 & $\partial A_f^{l_1,l-l_1} - \partial A_f^{l_1,+}$ &  &   & \\ \hline

$A_f^{ l_{  \scalebox{0.5}{(1)} }, l_{ \scalebox{0.5}{(2)} }}$ & $A_f^{ l_{  \scalebox{0.5}{(1)} }, l_{ \scalebox{0.5}{(2)} }} = A_f$ for & $A_f^{l, \alpha -l} = A_f^{}$ & $A_f^{\alpha,l- \alpha} = A_f^{}$ & $A_f^{\alpha,l- \alpha} = A_f^{}$ \\ 
 & most $l_{ \scalebox{0.5}{(1)} }$ and $l_{ \scalebox{0.5}{(2)} }$ & for any $l$ in $\mathcal{I}_f^1$ & for any $l$ in $\mathcal{I}_f^2$ & for any $l$ in $\mathcal{I}_f^2$ \\                        
  & $A_f^{l_1, \alpha - l_1} \subsetneq A_f^{}$ &  &  &   \\
  & $A_f^{\alpha, l_1 + l_2 - \alpha} \subsetneq A_f^{}$ &  &  &   \\ \hline
other & $\nexists$ & $(n_{k+1}, m_{k+1})$ & $(n_{k-1}, m_{k-1}) \Rightarrow A$ & $(n_1, m_1) \Rightarrow A$ \\ 
term &  & Case 2 or Case 4 & Case 3 or Case 4 & Case 3 \\ \hline
$G_z^{\alpha, +}$ & psh \& conti on & psh  \& conti on & psh  \& conti on & psh  \& conti on \\ 
 & $A_0 - E_z \cup E_{deg}$ & $A_0 - E_z \cup E_{deg}$ & $A_0 - E_z \cup E_{deg}$ & $A_0 - E_z \cup E_{deg}$ \\
 & $> 0$ on $A_f^{l_1, \alpha - l_1}$ & $> 0$ on $A_f^{l_1, \alpha - l_1}$ & $> 0$ on $A$ & $> 0$ on $A$ \\ 
 & $= 0$ on $A_0 - A_f^{\cdots}$ & $= 0$ on $A_0 - A_f^{\cdots}$ & $= 0$ on $A_0 - A$ & $= 0$ on $A_0 - A$ \\
 & if $d \geq 2$ & & & \\ \hline
$G_z$ & $G_z = \alpha G_p$ on $A_f$ & $G_z = \alpha G_p$ on $A_f$ & $G_z = \alpha G_p$ on $A_f$ & $G_z = \alpha G_p$ on $A_f$ \\ \hline
$G_f^{\alpha}$ & $G_f^{\alpha} = \alpha G_p$ on $A_0$ & $G_f^{\alpha} = \alpha G_p$ on $A_0$ & $G_f^{\alpha} = \alpha G_p$ on $A_0$ & psh on $A_0 - E_{deg}$ \\ 
 &  &  &  & $G_f^{\alpha} > \alpha G_p$ on $A$ \\ 
 &  &  &  & $= \alpha G_p$ on $A_0 - A$ \\ \hline
\end{tabular}
}
\end{table}

We first exhibit a result on invariant wedges, and
give a partial affirmative answer to the question 
whether the unions of all the preimages of the wedges 
coincide with the basin,
which are similar to Case 2.
Let $U^{l_{(1)},l_{(2)}} = \{ |z|^{l_{(1)} + l_{(2)}} < r^{l_{(2)}} |w|, |w| < r|z|^{l_{(1)}} \}$.
Note that $U = U^{l_1, l_2}$.
A detailed version of Lemma \ref{main lemma for previous reselts}
is given in the next section as Lemma \ref{detailed lemma for case 4},
in which the wedge $U^{l_{(1)},l_{(2)}}$ and 
intervals $\mathcal{I}_f^1$ and $\mathcal{I}_f^2$ are used.
Let $U^{l, +} = \{ |z| < r, |w| < r |z|^{l} \}$.
Although it might be more precise 
to define $U^{l, +}$ with the inequality $|z| < r^{1 + 1/l_2}$,
we define it by the inequality $|z| < r$ for simplicity.
Lemma \ref{detailed lemma for case 4} in the next section induces the following.

\begin{lemma} \label{main lemma}
The following holds 
for any $l$ in $\mathcal{I}_f^1$ if $d \geq 2$
and for any $l$ in $\mathcal{I}_f^1 - \{ \alpha \}$ 
if $d = 1$ and $T_k < \delta < T_{k-1}$, 
where
$\mathcal{I}_f^1 = [ l_1, \alpha ]$ if $T_k < \delta < T_{k-1}$,
$\mathcal{I}_f^1 = [ l_1, \alpha )$ if $\delta = T_k$, and
$\mathcal{I}_f^1 = \{ l_1 \} = \{ \alpha \}$ if $\delta = T_{k-1}$:
$f(U^{l, +}) \subset U^{l, +}$ for small $r >0$.
\end{lemma}

Note that $\min \mathcal{I}_f^1 = l_1$ and
$U^{l_1, +}$ is the largest wedge among $U^{l, +}$ for any $l$ in $\mathcal{I}_f^1$.
Let $U_{r_1, r_2}^{l, +} = \{ |z| < r_1, |w| < r_2 |z|^{l} \}$.
We can extend this lemma to the following theorem
by the same arguments as Case 2. 

\begin{theorem}\label{main thm on inv wedges for Case 4}
The following holds for any $0 < l < \alpha$: 
$f(U_{r_1, r_2}^{l, +}) \subset U_{r_1, r_2}^{l, +}$
for suitable small $r_1$ and $r_2$.
\end{theorem}

Note that $U_{r_1, r_2}^l$ is larger than $U_{r_1, r_2}^{l_1}$ for any $0 < l < l_1$
and that there is no condition on the degree $d$ in this theorem.

Let $A_f^{l, +}$ 
be the union of all the preimages of $U^{l, +}$.
We obtain the following partial affirmative answer to the question
as the same as Case 2.

\begin{theorem}\label{main thm on attr sets for Case 4}
If $(n_1, m_1) \neq (0, \delta)$, then
$A_f^{l, +} = A_0 - E_z$ for any $0 < l < \alpha$.
\end{theorem}

This theorem implies the following 
since $l_1 < \alpha$ if $\delta < T_{k-1}$.

\begin{corollary} \label{main cor on attr sets for Case 4}
If $\delta < T_{k-1}$, 
then $A_f^{l_1, +} = A_0 - E_z$.
\end{corollary}

We next exhibit main results on plurisubharmonic functions.
The existence of the B\"{o}ttcher coordinate induces the following.

\begin{proposition}\label{main prop on G_z^a for Case 4}
If $d \geq 2$ or $T_k < \delta < T_{k-1}$,
then $G_z^{\alpha}$ is pluriharmonic on $A_f$ and 
$G_z^{\alpha} (w) = \log \left| w/z^{\alpha} \right| + o(1)$
on $U$ as $r \to 0$.
\end{proposition}

Let $A_f^{l_{(1)}, l_{(2)}}$ 
be the union of all the preimages of $U^{l_{(1)}, l_{(2)}}$.
Thanks to Corollary \ref{main cor on attr sets for Case 4}, 
we obtain the following theorem
by investigating the asymptotic behavior of $G_z^{\alpha}$ 
toward the boundary $\partial A_f^{l_{(1)}, l_{(2)}}$.

\begin{theorem}\label{main thm on G_z^a,+ for case 4} 
Let $\delta < T_{k-1}$. 
If $d \geq 2$ or $\delta > T_k$, 
then $G_z^{\alpha, +}$ is well-defined on $A_0 - E_z$,
plurisubharmonic and continuous on $A_0 - E_z \cup E_{deg}$ and
pluriharmonic outside the boundary of 
$\{ G_z^{\alpha, +} > 0 \} \cap (A_0 - E_z \cup E_{deg})$. 
More precisely,
if $d \geq 2$, then
$G_z^{\alpha, +} > 0$ on $A_f^{l_1, \alpha - l_1}$ and
$G_z^{\alpha, +} = 0$ on $A_0 - A_f^{l_1, \alpha - l_1}$,
which is pluriharmonic on $A_0 - \partial A_f^{l_1, \alpha - l_1} \cup E_z \cup E_{deg}$.
\end{theorem}

The investigation of the asymptotic behavior of $G_z^{\alpha}$ 
toward $\partial A_f^{l_{(1)}, l_{(2)}}$ also shows us 
the magnitude relation of $A_f^{l_{(1)}, l_{(2)}}$
among any $l_{(1)}$ in $\mathcal{I}_f^1$
and $l_{(2)}$ in $\mathcal{I}_f^2$.

\begin{proposition} \label{prop on A_f^l for case 4}
If $T_k < \delta < T_{k-1}$, then
$A_f^{l_{(1)}, l_{(2)}} = A_f^{l_1, l_2}$ 
for any $l_1 \leq l_{(1)} < \alpha$ and 
$\alpha < l_{(1)} + l_{(2)} \leq l_1 + l_2$. 
Moreover,
$A_f^{l_1, \alpha - l_1} \subsetneq A_f^{l_1, l_2}$,
$A_f^{\alpha, l_1 + l_2 - \alpha} \subsetneq A_f^{l_1, l_2}$ and
$A_f^{l_1, \alpha - l_1} \cap A_f^{\alpha, l_1 + l_2 - \alpha} = \emptyset$.
If $\delta = T_{k}$ and $d \geq 2$,
then $A_f^{l_{}, \alpha - l_{}} = A_f^{l_1, \alpha - l_1}$ 
for any $l_1 \leq l < \alpha$.
If $\delta = T_{k-1}$ and $d \geq 2$,
then $A_f^{\alpha, l_{} - \alpha} = A_f^{\alpha, l_2 - \alpha}$ 
for any $0 < l \leq l_2$.
\end{proposition}

The existence of the B\"{o}ttcher coordinate and
Corollary \ref{main cor on attr sets for Case 4} induce the following.

\begin{theorem} \label{main thm on G_z and G_f^a for Case 4}
If $d \geq 2$ or $T_k < \delta < T_{k-1}$,
then $G_z = \alpha G_p$ on $A_f$.
If $(n_1, m_1) \neq (0, \delta)$, 
then $G_f^{\alpha} = \alpha G_p$ on $A_0$.
\end{theorem}


\subsection{Tables for typical plurisubharmonic functions}

We end this section with the following two tables for the typical plurisubharmonic functions.

\begin{table}[htb]
\caption{Comparison chart of typical psh functions} \label{table for typical psh funs}
\begin{tabular}{|c|c|c|c|} \hline
 & $G_z$ & $G_f$ & $G_f^{\alpha}$  \\ \hline
$\delta < d$ & $G_z=G_z^{\alpha}$ on $A_0$ & $G_f=0$ on $A_0-E_z$ & $G_f^{\alpha}=0$ on $A_0-E_z$ \\
 & (psh on $A_0$) & $= - \infty$ on $E_z$ &  $= \infty$ on $E_z$  \\ \hline
$\delta = d$ and & $G_z = - \infty$ on $A_0$ & $G_f = G_p$ on $A_0$ &  \\ 
$\gamma > 0$ &  & (psh on $A_0$) &  \\ \hline
$\delta = d$ and & ph on $A_f$ & ph on $A_f$ & psh on $A_0$ \\ 
$\gamma = 0$ & $G_z \fallingdotseq \log |w|$ on $U$ & & $G_f^{\alpha} \fallingdotseq \log |w|$ on $U$ \\ \hline
$\delta > d$, and & $G_z = \alpha G_p$ on ... & psh on ... & $G_f^{\alpha} = \alpha G_p$ on $A_0$ \\ 
$d \neq 1$ or $\delta \neq T_k$ & (see Table \ref{table for G_z}) & (see Table \ref{table for G_z}) & (psh on $A_0$) \\ \hline
\end{tabular}
\end{table}

\begin{table}[h]
\caption{Comparison chart of $G_z$ and $G_f$ when $\delta > d$} \label{table for G_z}
\begin{tabular}{|c|c|} \hline
on ... & if ... \\  \hline
$A_0 - E_w$ & if $\delta < T_{s-1}$ and $d \geq 1$ \\ \hline
$A_0$ & if $\delta < T_{s-1}$ and $d = 0$ \\ \hline
$A_f - E_w$ & if $\delta = T_{s-1}$ and $d \geq 2$ \\ \hline
$A_f$ & if $\delta > T_{s-1}$, and if $d \geq 2$ \\
 & or if $d = 1$ and $\delta \neq T_k$ \\ \hline
\end{tabular}
\end{table}

\section{Intervals of weights and Blow-ups} 

More precise review of the study in the previous paper \cite{ueno} is given in this section.
We deal with Cases 2, 3 and 4 in Sections 3.1, 3.2 and 3.3,
respectively.
For each case,
we give the definition of the interval or rectangle $\mathcal{I}_f$,
calculate it,
restate Lemma \ref{main lemma for previous reselts} and
Theorem \ref{main thm for previous reselts}
by using $\mathcal{I}_f$, and
illustrate them in terms of blow-ups.

\subsection{Intervals of weights and Blow-ups for Case 2}

Let $s > 1$, 
\[
\delta \leq T_{s-1}, \
(\gamma, d) = (n_s, m_s)
\text{ and }
l_1 = \frac{n_s - n_{s-1}}{m_{s-1} - m_s}.
\]
Note that $\gamma > 0$ and $d \geq 0$ by the setting.

\subsubsection{Intervals and B\"{o}ttcher coordinates}

We define the interval $\mathcal{I}_f$ as 
\[
\mathcal{I}_f = 
\left\{ \ l > 0 \ | 
\begin{array}{lcr}
l \delta \leq \gamma + ld \leq i + l j
\text{ for any $i$ and $j$ s.t. } b_{ij} \neq 0
\end{array}
\right\}.
\]
It follows that $\min \mathcal{I}_f = l_1$.
In fact, 
if $\delta > d$, then 
\[
\mathcal{I}_f 
=
\left[
\max \left\{ \dfrac{\gamma - i}{j - d} \right\},
\dfrac{\gamma}{\delta - d}
\right]
=
\left[
\frac{\gamma - n_{s-1}}{m_{s-1} - d},
\dfrac{\gamma}{\delta - d}
\right]
=
\left[
l_1,
\alpha
\right],
\]
which is mapped to $[\delta, T_{s-1}]$
by the transformation $l \to l^{-1} \gamma + d$.
If $\delta \leq d$, then
the inequality $l \delta \leq \gamma + ld$ is trivial and so
$\mathcal{I}_f =[ l_1, \infty )$.
We remark that
\[
\mathcal{I}_f = 
\left\{ \ l > 0 \ | 
\begin{array}{lcr}
l \delta \leq \gamma + ld \leq n_j + l m_j
\text{ for any } 1 \leq j \leq s-1
\end{array}
\right\},
\]
which expression is used in Case 4.

Let $U^l = \{ |z| < r, |w| < r |z|^{l} \}$.
Note that $U = U^{l_1}$. 

\begin{lemma}[\cite{ueno}] \label{detailed lemma for case 2}
The following hold for any $l$ in $\mathcal{I}_f$ if $d \geq 2$
and for any $l$ in $\mathcal{I}_f- \{ \alpha \}$ if $d = 1$ and $\delta < T_{s-1}$:
$f \sim f_0$ on $U^l$ as $r \to 0$, and
$f(U^l) \subset U^l$ for small $r >0$.
\end{lemma}

This lemma induces the B\"{o}ttcher coordinate for $f$ on $U^l$.

\begin{theorem}[\cite{ueno}] \label{}
The following holds for any $l$ in the previous lemma:
$\phi$ is biholomorphic on $U^l$ and conjugates $f$ to $f_0$
for small $r > 0$.
Moreover,
$\phi \sim id$ on $U^{l}$ as $r \to 0$.
\end{theorem}

Note that $U^{l_1}$ is the largest open set among $U^{l}$ for any $l$ in $\mathcal{I}_f$.

\subsubsection{Blow-ups}

Assuming that $l$ in $\mathcal{I}_f$ is integer,  
we illustrate our previous results in terms of blow-ups.
We also assume that $p(z) = z^{\delta}$ and $b_{\gamma d} = 1$ for simplicity.
Let $\pi_1 (z,c) = (z, z^{l} c)$
and $\tilde{f} = \pi_1^{-1} \circ f \circ \pi_1$.
Note that $\pi_1$ is the $l$-th compositions of the blow-up $(z,c) \to (z,zc)$. 
Then 
\begin{align*}
\tilde{f} (z,c) 
&= (\tilde{p}(z), \tilde{q}(z,c))
= \left( p(z), \ \dfrac{q(z,z^l c)}{p(z)^l} \right) \\
&= \left( z^{\delta}, \ 
z^{\gamma+ld - l \delta} c^d + \sum b_{ij} z^{i + lj - l \delta} c^{j} \right) \\
&= \left( z^{\delta}, \ 
z^{\gamma+ld - l \delta} c^d
\left\{ 1 + \sum b_{ij} z^{(i + lj) - (\gamma + ld)} c^{j - d} \right\}  \right).
\end{align*}
Let $\tilde{\gamma} = \gamma + l d - l \delta$
and $\tilde{i} = i + l j - l \delta$.
Then 
$\tilde{q}(z,c) = z^{\tilde{\gamma}} c^d + \sum b_{ij} z^{\tilde{i}} c^{j}$ and 
$0 \leq \tilde{\gamma} \leq \tilde{i}$ for any $(i,j)$ such that $b_{ij} \neq 0$.
Therefore,
the Newton polygon of $\tilde{q}$ has just one vertex $(\tilde{\gamma}, d)$,
$N(\tilde{q}) = D(\tilde{\gamma}, d)$,
and we have the following. 

\begin{proposition}[\cite{ueno}]
If $l$ in $\mathcal{I}_f$ is integer,
then $\tilde{f}$ is holomorphic and skew product 
on a neighborhood of the origin. 
More precisely,
\[
\tilde{f} (z,c) = \left( z^{\delta}, \ 
z^{\tilde{\gamma}} c^d \left\{ 1 + \eta (z,c) \right\}  \right),
\]
where $\eta \to 0$ as $z$ and $c \to 0$.
Moreover,
it has a superattracting fixed point at the origin
if $d \geq 2$
or if $d = 1$ and $l \delta < \gamma + l d$.
\end{proposition}

\begin{remark}
Even if $l$ is rational,
we can lift $f$ to a holomorphic skew product similar to $\tilde{f}$
as stated in Section 5.1 in \cite{ueno}.
In particular,
$l_1$ is rational.
\end{remark}

Because $\tilde{f}$ is a holomorphic skew product in Case 1,
it is easy to construct the B\"{o}ttcher coordinate for $\tilde{f}$
on a neighborhood of the origin,
which induces the B\"{o}ttcher coordinate for $f$ on $U^l$.

\begin{remark}
The affine transformation
\[
A_1
\begin{pmatrix} i \\ j \end{pmatrix} =
\begin{pmatrix} i + l j - l \delta \\ j \end{pmatrix} =
\begin{pmatrix} 1 & l \\ 0 & 1 \end{pmatrix}
\begin{pmatrix} i \\ j \end{pmatrix} - \begin{pmatrix} l \delta \\ 0 \end{pmatrix}
\]
maps the basis $\{ (1,0), (-l,1) \}$ to $\{ (1,0), (0,1) \}$.
In other words, 
it maps a horizontal line and a line with slope $-l^{-1}$
to the same horizontal line and a vertical line.
More precisely,
it maps the line with slope $-l^{-1}$ that passes through $(i, \delta)$ 
to the vertical line $\{ x = i \}$.
In particular,
if $l = l_1$, then
it maps the line $L_{s-1}$ with slope $-l_1^{-1}$ and 
the horizontal line that intersect at $(\gamma, d)$,
to the vertical and horizontal lines 
that intersect at $(\tilde{\gamma}, d)$.
\end{remark}

\subsection{Intervals of weights and Blow-ups for Case 3}

Let $s > 1$, 
\[
T_1 \leq \delta, \
(\gamma, d) = (n_1, m_1)
\text{ and }
l_2 = \frac{n_2 - n_1}{m_1 - m_2}.
\]
Note that $\gamma \geq 0$ and 
$\delta \geq T_1 \geq d \geq 1$ by the setting. 

\subsubsection{Intervals and B\"{o}ttcher coordinates}

We define the interval $\mathcal{I}_f$ as
\[
\mathcal{I}_f = 
\left\{ \ l > 0 \ | 
\begin{array}{lcr}
\gamma + ld \leq l \delta  
\text{ and }
\gamma + ld \leq i + lj
\text{ for any $i$ and $j$ s.t. } b_{ij} \neq 0
\end{array}
\right\}. 
\]
It follows that $\max \mathcal{I}_f = l_2$.
In fact,
if $\gamma > 0$, then $\delta > d$ and 
\[
\mathcal{I}_f 
= 
\left[
\dfrac{\gamma}{\delta - d},
\min \left\{ \dfrac{i - \gamma}{d - j}  \right\}
\right]
=
\left[ 
\dfrac{\gamma}{\delta - d}, 
\frac{n_2 - \gamma}{d - m_2} 
\right]
= \left[ \alpha, l_2 \right],
\]
which is mapped to $[T_1, \delta]$
by the transformation $l \to l^{-1} \gamma + d$.
If $\gamma = 0$, then
the inequality $\gamma + ld \leq l \delta$ is trivial since $d \leq \delta$,
and so $\mathcal{I}_f = (0, l_2]$.

Let $U^l = \{ |z|^{l} < r^{l} |w|, |w| < r \}$.
Note that $U = U^{l_2}$. 

\begin{lemma}[\cite{ueno}] \label{detailed lemma for case 3}
The following hold for any $l$ in $\mathcal{I}_f$ if $d \geq 2$
and for any $l$ in $\mathcal{I}_f- \{ \alpha \}$ if $d = 1$ and $\delta > T_{1}$:
$f \sim f_0$ on $U^l$ as $r \to 0$, and
$f(U^l) \subset U^l$ for small $r >0$.
\end{lemma}

\begin{theorem}[\cite{ueno}] \label{}
The following holds for any $l$ in the previous lemma:
$\phi$ is biholomorphic on $U^l$ and conjugates $f$ to $f_0$
for small $r > 0$.
Moreover,
$\phi \sim id$ on $U^{l}$ as $r \to 0$.
\end{theorem}

Note that $U^{l_2}$ is the largest wedge among $U^{l}$ for any $l$ in $\mathcal{I}_f$.

\subsubsection{Blow-ups}

Assuming that $l^{-1}$ is integer for $l$ in $\mathcal{I}_f$,
we illustrate our previous results in terms of blow-ups.
Let $\pi_2 (t,w) = (t w^{l^{-1}}, w)$ and $\tilde{f} = \pi_2^{-1} \circ f \circ \pi_2$.
Note that $\pi_2$ is the $l^{-1}$-th compositions of the blow-up $(t,w) \to (tw,w)$.
Let $b_{\gamma d} = 1$ for simplicity.
Then 
\[
\tilde{f} (t,w) 
= (\tilde{p} (t,w), \tilde{q} (t,w))
= \left( \dfrac{p(tw^{l^{-1}})}{q(tw^{l^{-1}}, w)^{l^{-1}}}, \ q(tw^{l^{-1}}, w) \right)
\text{ and }
\]
\begin{align*}
\tilde{q} (t,w) &= q(tw^{l^{-1}},w)
= t^{\gamma} w^{l^{-1} \gamma + d} + \sum b_{ij} t^{i} w^{l^{-1} i + j} \\
&= t^{\gamma} w^{l^{-1} \gamma + d} 
  \left\{ 1 + \sum b_{ij} t^{i - \gamma} w^{(l^{-1} i + j) - (l^{-1} \gamma + d)} \right\}.
\end{align*}
Let $\tilde{d} = l^{-1} \gamma + d$ and $\tilde{j} = l^{-1} i + j$.
Then $\tilde{q}(t,w) = t^{\gamma} w^{\tilde{d}} + \sum b_{ij} t^{i} w^{\tilde{j}}$,  
$d \leq \tilde{d} \leq \delta$ and
$\tilde{d} \leq \tilde{j}$ for any $(i,j)$ such that $b_{ij} \neq 0$.
Therefore,
the Newton polygon of $\tilde{q}$ has just one vertex $(\gamma, \tilde{d})$,
$N(\tilde{q}) = D(\gamma, \tilde{d})$,
and we have the following. 

\begin{proposition}[\cite{ueno}]
If $l^{-1}$ is integer for $l$ in $\mathcal{I}_f$,
then $\tilde{f}$ is holomorphic 
on a neighborhood of the origin. 
More precisely,
\[
\tilde{f} (t,w) = 
\left( t^{\delta - l^{-1} \gamma} w^{l^{-1} ( \delta - \tilde{d} )} \{ 1 + \zeta (t,w) \},
\ t^{\gamma} w^{\tilde{d}} \{ 1 + \eta (t,w) \} \right),
\]
where $\zeta$ and $\eta \to 0$ as $t$ and $w \to 0$.
Moreover,
it has a superattracting fixed point at the origin
if $d \geq 2$ 
or if $d = 1$ and $\gamma + ld < l \delta $.
\end{proposition}

Although $\tilde{f}$ is not skew product,
it is a perturbation of a monomial map near the origin.
Hence we can construct the B\"{o}ttcher coordinate for $\tilde{f}$
on a neighborhood of the origin
by similar arguments as in Section 6 in \cite{ueno},
or one may refer to pp.498-499 in \cite{f}, 
which induces the B\"{o}ttcher coordinate for $f$ on $U^l$.

\begin{remark}
The linear transformation
\[
A_2 
\begin{pmatrix} i \\ j \end{pmatrix} =
\begin{pmatrix} i \\ l^{-1} i + j \end{pmatrix} =
\begin{pmatrix} 1 & 0 \\ l^{-1} & 1 \end{pmatrix}
\begin{pmatrix} i \\ j \end{pmatrix}
\]
maps the basis $\{ (1,-l^{-1}), (0,1) \}$ to $\{ (1,0), (0,1) \}$.
In other words, 
it maps a line with slope $-l^{-1}$ and a vertical line 
to a horizontal line and the same vertical line.
In particular,
if $l = l_2$, then
it maps the line $L_{1}$ with slope $-l_2^{-1}$ and the vertical line
that intersect at $(\gamma, d)$, 
to the horizontal and vertical lines
that intersect at $(\gamma, \tilde{d})$.
\end{remark}

\subsection{Intervals, Rectangles and Blow-ups for Case 4}

Let $s > 2$, 
\[
T_k \leq \delta \leq T_{k-1}
\text{ for some } 
2 \leq k \leq s-1,
\]
\[
(\gamma, d) = (n_k, m_k), \,
l_1 = \frac{n_k - n_{k-1}}{m_{k-1} - m_k} 
\text{ and } 
l_1 + l_2= \frac{n_{k+1} - n_k}{m_k - m_{k+1}}.
\]
Note that $\gamma > 0$ and $\delta > d \geq 1$ by the setting. 

\subsubsection{Intervals, Rectangles and B\"{o}ttcher coordinates}

We define the interval $\mathcal{I}_f^1$ as
\[
\mathcal{I}_f^1 = 
\left\{ \ l_{(1)} > 0 \ \Bigg| 
\begin{array}{lcr}
l_{(1)} \delta \leq \gamma + l_{(1)} d  \\
\gamma + l_{(1)} d \leq n_{j} + l_{(1)} m_{j} \text{ for } 1 \leq j \leq k-1 \\ 
\gamma + l_{(1)} d  <   n_{j} + l_{(1)} m_{j} \text{ for } k+1 \leq j \leq s
\end{array}
\right\},
\]
the interval $\mathcal{I}_f^2$ associated with $l_{(1)}$ in $\mathcal{I}_f^1$ as
\[
\mathcal{I}_f^2  = \mathcal{I}_f^2 (l_{(1)}) = 
\left\{ \ l_{(2)} > 0 \ \Big| 
\begin{array}{lcr}
\tilde{\gamma} + l_{(2)} d \leq l_{(2)} \delta
\text{ and }
\tilde{\gamma} + l_{(2)} d \leq \tilde{i} + l_{(2)} j \\
\text{for any $i$ and $j$ s.t. } b_{ij} \neq 0 
\end{array}
\right\},
\]
where 
$\tilde{\gamma} = \gamma + l_{(1)} d - l_{(1)} \delta$
and $\tilde{i} = i + l_{(1)} j - l_{(1)} \delta$,
and the rectangle $\mathcal{I}_f$ as
\[
\mathcal{I}_f =
\{ (l_{(1)}, l_{(1)} + l_{(2)}) \ | \ l_{(1)} \in \mathcal{I}_f^1, l_{(2)} \in \mathcal{I}_f^2 \}.
\]

Let us calculate the intervals and rectangle more practically.
Note that $\alpha > 0$ since $\delta > d$ and $\gamma > 0$.
Since $n_j < \gamma$ and $m_j > d$ for any $1 \leq j \leq k-1$,
and $n_j > \gamma$ and $m_j < d$ for any $k+1 \leq j \leq s$,  
\begin{eqnarray*}
\mathcal{I}_f^1 &=& 
\left[
\max_{1 \leq j \leq k-1}
\left\{ 
\dfrac{\gamma - n_j}{m_j - d} 
\right\},
\min_{k+1 \leq j \leq s}
\left\{ 
\dfrac{n_j - \gamma}{d - m_j} 
\right\}
\right)
\cap
\left(
0,
\dfrac{\gamma}{\delta - d}
\right] \\
&=& 
\left[
\dfrac{\gamma - n_{k-1}}{m_{k-1} - d},
\dfrac{n_{k+1} - \gamma}{d - m_{k+1}} 
\right)
\cap
\left(
0,
\dfrac{\gamma}{\delta - d}
\right]
= [ l_1, l_1 + l_2 ) \cap ( 0, \alpha ].
\end{eqnarray*}
On the other hand,
\begin{eqnarray*}
\mathcal{I}_f^2 
&=&
\left[
\dfrac{\tilde{\gamma}}{\delta - d},
\dfrac{\tilde{n}_{k+1} - \tilde{\gamma}}{d - m_{k+1}}  
\right]
\cap \mathbb{R}_{>0}
=
\left[
\dfrac{\gamma}{\delta - d} - l_{(1)},
\dfrac{n_{k+1} - \gamma}{d - m_{k+1}} - l_{(1)}  
\right]
\cap \mathbb{R}_{>0} \\
&=& 
[ \alpha - l_{(1)}, l_1 + l_2 - l_{(1)} ]
\cap \mathbb{R}_{>0}.
\end{eqnarray*}
If $T_k < \delta < T_{k-1}$, 
then it follows from the inequality $l_1 < \alpha < l_1 + l_2$ that
\[
\mathcal{I}_f^1 = [ l_1, \alpha ], \ 
\mathcal{I}_f^2 = [ \alpha - l_{(1)}, l_1 + l_2 - l_{(1)} ] \cap \mathbb{R}_{>0}
\text{ and so }
\]
\[
\mathcal{I}_f 
= [ l_1, \alpha ] \times [ \alpha, l_1 + l_2 ] - \{ (\alpha, \alpha) \}.
\]
If $T_{k} = \delta < T_{k-1}$, then it follows from the inequality $l_1 < \alpha = l_1 + l_2$ that
\[
\mathcal{I}_f^1 = [ l_1, \alpha ), \ 
\mathcal{I}_f^2 = \{ \alpha - l_{(1)} \}
\text{ and so }
\mathcal{I}_f 
= [ l_1, \alpha ) \times \{ \alpha \}.
\]
If $T_k < \delta = T_{k-1}$, 
then it follows from the inequality $l_1 = \alpha < l_1 + l_2$ that
\[
\mathcal{I}_f^1 = \{ \alpha \}, \ 
\mathcal{I}_f^2 = (0, l_2] 
\text{ and so }
\mathcal{I}_f 
= \{ \alpha \} \times ( \alpha, l_1 + l_2 ].
\]
In particular, 
$\min \mathcal{I}_f^1 = l_1$ and
$\max \{ l_{(1)} + l_{(2)} \ | \ l_{(1)} \in \mathcal{I}_f^1, l_{(2)} \in \mathcal{I}_f^2 \} = l_1 + l_2$.

Let $U^{l_{(1)},l_{(2)}} = \{ |z|^{l_{(1)} + l_{(2)}} < r^{l_{(2)}} |w|, |w| < r|z|^{l_{(1)}} \}$.
Note that $U = U^{l_1, l_2}$.

\begin{lemma}[\cite{ueno}] \label{detailed lemma for case 4}
The following hold for any $l_{(1)}$ in $\mathcal{I}_f^1$ and 
$l_{(2)}$ in $\mathcal{I}_f^2$ if $d \geq 2$, and
for any $l_{(1)}$ in $\mathcal{I}_f^1 - \{ \alpha \}$ and 
$l_{(2)}$ in $\mathcal{I}_f^2 - \{ \alpha - l_{(1)} \}$
if $d = 1$ and $T_k < \delta < T_{k-1}$:
$f \sim f_0$ on $U^{l_{(1)},l_{(2)}}$ as $r \to 0$, and
$f(U^{l_{(1)},l_{(2)}}) \subset U^{l_{(1)},l_{(2)}}$ for small $r >0$.
\end{lemma}

\begin{theorem}[\cite{ueno}] \label{}
The following holds for any $l_{(1)}$ and $l_{(2)}$ in the previous lemma:
$\phi$ is biholomorphic on $U^l$ and conjugates $f$ to $f_0$
for small $r > 0$.
Moreover,
$\phi \sim id$ on $U^{l_{(1)},l_{(2)}}$ as $r \to 0$.
\end{theorem}

Note that $U^{l_1, l_2}$ is the largest wedge among $U^{l_{(1)},l_{(2)}}$ 
for any $l_{(1)}$ in $\mathcal{I}_f^1$ and $l_{(2)}$ in $\mathcal{I}_f^2$.

\subsubsection{Blow-ups}

Assuming that $l_{(1)}$ and $l_{(2)}^{-1}$ are integer
for $l_{(1)}$ in $\mathcal{I}_f^1$ and $l_{(2)}$ in $\mathcal{I}_f^2$,
we illustrate our previous results in terms of blow-ups.
The strategy is to combine the blow-ups in Cases 2 and 3.
We first blow-up $f$ to $\tilde{f}_1$
by $\pi_1$ as in Case 2.
It then turns out that $\tilde{f}_1$ is a holomorphic skew product in Case 3.
We next blow-up $\tilde{f}_1$ to $\tilde{f}_2$
by $\pi_2$ as in Case 3.
The map $\tilde{f}_2$ is a perturbation of a monomial map near the origin,
and we obtain the B\"{o}ttcher coordinate.

Let us restate the first blow-up.
Let $\pi_1 (z,c) = (z, z^{l_{(1)}} c)$
and $\tilde{f}_1 = \pi_1^{-1} \circ f \circ \pi_1$ as in Case 2,
and let $\tilde{f}_1 (z,c) = (\tilde{p}_1 (z), \tilde{q}_1 (z,c))$.
Let $p(z) = z^{\delta}$ and $b_{\gamma d} = 1$ for simplicity.
Then 
\[
\tilde{f}_1 (z,c) 
= \left( z^{\delta}, z^{\gamma + l_{(1)} d - l_{(1)} \delta} c^d 
+ \sum b_{ij} z^{i + l_{(1)} j - l_{(1)} \delta} c^{j} \right).
\]
Let $\tilde{\gamma} = \gamma + l_{(1)} d - l_{(1)} \delta$
and $\tilde{i} = i + l_{(1)} j - l_{(1)} \delta$
as in Case 2.
Then
$0 \leq \tilde{\gamma} \leq \tilde{i}$ for any $(i,j)$ such that $b_{ij} \neq 0$,
and we have the following. 

\begin{proposition}[\cite{ueno}]
If $l_{(1)}$ in $\mathcal{I}_f^1$ is integer,
then $\tilde{f}_1$ is holomorphic and skew product 
on a neighborhood of the origin.  
More precisely,
\[
\tilde{f}_1 (z,c) =
\left( z^{\delta}, z^{\tilde{\gamma}} c^d + \sum b_{ij} z^{\tilde{i}} c^{j} \right).
\]
Moreover, 
it has a superattracting fixed point at the origin
if $d \geq 2$ or if $d = 1$ and $l_{(1)} \delta < \gamma + l_{(1)} d$. 
\end{proposition}

Note that 
$(\tilde{\gamma}, d)$ is the vertex of the Newton polygon $N(\tilde{q}_1)$
whose $x$-coordinate is minimum,
and that
$N(\tilde{q}_1)$ has other vertices such as $(\tilde{n}_{k+1}, m_{k+1})$.
Hence the situation resembles that of Case 3.

We illustrate that $\tilde{f}_1$ is actually in Case 3. 
Recall that $L_{k}$ is the line passing through 
the vertices $(\gamma, d)$ and $(n_{k+1}, m_{k+1})$,
and $T_{k}$ is the $y$-intercept of $L_{k}$.
The slope of $L_{k}$ is $-(l_1 + l_2)^{-1}$ 
and so $T_{k} = (l_1 + l_2)^{-1} \gamma + d$.
Let $\tilde{L}_{k}$ be the line passing through  
the vertices $(\tilde{\gamma}, d)$ and $(\tilde{n}_{k+1}, m_{k+1})$,
and $\tilde{T}_{k}$ the $y$-intercept of $\tilde{L}_{k}$,
where $\tilde{n}_{k+1} = n_{k+1} + l_{(1)} m_{k+1} - l_{(1)} \delta$.
Then the slope of $\tilde{L}_{k}$ is $-(l_1 + l_2 - l_{(1)})^{-1}$ 
and so $\tilde{T}_{k} = (l_1 + l_2 - l_{(1)})^{-1} \tilde{\gamma} + d$.
Therefore,
the inequality $T_k \leq \delta$ implies 
the inequality $\tilde{T}_k \leq \delta$. 
More precisely,
$\tilde{T}_k < \delta$ if $T_k < \delta$,
and $\tilde{T}_k = \delta$ if $T_k = \delta$.
In particular, 

\begin{proposition}[\cite{ueno}]
If $l_{(1)}$ in $\mathcal{I}_f^1$ is integer,
then $\tilde{f}_1$ is a holomorphic skew product in Case 3. 
\end{proposition}

Let us restate the second blow-up.
Let $\pi_2 (t,c) = (t c^{l_{(2)}^{-1}}, c)$ 
and $\tilde{f}_2 = \pi_2^{-1} \circ \tilde{f}_1 \circ \pi_2$ as in Case 3,
and let $\tilde{f}_2 (t,c) = (\tilde{p}_2 (t,c), \tilde{q}_2 (t,c))$.
Let $\tilde{d} = l_{(2)}^{-1} \tilde{\gamma} + d$
and $\tilde{j} = l_{(2)}^{-1} \tilde{i} + j$
as in Case 3.
Then
$\tilde{d} \leq \tilde{j}$ and $\tilde{d} \leq \delta$ for any $(i,j)$ such that $b_{ij} \neq 0$.
Note that
$(\tilde{\gamma}, \tilde{d})$ is the minimum in the sense that
$0 \leq \tilde{\gamma} \leq \tilde{i}$ and 
$d \leq \tilde{d} \leq \tilde{j}$ for any $(i,j)$ such that $b_{ij} \neq 0$.
Therefore,
the Newton polygon of $\tilde{q}_2$ has just one vertex $(\tilde{\gamma}, \tilde{d})$:
$N(\tilde{q}_2) = D(\tilde{\gamma}, \tilde{d})$,
and we have the following. 

\begin{proposition}[\cite{ueno}]
If $l_{(1)}$ and $l_{(2)}^{-1}$ are integer
for $l_{(1)}$ in $\mathcal{I}_f^1$ and $l_{(2)}$ in $\mathcal{I}_f^2$,
then $\tilde{f}_2$ is holomorphic 
on a neighborhood of the origin.  
More precisely,
\[
\tilde{f}_2 (t,c) =
\left( t^{\delta - l_{(2)}^{-1} \tilde{\gamma}} c^{l_{(2)}^{-1} (\delta - \tilde{d})} \{ 1 + \zeta_2 (t,c) \}, 
\ t^{\tilde{\gamma}} c^{\tilde{d}} \{ 1 + \eta_2 (t,c) \} \right),
\]
where $\zeta_2$ and $\eta_2 \to 0$ as $t$ and $c \to 0$.
Moreover, 
it has a superattracting fixed point at the origin
if $d \geq 2$ or if $d = 1$, $l_{(1)} \delta < \gamma + l_{(1)} d$
and $\tilde{\gamma} + l_{(2)} d < l_{(2)} \delta$. 
\end{proposition}

Consequently,
we can construct the B\"{o}ttcher coordinate for $\tilde{f}_2$
on $\{ |t| < r, |c| < r \}$,
which induces that for $\tilde{f}_1$ 
on $\{ |z| < r |c|^{l_{(2)}^{-1}}, |c| < r \}$ and 
that for $f$ 
on $U^{l_{(1)},l_{(2)}}$.  

\begin{remark}
The affine transformation
\[
A 
\begin{pmatrix} i \\ j \end{pmatrix} =
\begin{pmatrix} 1 & 0 \\ l_2^{-1} & 1 \end{pmatrix}
\left\{  
\begin{pmatrix} 1 & l_1 \\ 0 & 1 \end{pmatrix}
\begin{pmatrix} i \\ j \end{pmatrix} - 
\begin{pmatrix} l_1 \delta \\ 0 \end{pmatrix}
\right\}
\]
is the composition of the two affine transformations
\[
A_1
\begin{pmatrix} i \\ j \end{pmatrix} =
\begin{pmatrix} i + l_1 j - l_1 \delta \\ j \end{pmatrix} 
\text{ and }
A_2
\begin{pmatrix} i \\ j \end{pmatrix} =
\begin{pmatrix} i \\ l_2^{-1} i + j \end{pmatrix}. 
\]
Although statements similar to below hold 
for any $l_{(1)}$ in $\mathcal{I}_f^1$ and $l_{(2)}$ in $\mathcal{I}_f^2$,
instead of $l_1$ and $l_2$,
we illustrate only the case $l_{(1)} = l_1$ and $l_{(2)} = l_2$ for simplicity.
The transformation $A_1$ maps 
the basis $\{ (1,-(l_1 +l_2)^{-1}), (-l_1,1) \}$ to $\{ (1,-l_2^{-1}), (0,1) \}$. 
In other words, 
it maps the line $L_{k}$ with slope $-(l_1 +l_2)^{-1}$ and the line $L_{k-1}$ with slope $-l_1^{-1}$,
which intersect at $(\gamma, d)$,
to the line $\tilde{L}_{k}$ with slope $-l_2^{-1}$ and the vertical line
that intersect at $(\tilde{\gamma}, d)$. 
The transformation $A_2$ maps 
the basis $\{ (1,-l_2^{-1}),  (0,1) \}$ to $ \{ (1,0), (0,1) \}$.
In other words, 
it maps the line $\tilde{L}_{k}$ and the vertical line
that intersect at $(\tilde{\gamma}, d)$, 
to the horizontal and vertical lines
that intersect at $(\tilde{\gamma}, \tilde{d})$. 
\end{remark}

\section{Case 2: Invariant wedges and Attracting sets}

We deal with Case 2 in this and the next sections:
we prove Theorems \ref{main thm on inv wedges for Case 2}
and \ref{main thm on attr sets for Case 2},
the main results on invariant wedges and the unions of all the preimages of the wedges,
in this section and 
prove Theorems \ref{main thm on G_z^a for Case 2} and 
\ref{main thm on G_z, G_f and G_f^a for Case 2} and
Propositions \ref{main propo on A_f^l for case 2} and 
\ref{main prop on the unboundness for case 2},
the main results on plurisubharmonic functions,
in the next section.
Let $s > 1$, 
\[
\delta \leq T_{s-1}
\text{ and }
(\gamma, d) = (n_s, m_s).
\]
Note that $\gamma > 0$ and $d \geq 0$ by the setting. 

Before we give proofs of Theorems \ref{main thm on inv wedges for Case 2}
and \ref{main thm on attr sets for Case 2} in Section 4.2,
we illustrate the theorems in terms of blow-ups in Section 4.1.
Actually, we check the following.

\begin{proposition}\label{properties of blow-up for Case 2}
Let $\delta \leq d$.
Then $\tilde{f}$ is superattracting at the origin and degenerates the $c$-axis for any $l > 0$. 
Let $\delta > d$.
Then $\tilde{f}$ is superattracting at the origin for any $0 < l < \alpha$,
even for $l = \alpha$ if $d \geq 2$, and
degenerates the $c$-axis for any $0 < l < \alpha$ unless $(n_1, m_1) = (0, \delta)$.
\end{proposition}

Roughly speaking,
if $\tilde{f}$ is superattracting at the origin, 
then the open set $U_{r_1, r_2}^l$ should be invariant under $f$, and 
if $\tilde{f}$ degenerates the $c$-axis,
then one may expect that $A_f^l = A_0 - E_z$.

\begin{remark}[Another illustration of Theorem \ref{main thm on attr sets for Case 2}]
One may expect that the vertical dynamics 
eventually converges to the dynamics on the $w$-axis
since $p^n(z) \to 0$ as $n \to \infty$.
If $n_1 > 0$, then $f$ degenerates the $w$-axis to the origin
and so one may expect that
Theorem \ref{main thm on attr sets for Case 2} holds.
If $n_1 = 0$, then $q(0,w) = b_{n_1 m_1} w^{m_1} (1 + o(1))$
and so one may expect that
Theorem \ref{main thm on attr sets for Case 2} holds 
if $\delta < m_1$
since $p(z) \sim a z^{\delta}$.
\end{remark}

\subsection{Illustration of Theorems \ref{main thm on inv wedges for Case 2}
and \ref{main thm on attr sets for Case 2} in terms of blow-ups}

Recall that $\tilde{f} = \pi_1^{-1} \circ f \circ \pi_1$,
where $\pi_1 (z,c) = (z, z^{l} c)$.
Let $\tilde{f} (z,c) = (\tilde{p}(z), \tilde{q}(z,c))$ and
we assume that $p(z) = z^{\delta}$ for simplicity.
Then $\tilde{p}(z) = z^{\delta}$ and 
$\tilde{q}(z,c) = \sum b_{ij} z^{i + lj - l \delta} c^{j}$.
Let $\tilde{\gamma} = \gamma + l d - l \delta$
and $\tilde{n}_j = n_j + l m_j - l \delta$. 
The shape of the Newton polygon $N(\tilde{q})$ of $\tilde{q}$ 
is useful to check 
whether $\tilde{f}$ is superattracting at the origin, and 
whether $\tilde{f}$ degenerates the $c$-axis.
Recall that $A_1$ maps 
the line with slope $-l^{-1}$ that passes through the point $(i, \delta)$
to the vertical line $\{ x = i \}$,
whereas it preserves all the horizontal lines.

Let us describe the shape of the Newton polygon of $\tilde{q}$ and 
properties of $\tilde{f}$ for the cases $\delta \leq d$,
$\delta > d$ and $\delta < T_{s-1}$, and $\delta > d$ and $\delta = T_{s-1}$
in Sections 4.1.1, 4.1.2 and 4.1.3, respectively,
which implies Proposition \ref{properties of blow-up for Case 2}.

\subsubsection{$\delta \leq d$}

Recall that $\mathcal{I}_f = [ l_1, \infty )$ if $\delta \leq d$.
See Table \ref{shape when delta <= d} below for a summary of 
the shape of $N(\tilde{q})$ and properties of $\tilde{f}$,
where the notation NA means that
$\tilde{f}$ is superattracting at the origin, and
the notation Deg means that
$\tilde{f}$ degenerates the $c$-axis.

\begin{table}[htb]
\caption{Shape of $N(\tilde{q})$ when $\delta \leq d$} \label{shape when delta <= d}
\begin{tabular}{|c|c|c|c|} \hline
$l$ & $0 < l < l_1$ & $l_1$ & $l_1 < l$  \\ \hline
$\tilde{\gamma}$ & $0 < \tilde{n}_{s-1} < \tilde{\gamma}$ & $0 < \tilde{\gamma} = \tilde{n}_{s-1} < \tilde{n}_j$ & $0 < \tilde{\gamma} < \tilde{n}_j$ \\ 
              & $0 < \tilde{n}_j$ for any $j$ & for $j \leq s-2$ & for $j \leq s-1$  \\ \hline
$N(\tilde{q})$ & has other & has the only & has the only  \\ 
  & vertices & one vertex & one vertex  \\ 
              & than $(\tilde{\gamma}, d)$ & $(\tilde{\gamma}, d)$ & $(\tilde{\gamma}, d)$  \\ \hline
$\tilde{f}$  & SA and Deg & SA and Deg & SA and Deg  \\ \hline
\end{tabular}
\end{table}

For any $l$ in $\mathcal{I}_f$,
$N(\tilde{q}) = (\tilde{\gamma}, d)$ and $\tilde{\gamma} > 0$.
Hence
\[
\tilde{q}(z,c) = z^{\tilde{\gamma}} c^d
\left\{ 1 + \eta (z,c) \right\},
\]
where $\eta \to 0$ as $z$ and $c \to 0$,
and $\tilde{f}$ is superattracting at the origin and
degenerates the $c$-axis.
More precisely,
$\tilde{\gamma} = \tilde{n}_{s-1} < \tilde{n}_j$
for any $1 \leq j \leq s-2$ if $l = l_1$, and
$\tilde{\gamma} < \tilde{n}_j$
for any $1 \leq j \leq s-1$ if $l > l_1$.

On the other hand,
for any $0 < l < l_1$,
$N(\tilde{q})$ has other vertices than $(\tilde{\gamma}, d)$,
whereas $(\tilde{\gamma}, d)$ is the vertex of $N(\tilde{q})$
whose $y$-coordinate is the minimum and $\tilde{\gamma} > 0$.
More precisely,
the candidates of vertices of $N(\tilde{q})$ are $(\tilde{n}_j, m_j)$'s, and
$(\tilde{n}_{s-1}, m_{s-1})$ is actually one of the vertices of $N(\tilde{q})$.
However,
since $\tilde{n}_j > 0$ for any $j$,
$\tilde{f}$ is superattracting at the origin and
degenerates the $c$-axis.

\subsubsection{$\delta > d$ and $\delta < T_{s-1}$}

Recall that $\mathcal{I}_f = [ l_1, \alpha ]$ if $\delta > d$ and $\delta < T_{s-1}$. 
See Table \ref{shape when d < delta < T} below for a summary of 
the shape of $N(\tilde{q})$ and properties of $\tilde{f}$.

\begin{table}[htb]
\caption{Shape of $N(\tilde{q})$ when $\delta > d$ and $\delta < T_{s-1}$} \label{shape when d < delta < T}
\begin{tabular}{|c|c|c|c|c|} \hline
$l$  & $0 < l < l_1$ & $l_1$ & $l_1 < l < \alpha$ & $\alpha$  \\ \hline
$\tilde{\gamma}$ & $0 < \tilde{n}_{s-1} < \tilde{\gamma}$ & $0 < \tilde{\gamma}  = \tilde{n}_{s-1} < \tilde{n}_j$ & $0 < \tilde{\gamma} < \tilde{n}_j$ & $0 = \tilde{\gamma} < \tilde{n}_j$   \\ 
              & $0 < \tilde{n}_j$ for any $j$ & for $j \leq s-2$ & for $j \leq s-1$ & for $j \leq s-1$  \\ \hline
$N(\tilde{q})$ & has other & has the only & has the only & has the only  \\ 
  & vertices & one vertex & one vertex & one vertex   \\ 
              & than $(\tilde{\gamma}, d)$ & $(\tilde{\gamma}, d)$ & $(\tilde{\gamma}, d)$ & $(\tilde{\gamma}, d) = (0, d)$  \\ \hline
$\tilde{f}$  & SA and Deg & SA and Deg & SA and Deg & SA if $d \geq 2$  \\ \hline
\end{tabular}
\end{table}

For any $l$ in $\mathcal{I}_f$,
$N(\tilde{q}) = (\tilde{\gamma}, d)$.
Moreover,
if $l < \alpha$, then
$\tilde{\gamma} > 0$ and so
$\tilde{f}$ is superattracting at the origin and
degenerates the $c$-axis.
On the other hand,
if $l = \alpha$, 
then $\tilde{\gamma} = 0$,
\[
\tilde{q}(z,c) = c^d \left\{ 1 + \eta (z,c) \right\},
\]   
and so
$\tilde{f}$ is superattracting at the origin only if $d \geq 2$, and
$\tilde{f}$ does not degenerate the $c$-axis.
For any $0 < l < l_1$,
the same things hold as the previous case.
We remark that if $l > \alpha$, then
$\tilde{\gamma} < 0$ and so $\tilde{f}$ is not holomorphic.

\subsubsection{$\delta > d$ and $\delta = T_{s-1}$}
 
If $\delta = T_{s-1}$,
then $\delta > d$ and $\mathcal{I}_f = \{ l_1 \} = \{ \alpha \}$.
The shape of $N(\tilde{q})$ and properties of $\tilde{f}$ differ whether $(n_1, m_1) = (0, \delta)$ or not.

If $(n_1, m_1) \neq (0, \delta)$,
then the situation is somehow similar to the previous case.
If $l = \alpha$, 
then $\tilde{\gamma} = \tilde{n}_{s-1} = 0 < \tilde{n}_j$
for any $1 \leq j \leq s-2$ and so
$\tilde{f}$ is superattracting at the origin only if $d \geq 2$, and
it does not degenerate the $c$-axis.
For any $0 < l < l_1$,
the same things hold as the previous case.
See Table \ref{shape when delta = T} below for a summary.

\begin{table}[htb]
\caption{Shape of $N(\tilde{q})$ when $\delta = T_{s-1}$ and $(n_1, m_1) \neq (0, \delta)$} \label{shape when delta = T} 
\begin{tabular}{|c|c|c|} \hline
$l$ & $0 < l < l_1 = \alpha$ & $l_1 = \alpha$   \\ \hline
$\tilde{\gamma}$ & $0 < \tilde{n}_{s-1} < \tilde{\gamma}$ & $0 = \tilde{\gamma} = \tilde{n}_{s-1} < \tilde{n}_j$   \\ 
              & $0 < \tilde{n}_j$ for any $j$ & for $j \leq s-2$  \\ \hline
$N(\tilde{q})$ & has other & has the only   \\ 
  & vertices & one vertex    \\ 
              & than $(\tilde{\gamma}, d)$ & $(\tilde{\gamma}, d) = (0, d)$   \\ \hline
$\tilde{f}$  & SA and Deg & SA if $d \geq 2$  \\ \hline
\end{tabular}
\end{table}


On the other hand,
if $(n_1, m_1) = (0, \delta)$,
then the situation differs from the previous case.
Note that $s = 2$ for this case.
If $l = \alpha$, 
then $\tilde{\gamma} = \tilde{n}_1 = 0$,
\[
\tilde{q}(z,c) 
= b_{0 \delta} c^{\delta} + \cdots + b_{\gamma d} c^d
+ \sum_{i+lj > l \delta} b_{ij} z^{i + lj - l \delta} c^{j}
= c^d \left\{ 1 + \eta (z,c) \right\}
\]  
and so
$\tilde{f}$ is superattracting at the origin only if $d \geq 2$, and
$\tilde{f}$ does not degenerate the $c$-axis.
For any $0 < l < l_1$,
$\tilde{\gamma} > \tilde{n}_1 = 0$, 
\[
\tilde{q}(z,c) = b_{0 \delta} c^{\delta}
+ \sum_{i+lj > l \delta} b_{ij} z^{i + lj - l \delta} c^{j}
\] 
and so
$\tilde{f}$ is superattracting at the origin,
but it does not degenerate the $c$-axis.
See Table \ref{shape for special case} below for a summary.

\begin{table}[htb]
\caption{Shape of $N(\tilde{q})$ when $(n_1, m_1) = (0, \delta)$} \label{shape for special case}
\begin{tabular}{|c|c|c|} \hline
$l$  & $0 < l < l_1 = \alpha$ & $l_1 = \alpha$   \\ \hline
$\tilde{\gamma}$ & $0 = \tilde{n}_1 < \tilde{\gamma}$ & $0 = \tilde{n}_1 = \tilde{\gamma}$   \\ \hline
$N(\tilde{q})$ & has other & has the only   \\ 
  & vertices & one vertex    \\ 
              & than $(\tilde{\gamma}, d)$ & $(\tilde{\gamma}, d) = (0, d)$   \\ \hline
$\tilde{f}$  & SA & SA if $d \geq 2$  \\ \hline
\end{tabular}
\end{table}

\subsection{Proof of Theorems \ref{main thm on inv wedges for Case 2}
and \ref{main thm on attr sets for Case 2}}

The following quantity plays an important role in the proofs:
\[
D = D_l = \min \left\{ l^{-1} i + j \ | \ b_{ij} \neq 0 \right\}. 
\]
By definition,
$D > d$ for any $l > 0$
since $D \geq l^{-1} \gamma + d$ and $\gamma > 0$.
Note that $l^{-1} i + j$ is the $y$-intercept of the line with slope $- l^{-1}$
that passes through $(i,j)$.
Hence, for any $l > 0$, 
the minimum is attached by a vertex $(n_j, m_j)$. 
Let $(N,M)$ be such a vertex: $D = l^{-1} M + N$.
Then we have the following. 

\begin{lemma}\label{properties of D for Case 2}
It follows that $(N,M) = (\gamma, d)$ and
$T_{s-1} \geq D = l^{-1} \gamma + d \geq \delta$ for any $l$ in $\mathcal{I}_f$.
Moreover,
$D = T_{s-1}$ for $l = l_1$, and $D = \delta$ for $l = \alpha$ if $\delta > d$.
On the other hand,
$(N,M) \neq (\gamma, d)$ and
$D = l^{-1} M + N \geq T_{s-1} \geq \delta$ for any $0 < l < l_1$.
More precisely,
$(M, N) = (n_1, m_1)$ or $(M, N) = (n_k, m_k)$ for some $2 \leq k \leq s-1$ 
\[
\text{ if }
0 < l \leq \dfrac{n_{2} - n_{1}}{m_{1} - m_{2}}
\text{ or }
\dfrac{n_{k} - n_{k-1}}{m_{k-1} - m_{k}} \leq l \leq \dfrac{n_{k+1} - n_{k}}{m_{k} - m_{k+1}}.
\]
\end{lemma}

Consequently,
the inequality $D \geq \delta$ holds
for any $l > 0$ if $\delta \leq d$, and 
for any $0 < l \leq \alpha$ if $\delta > d$.
Let us make clear when the inequality $D > \delta$ or 
the equality $D = \delta$ holds.
If $\delta \leq d$,
then $D > d \geq \delta$ for any $l > 0$.
If $\delta > d$,
then $D \geq \delta$ for any $0 < l < \alpha$,
and $D = \delta$ for $l = \alpha$.
More precisely,
if $\delta > d$,
then the inequality $D > \delta$ holds
for any $0 < l < \alpha$
except the special case;
if $(n_1, m_1) = (0, \delta)$,
then $s = 2$ and
$D = T_{1} = T_{s-1} = \delta$ for any $0 < l < \alpha$.

\begin{lemma}\label{lem on D for Case 2}
The inequality $D > \delta$ holds
for any $l > 0$ if $\delta \leq d$ and 
for any $0 < l < \alpha$ if $\delta > d$ and $(n_1, m_1) \neq (0, \delta)$.
On the other hand,
the equality $D = \delta$ holds
for $l = \alpha$ if $\delta > d$ and
for any $0 < l \leq \alpha$ if $\delta > d$ and $(n_1, m_1) = (0, \delta)$.
\end{lemma}

\subsubsection{Proof of Theorem \ref{main thm on inv wedges for Case 2}}

Theorem \ref{main thm on inv wedges for Case 2} 
follows from the inequality $D \geq \delta$.
We use the following for the case $D > \delta$. 

\begin{lemma}\label{lem on inv wedges for Case 2}
For any $l > 0$ and
for any small $r_1$ and $r_2$,
there exists a constant $C > 0$ such that
$|q(z,w)| \leq C |z|^{l D}$ on $U^l_{r_1, r_2}$.
\end{lemma}

\begin{proof}
Since $|w| < r_2 |z|^{l}$ on $U^l_{r_1, r_2}$
and $i+lj \geq lD$ for any $(i,j)$ such that $b_{ij} \neq 0$,
\begin{align*}
|q(z,w)| & \leq \sum |b_{ij}| |z|^{i} |w|^{j} < \sum |b_{ij}| r_2^{j} |z|^{i+lj} \\
\vspace{4mm}
& = \sum_{i+lj = lD} |b_{ij}| r_2^{j} |z|^{lD} + \sum_{i+lj > lD} |b_{ij}| r_2^{j} |z|^{i+lj} \\
& = |z|^{l D} \Big\{ \sum_{i+lj = lD} |b_{ij}| r_2^{j} + \sum_{i+lj > lD} |b_{ij}| r_2^{j} |z|^{i+lj-lD} \Big\} \\
& \leq |z|^{l D} \Big\{ \sum_{i+lj = lD} |b_{ij}| r_2^{j} + \sum_{i+lj > lD} |b_{ij}| r_2^{j} r_1^{i+lj-lD} \Big\}
\end{align*}
on $U^l_{r_1, r_2}$.
The first sum in the last expression is bounded from above
since the set of the pairs $(i,j)$ such that $i+lj = lD$ and $b_{ij} \neq 0$ is finite, and
the second sun tends to $0$ as $r_1 \to 0$
since $i+lj-lD > 0$. 
\end{proof}

The following is a precise version of Theorem \ref{main thm on inv wedges for Case 2}. 

\begin{theorem}\label{thm on inv wedges for Case 2}
The following holds for any $l > 0$ if $\delta \leq d$ and 
for any $0< l < \alpha$ if $\delta > d$:
$f(U_{r_1, r_2}^l) \subset U_{r_1, r_2}^l$
for suitable small $r_1$ and $r_2$.
More precisely,
$r_1$ and $r_2$ have to satisfy the inequality
$C r_1^{l (D - \delta)} < r_2$ or
$2 C r_1^{l (D^* - \delta)} < r_2$
for some constants $C$ and $D^*$
if $D > \delta$ or $D = \delta$,
where $D^* > \delta$.
\end{theorem}

\begin{proof} 
Let $U = U^l_{r_1, r_2}$ in the proof.
Since the inequality $|p(z)| < r_1$ is trivial,
it is enough to show the existence of $r_1$ and $r_2$ 
such that $|q(z,w)| < r_2 |p(z)|^l$ for any $(z,w)$ in $U$.  
We have the inequality $D \geq \delta$
by Lemma \ref{lem on D for Case 2}.
The proof differs whether $D > \delta$ or $D = \delta$.

We first consider the case $D > \delta$.
It follows from Lemma \ref{lem on inv wedges for Case 2} that
\[
|p(z)| > C_1 |z|^{\delta}
\text{ and }
|q(z,w)| \leq C_2 |z|^{l D}
\text{ on } U
\]
for some constants $C_1$ and $C_2$.
Hence
\[
\left| \dfrac{q(z,w)}{p(z)^l} \right|
< \dfrac{C_2 |z|^{l D}}{(C_1 |z|^{\delta})^l}
= \dfrac{C_2}{C_1^l} |z|^{l (D - \delta)}
< C r_1^{l (D - \delta)}
\text{ on } U,
\]
where $C = C_2 \cdot C_1^{-l}$.
Since $D > \delta$,
the map $f$ preserves $U$ if $C r_1^{l (D - \delta)} < r_2$.

We next consider the case $D = \delta$;
that is, the special case $(n_1, m_1) = (0, \delta)$.
Recall that $\delta > d$ and
$(N, M) = (n_1, m_1)$ for any $0 < l < \alpha$
if $(n_1, m_1) = (0, \delta)$.
Let us assume that $p(z) = z^{\delta}$ for simplicity.
Then 
\begin{align*}
\left| \dfrac{q(z,w)}{p(z)^l} \right|
& \leq \dfrac{\sum |b_{ij}| |z|^{i} |w|^{j}}{|z|^{l \delta}}
\leq \sum |b_{ij}| r_2^{j} |z|^{i+lj-l \delta} \\
& = \sum_{i+lj = l \delta} |b_{ij}| r_2^{j} + \sum_{i+lj > l \delta} |b_{ij}| r_2^{j} |z|^{i+lj-l \delta} \\
& = |b_{NM}| r_2^{M} + \sum_{i+lj > l \delta} |b_{ij}| r_2^{j} |z|^{i+lj-l \delta}
\text{ on } U.
\end{align*}
Since $M = m_1 = \delta \geq 2$,
$|b_{NM}| r_2^{M} \leq |b_{NM}| r_2^2 < r_2/2$
if $r_2 < (2|b_{NM}|)^{-1}$.
Let 
\[
D^* = \min \left\{ l^{-1} i + j \ | \ (i,j) \neq (0, \delta) \text{ and } b_{ij} \neq 0 \right\}.
\] 
Then $D^* > D = \delta$ and 
the following inequality holds on $U$ for some constant $C$:
\[
\sum_{i+lj > l \delta} |b_{ij}| r_2^{j} |z|^{i+lj-l \delta}
= \Big\{ \sum_{i+lj \geq l D^*} |b_{ij}| r_2^{j} |z|^{i+lj-l D^*} \Big\} |z|^{l (D^* - \delta)}
\]
\[
< \Big\{ \sum_{i+lj \geq l D^*} |b_{ij}| r_2^{j} r_1^{i+lj-l D^*} \Big\} r_1^{l (D^* - \delta)}
\leq C r_1^{l (D^* - \delta)}.
\]
Therefore,
if $2 C r_1^{l (D^* - \delta)} < r_2 < (2|b_{NM}|)^{-1}$,
then
\[
\left| \dfrac{q(z,w)}{p(z)^l} \right|
< \dfrac{r_2}{2} + \dfrac{r_2}{2} = r_2.
\] 
\end{proof}

\begin{remark}
As the same strategy,
we can show that 
the statement in the theorem above holds even for $l = \alpha$ 
if $\delta > d \geq 2$,
without using the fact $f \sim f_0$ on $U$ 
as in Lemmas 2.1 and \ref{detailed lemma for case 2}. 
Let us give a short proof.
As in the proof above,
\[
\left| \dfrac{q(z,w)}{p(z)^l} \right|
\leq \sum_{i + \alpha j = \alpha \delta} |b_{ij}| r_2^{j} 
+ \sum_{i + \alpha j > \alpha \delta} |b_{ij}| r_2^{j} |z|^{i + \alpha j - \alpha \delta}
\text{ on } U_{r_1, r_2}^{\alpha}.
\]
Note that 
the set of the pairs $(i,j)$ such that $i + \alpha j = \alpha \delta$ and $b_{ij} \neq 0$ is finite,
which includes $(\gamma, d)$, and
the minimum $j$ such that $i + \alpha j = \alpha \delta$ is $d$.
Since $d \geq 2$, 
the first sum is less than $r_2/2$ for small $r_2$.
Let
\begin{alignat*}{2}
D^* &= \min \left\{ {\alpha}^{-1} i + j \ | \ (i,j) \neq (\gamma, d) \text{ and } b_{ij} \neq 0 \right\}
& & \text{ if } \delta < T_{s-1},
\text{ and} \\
D^* &= \min \left\{ {\alpha}^{-1} i + j \ | \ i+ \alpha j > \alpha \delta \text{ and } b_{ij} \neq 0 \right\}
& & \text{ if } \delta = T_{s-1}.
\end{alignat*}
Then $D^* > D = \delta$ and
the second sum is less than or equal to $C r_1^{l (D^* - \delta)}$
for some constant $C$,
which is less than $r_2/2$ for small $r_1$.
\end{remark}

\begin{remark} 
As in the proof above,
if $D > \delta$,
then $r_1$ has to satisfy the inequality
\[
r_1 \leq \left( \dfrac{r_2}{C} \right)^{\frac{1}{l(D - \delta)}}.
\]
Note that 
$l(D - \delta) = N+lM-l \delta \to N$ as $l \to 0$, and
$N = n_1$ for small enough $l$ by Lemma \ref{properties of D for Case 2}.
Hence,
for small enough $l$,
it is enough to satisfy the inequality
\[
r_1 \leq \left( \dfrac{r_2}{C} \right)^{\frac{1}{n_1}}.
\]
Therefore, 
$r_1$ needs not converge to $0$ as $l \to 0$.
\end{remark}

\subsubsection{Proof of Theorem \ref{main thm on attr sets for Case 2}}

Let $V^l = \{ 0 < |z| < r, |w| \geq r |z|^{l}, |w| < r_3 \}$.
Then the union $U^l \cup V^l \cup \{ z = 0 \}$ is a neighborhood of the origin,
where $U^l = \{ |z| < r, |w| < r |z|^l \}$.

\begin{lemma}\label{lem on attr sets for Case 2}
For any $l > 0$ and
for any small $r$ and $r_3$,
there exists a constant $C > 0$ such that
$|q(z,w)| \leq C |w|^D$ on $V^l$.
\end{lemma}

\begin{proof}
Since $|z| \leq (r^{-1} |w|)^{l^{-1}}$ on $V^l$,
\begin{align*}
|q(z,w)| 
& \leq \sum |b_{ij} z^{i} w^{j}| \leq \sum |b_{ij}| r^{-l^{-1} i} |w|^{l^{-1} i + j} \\
& = \sum_{i+lj=lD} |b_{ij}| r^{-l^{-1} i} |w|^{D} + \sum_{i+lj>lD} |b_{ij}| r^{-l^{-1} i} |w|^{l^{-1} i + j} \\
& \leq |w|^{D} \Big\{ \sum_{i+lj=lD} |b_{ij}| r^{-l^{-1} i} 
+ \sum_{i+lj>lD} |b_{ij}| r^{-l^{-1} i} r_3^{(l^{-1} i + j) - D} \Big\}
\end{align*}
on $V^l$.
The first sum in the last expression is bounded from above
since the set of the pairs $(i,j)$ such that $i+lj = lD$ and $b_{ij} \neq 0$ is finite, and
the second sun tends to $0$ as $r_3 \to 0$ 
since $l^{-1} i+j-D > 0$.
\end{proof}

The proposition below follows from the inequality $D > \delta$ and the lemma above,
which implies Theorem \ref{main thm on attr sets for Case 2}.

\begin{proposition}\label{prop on attr sets for Case 2}
The following holds for any $l > 0$ if $\delta \leq d$
and for any $0 < l < \alpha$ if $\delta > d$ and $(n_1, m_1) \neq (0, \delta)$:
for any $(z,w)$ in $V^l$ with suitable small $r$ and $r_3$,
there exists an integer $n$ such that
$f^n(z,w)$ belongs to $U^l$. 
\end{proposition}

\begin{proof}
It follows from Lemma \ref{lem on attr sets for Case 2} that
\[
|p(z)| > (C_1 |z|)^{\delta}
\text{ and }
|q(z,w)| < (C_2 |w|)^D \text{ on } V^l
\]
for some constants $C_1$ and $C_2$.
We derive a contradiction by assuming that
$f^n(z,w)$ belongs to $V^l$ for all $n > 0$.
Let $f^n(z,w) = (z_n, w_n)$.
Then $r |z_n|^l \leq |w_n|$ by the assumption.
Since
\[
|z_n| > C_1^{\frac{\delta^n-1}{\delta-1}} |z|^{\delta^n} > C_1^{\frac{\delta^n}{\delta-1}} |z|^{\delta^n}
\text{ and }
|w_n| < C_2^{\frac{D^n-1}{D-1}} |w|^{D^n} < C_2^{\frac{D^n}{D-1}} |w|^{D^n},
\]
we have the inequality 
\[
r (C_1^{\frac{1}{\delta-1}} |z|)^{l \delta^n} < r |z_n|^l \leq |w_n| < (C_2^{\frac{1}{D-1}} |w|)^{D^n}.
\]
Recall that $D > \delta$ by Lemma \ref{lem on D for Case 2}.
By taking $r$ and $r_3$ small enough,
we may assume that 
$C_1^{\frac{1}{\delta-1}} r < 1$ and $C_2^{\frac{1}{D-1}} r_3 < 1$.
Therefore,
the convergence speed of $(C_2^{\frac{1}{D-1}} |w|)^{D^n}$ is faster 
than that of $\{ (C_1^{\frac{1}{\delta-1}} |z|)^l \}^{\delta^n}$,
which contradict the inequality above.
\end{proof}

\section{Case 2: Plurisubharmonic functions}

We continue the study of the dynamics for Case 2.
We show the existence and properties of 
plurisubharmonic functions and
prove Theorems \ref{main thm on G_z^a for Case 2} and 
\ref{main thm on G_z, G_f and G_f^a for Case 2} and
Propositions \ref{main propo on A_f^l for case 2} and 
\ref{main prop on the unboundness for case 2}
in this section.  

We first deal with the monomial skew products in Section 5.1.
Although they are not exactly in Case 2, 
the investigation in this subsection is helpful 
to understand the other cases.
We then deal with the cases $\delta < d$ and $\delta = d$
in Sections 5.2 and 5.3, respectively.
The case $T_{s-1} > \delta > d$ is dealt with in Section 5.4,
which is separated into the subcases $d \geq 1$ and $d = 0$.
We finally deal with the case $\delta = T_{s-1}$ in Section 5.5,
which is separated into the subcases $n_{s-1} > 0$ and $n_{s-1} = 0$.
One may refer Table \ref{table for case 2} 
in Section 2.1 for a comparison chart.

Theorem \ref{main thm on G_z^a for Case 2} follows from
Propositions \ref{prop on G_z^a when delta < d}, 
\ref{another prop on G_z^a when delta < d},
\ref{prop on G_z^a when delta = d},
\ref{prop on G_z^a when T > delta > d} and
\ref{prop on G_z^a when delta = T}.
Theorem \ref{main thm on G_z, G_f and G_f^a for Case 2} follows from
Propositions \ref{prop on G_z^a when delta < d},
\ref{prop on G_z when delta = d},
Corollary \ref{prop on G_f when delta = d},
Proposition \ref{prop on G_z when T > delta > d},
Corollary \ref{cor on G_f^a when T > delta > d}, 
Propositions \ref{prop on G_z when d = 0},
\ref{prop on G_z and G_f^a when delta = T and n > 0} and
\ref{prop on G_z and G_f^a when delta = T and n = 0}.
Proposition \ref{main propo on A_f^l for case 2} follows from
Corollary \ref{cor on A_f^l when T > delta > d}.
Proposition \ref{main prop on the unboundness for case 2} follows from
Corollary \ref{cor on the unboundness when T > delta > d}.

\subsection{Monomial maps}

Let $f_0 (z,w) = (z^{\delta}, z^{\gamma} w^{d})$,
$\delta \geq 2$, $\gamma \geq 0$, $d \geq 1$ and $\gamma + d \geq 2$.
Then $f_0$ has a superattracting fixed point at  the origin.
Let
\[
\alpha = \dfrac{\gamma}{\delta - d}
\]
if $\delta \neq d$.
Note that $\alpha < 0$ if $\delta < d$,
and $\alpha > 0$ if $\delta > d$.

Let $f_0^n (z,w) = (z_n, w_n)$.
Then $z_n = z^{\delta^n}$ and $w_n = z^{\gamma_n} w^{d^n}$,
where 
\[
\gamma_n = \alpha (\delta^n - d^n)  \text{ if } \delta \neq d,
\text{ and }
\gamma_n = n \gamma d^{n-1} \text{ if } \delta = d.
\]
It is clear that $G_p (z) = \log |z|$.

We first consider the case $\delta \neq d$.
If $\delta \neq d$ and $z \neq 0$,
then
\[
w_n = z^{\alpha (\delta^n - d^n)} w^{d^n}
= z^{\alpha \delta^n} \left( \dfrac{w}{z^{\alpha}} \right)^{d^n}.
\]
Hence
\[
\dfrac{w_n}{z_n^{\alpha}} = \left( \dfrac{w}{z^{\alpha}} \right)^{d^n}
\text{ and so }
G_z^{\alpha} (w) 
= \log \left| \dfrac{w}{z^{\alpha}} \right|.
\]

\begin{lemma}
If $\delta < d$, then
$G_z^{\alpha} (w) = \log \left| z^{- \alpha} w \right|$ on $\mathbb{C}^2$.
If $\delta > d$ and $\gamma > 0$, then
$G_z^{\alpha} (w) = \log \left| w/z^{\alpha} \right|$ on $\{ z \neq 0 \}$.
If $\delta \neq d$ and $\gamma = 0$, then
$G_z^{\alpha} (w) = \log |w|$ on $\mathbb{C}^2$.
\end{lemma}

We also have the following. 

\begin{lemma}
If $\delta < d$, then
$G_z (w) = \log \left| z^{- \alpha} w \right|$ on $\mathbb{C}^2$.
If $\delta > d$ and $\gamma > 0$, then
$G_z (w) = \alpha \log |z|$ on $\{ w \neq 0 \}$ and
$G_z (w) = - \infty$ on $\{ w = 0 \}$.
If $\delta > d$ and $\gamma = 0$, then
$G_z (w) = 0$ on $\{ w \neq 0 \}$ and $G_z (w) = - \infty$ on $\{ w = 0 \}$.
\end{lemma}

\begin{proof}
If $\delta < d$, then
\begin{align*}
G_z (w) 
&= \lim_{n \to \infty} \dfrac{1}{d^n} \log \left| z^{\alpha \delta^n} \left( \dfrac{w}{z^{\alpha}} \right)^{d^n} \right| 
= \lim_{n \to \infty} \dfrac{1}{d^n} \log \left| \dfrac{(z^{- \alpha} w)^{d^n}}{(z^{- \alpha})^{\delta^n}}  \right| \\
&= - \lim_{n \to \infty} (- \alpha) \left( \dfrac{\delta}{d} \right)^n \log |z| + \log \left| z^{- \alpha} w \right| 
= \log \left| z^{- \alpha} w \right|
\end{align*}
on $\{ z \neq 0 \}$.
It is clear that $G_z (w) = - \infty$ on $\{ z = 0 \}$.

If $\delta > d$ and $\gamma > 0$, then
\begin{align*}
G_z (w) 
&= \lim_{n \to \infty} \dfrac{1}{\delta^n} \log \left| z^{\alpha \delta^n} \left( \dfrac{w}{z^{\alpha}} \right)^{d^n} \right| \\
&= \alpha \log |z| + \lim_{n \to \infty} \left( \dfrac{d}{\delta} \right)^n \log \left| \dfrac{w}{z^{\alpha}} \right| 
= \alpha \log |z| 
\end{align*}
on $\{ zw \neq 0 \}$.
It is clear that $G_z (w) = - \infty$ on $\{ zw = 0 \}$.
\end{proof}

\begin{remark}
If $\delta \neq d$, then
$f_0$ is semiconjugate to the product
$(z,c) \to (z^{\delta}, c^d)$
by $\pi (z,c) = (z^{\delta - d}, z^{\gamma} c)$.
We can replace $\pi$ 
by $\pi (z,c) = (z, z^{\alpha} c)$
if $\alpha$ is integer.
\end{remark}

We next consider the case $\delta = d$.
If $\delta = d$ and $z \neq 0$,
then
\[
w_n = z^{n \gamma d^{n-1}} w^{d^n}
\text{ and so }
G_z^{\infty} (w) = \log |w|.
\]
If $\delta = d$, $\gamma > 0$ and $|z| < 1$, then 
\begin{align*}
G_z (w) 
&= \lim_{n \to \infty} \dfrac{1}{d^n} \log \left| z^{n \gamma d^{n-1}} w^{d^n} \right| \\
&= \lim_{n \to \infty} \left\{ n \dfrac{\gamma}{d} \log |z| + \log |w| \right\} 
= - \infty.
\end{align*}

\begin{lemma}
If $\delta = d$ and $\gamma > 0$, then 
$G_z^{\infty} (w) = \log |w|$ on $\{ z \neq 0 \}$
and
$G_z (w) = - \infty$ on $\{ |z| < 1 \}$.
If $\delta = d$ and $\gamma = 0$, then 
$G_z^{\infty} (w) = G_z (w) = \log |w|$ on $\mathbb{C}^2$.
\end{lemma}

These lemmas can be summarized into the following two tables.

\begin{table}[htb]
\caption{Psh functions for monomial maps when $\gamma > 0$}
\begin{tabular}{|c|c|c|c|} \hline
 & $G_z^{\alpha}$ or $G_z^{\infty}$                                     & $G_z$ & $A_0$ \\ \hline
$\delta < d$   & $\log \left| z^{- \alpha} w \right|$ on $\mathbb{C}^2$         & $\log \left| z^{- \alpha} w \right|$ on $\mathbb{C}^2$ & $\{ |z| < 1, |z^{- \alpha} w| < 1 \}$ \\  \hline
$\delta > d$   & $\log \left| w/z^{\alpha} \right|$ on $\{ z \neq 0 \}$    & $\alpha \log \left| z \right|$ on $\{ w \neq 0 \}$ & $\{ |z| < 1 \}$   \\  \hline
$\delta = d$   & $\log \left| w \right|$ on $\{ z \neq 0 \}$                         & $- \infty$ on $\{ |z| < 1 \}$ & $\{ |z| < 1 \}$  \\ \hline
\end{tabular}
\end{table}

\begin{table}[htb]
\caption{Psh functions for monomial maps when $\gamma = 0$}
\begin{tabular}{|c|c|c|c|} \hline
 & $G_z^{\alpha}$ or $G_z^{\infty}$              & $G_z$                                           & $A_0$ \\ \hline
$\delta < d$  & $\log \left| w \right|$ on $\mathbb{C}^2$   & $\log \left| w \right|$ on $\mathbb{C}^2$ & $\{ |z| < 1, |w| < 1 \}$  \\  \hline 
$\delta > d$  & $\log \left| w \right|$ on $\mathbb{C}^2$   & $0$ on $\{ w \neq 0 \}$                       & $\{ |z| < 1, |w| < 1 \}$   \\  \hline
$\delta = d$  & $\log \left| w \right|$ on $\mathbb{C}^2$   & $\log \left| w \right|$ on $\mathbb{C}^2$ & $\{ |z| < 1, |w| < 1 \}$  \\ \hline
\end{tabular}
\end{table}

\subsection{$\delta < d$}

Recall that
$\alpha < 0$ and $\mathcal{I}_f = [ l_1, \infty )$ if $\delta < d$.

We first show the existence and properties of plurisubharmonic functions 
on $A_f^l$ for any $l$ in $\mathcal{I}_f$ in Section 5.2.1,
and then evaluate the limit supremum and inferior of $G_z^{\alpha}$
toward the boundary $\partial A_f^l$ in Sections 5.2.2 and 5.2.3, respectively.

\subsubsection{Existence and basic properties} 

Because $f$ is conjugate to $f_0$ on $U^l$ 
for any $l$ in $\mathcal{I}_f$, 
the limits $G_z^{\alpha}$ and $G_z$ exist on $U^l$. 
More precisely,
as shown below,
\[
G_z^{\alpha} (w) = G_z (w) = \log \left| \phi_1 (z)^{- \alpha} \phi_2 (z,w) \right| 
= \log \left| z^{- \alpha} w \right| + o(1) \text{ on } U^l
\text{ as } r \to 0,
\]
where $\phi=(\phi_1, \phi_2)$ is the B\"{o}ttcher coordinate for $f$ on $U^l$.
It extends to a plurisubharmonic function on $A_f^l$
via the equation $G^{\alpha} \circ f = d G^{\alpha}$.
We remark that 
$\left| \phi_1^{- \alpha} \phi_2 \right|$ extends to a continuous function
from $A_f^l$ to $[0, 1)$.
Theorem \ref{main thm on attr sets for Case 2} indicates that $A_f^l = A_0 - E_z$ 
for any $l$ in $\mathcal{I}_f$. 
Consequently,
we obtain the following.

\begin{proposition} \label{prop on G_z^a when delta < d}
If $\delta < d$, then $G_z^{\alpha} = G_z$ on $A_0$.
It is plurisubharmonic and negative on $A_0 - E_z$,
pluriharmonic on $A_0 - E$ 
and $- \infty$ on $E \cap A_0$.
Moreover,
$G_z^{\alpha} (w) = \log \left| z^{- \alpha} w \right| + o(1)$
on $U^{l_1}$ as $r \to 0$.
\end{proposition}

\begin{proof}
Let us show the equality $G_z^{\alpha} = G_z$ on $A_0$.
It follows that
\begin{align*}
G_z^{\alpha} (w)
&= \lim_{n \to \infty} \dfrac{1}{d^n} \log \left| \dfrac{w_n}{z_n^{\alpha}} \right| 
= \lim_{n \to \infty} \dfrac{1}{d^n} \log |w_n|
- \lim_{n \to \infty} \dfrac{\alpha}{d^n} \log |z_n| \\
&= G_z (w) - 0 = G_z (w) 
\text{ on } U^{l_1},
\end{align*}
which extends to $A_0 - E_z$.
The equality
$G_z^{\alpha} = - \infty$ on $E_z$
follows from the definition.
On the other hand,
the equality $G_z = - \infty$ on $E_z \cap A_0$
follows from the dynamics of $f$ on the $w$-axis. 
In fact, $q(0,w)$ degenerates the $w$-axis 
to the origin if $n_1 > 0$, or
it has a superattracting fixed point at $w = 0$ with the order $m_1$ if $n_1 = 0$,
where $m_1 > d$.

The last estimate in the statement follows from the lemma below.
\end{proof}

\begin{lemma}
If $\delta < d$,
then $G_p = \log |\phi_1|$ and
$G_z^{\alpha} = \log \left| \phi_1^{- \alpha} \phi_2 \right|$ on $U^{l_1}$.
\end{lemma}

\begin{proof}
Since $f_0^n (z,w) = (z^{\delta^n}, z^{\gamma_n} w^{d^n})$
and $\gamma_n = \alpha (\delta^n - d^n)$,
\[
f_0^{-n} (Z, W) 
= \left( Z^{1/ \delta^n}, \left( \dfrac{W}{Z^{\gamma_n/ \delta^n}} \right)^{1/d^n} \right)
= \left( Z^{1/ \delta^n}, Z^{\alpha/ \delta^n} \left( \dfrac{W}{Z^{\alpha}} \right)^{1/d^n} \right).
\] 
Let $(z_n, w_n) = f^n (z,w)$.
Since $\displaystyle \phi (z,w) = \lim_{n \to \infty} f_0^{-n} \circ f^n (z,w)$,
\[
\phi (z,w) = \lim_{n \to \infty} 
\left( z_n^{1/ \delta^n}, z_n^{\alpha/ \delta^n} \left( \dfrac{w_n}{z_n^{\alpha}} \right)^{1/d^n} \right).
\]
Hence
$\displaystyle \log | \phi_1 (z) | = \lim_{n \to \infty} \delta^{-n} \log |z_n| = G_p (z)$ and
\[
\log | \phi_2 (z,w) | 
= \lim_{n \to \infty} \dfrac{\alpha}{\delta^n} \log |z_n| 
+ \lim_{n \to \infty} \dfrac{1}{d^n} \log \left| \dfrac{w_n}{z_n^{\alpha}} \right| 
= \alpha G_p (z) + G_z^{\alpha} (w). 
\]
Therefore,
$G_z^{\alpha} (w) = \log | \phi_2 (z,w) | - \alpha \log | \phi_1 (z) | = \log \left| \phi_1 (z)^{- \alpha} \phi_2 (z,w) \right|$ on $U^{l_1}$.
\end{proof}

Moreover,
$G_z^{\alpha}$ is plurisubharmonic on $A_0$.

\begin{proposition} \label{another prop on G_z^a when delta < d}
If $\delta < d$,
then $G_z^{\alpha}$ is plurisubharmonic on $A_0$.
\end{proposition}

\begin{proof}
There is a constant $C > 0$ such that
$|q(z,w)| \leq C |w|^d$ on a neighborhood of the origin,
because $d \leq j$ for any $j$ such that $b_{ij} \neq 0$.
Hence we have the inequality
\[
A |q(z,w)| \leq (A |w|)^d
\]
on the neighborhood,
where $A = |C|^{1/(d-1)}$.
Let $g_n(z,w) = d^{-n} \log |A w_n|$.
Then $g_{n+1} \leq g_n$ and so
$\{ g_n \}$ is a decreasing sequence of plurisubharmonic functions.
Therefore,
$g_n$ converges to a plurisubharmonic function $g_{\infty}$
on the neighborhood.
Since $g_n(z,w) = d^{-n} \log |w_n| + d^{-n} \log |A|$,
it follows that $g_{\infty} (z,w) = G_z^{\alpha} (w)$.
\end{proof}

\subsubsection{Asymptotic behavior of $G_z^{\alpha}$ toward $\partial A_f^l$: Limit supremum}

We omit to denote the condition $\delta < d$ 
in the statements in this and the next subsections. 

Let $S^l = \{ |z| < r, |w| = r|z|^{l} \} = \partial U^l \cap \{ |z| < r \}$ 
and $S_n^l = f^{-n} (S^l)$.

\begin{lemma}\label{main lem on asy behavior when delta < d}
It follows for any $l$ in $\mathcal{I}_f$ and for any $z$ in $A_p - E_p$ that
$S_n^l \cap \mathbb{C}_z \neq \emptyset$ for any large $n$, and
$G_z^{\alpha} |_{S_n^l} \to 0$ as $n \to \infty$. 
\end{lemma}

The proof is similar to that of Theorem 4.5 in \cite{u1}, but not the same;
we deal with any $l$ in $\mathcal{I}_f$ not only $\alpha$.
We also remark that
one can prove the lemma by a similar fashion to Theorem 5.4 in \cite{u1}
since $G_z^{\alpha} = G_z$.

\begin{proof}
By Proposition \ref{prop on G_z^a when delta < d},
there is a constant $C > 0$ such that
\[
\left| G_z^{\alpha}(w) - \log \left| z^{- \alpha} w \right| \right| < C
\text{ on } U^l.
\]
Since
$\left| z^{- \alpha} w \right| = r |z|^{l - \alpha}$ on $S^l$,
where $l - \alpha > l \geq l_1 > 0$,
\[
\left| \ G_{p^n(z)}^{\alpha}(Q_z^n(w)) - \log r |p^n(z)|^{l - \alpha} \ \right|
\leq C \ \text{ on } S_n^l. 
\]
Since $G_{p^n(z)}^{\alpha}(Q_z^n(w)) = d^n G_z^{\alpha}(w)$,
\[
\left| \ G_z^{\alpha}(w) - d^{-n} \log r |p^n(z)|^{l - \alpha} \ \right| \leq d^{-n} C
\ \text{ on } S_n^l. 
\]
Since $\deg p = \delta < d$ and $G_p(z) \neq - \infty$ 
by the setting,
\[ 
\dfrac{1}{d^{n}} \log |p^n(z)|^{l - \alpha} 
 = \left( \dfrac{\delta}{d} \right)^n \cdot \dfrac{l - \alpha}{\delta^n} \log |p^n(z)| 
 \to 0 \cdot (l - \alpha) G_p(z) = 0
\]
and $d^{-n} C \to 0$ as $n \to \infty$.
\end{proof}

We want to show that $S_n^l$ converges to the boundary $\partial A_f^l$ 
in $A_p \times \mathbb{P}^1$ as $n \to \infty$.
Let $S_{\infty}^l$ be the set of 
the accumulation points of $S_n^l$ 
in $\overline{A_p} \times \mathbb{P}^1$ as $n \to \infty$:
\[
S_{\infty}^l
= \{ x \in \overline{A_p} \times \mathbb{P}^1 \, | \,
\exists \ x_{k} \in S_{n_k}^l \text{ such that } x_k \to x \text{ and } n_k \to \infty \text{ as } k \to \infty \}.
\]
Since the sequence $f^{-n} (U^l)$ is monotonically increasing, 
\[
S_{\infty}^l 
= \{ x \in \overline{A_p} \times \mathbb{P}^1 \, | \,
\exists \ x_n \in S_n^l \text{ such that } x_n \to x \text{ as } n \to \infty \}.
\]
To show the convergence of $S_n^l$ to $\partial A_f^l$,
we impose a non-degenerate condition.
We say that $f$ is non-degenerate
if $f$ does not degenerate any fiber to a point.
We first prepare a lemme and a remark,
assuming that $f$ is non-degenerate.

\begin{lemma}\label{lem for asy behavior}
If $f$ is non-degenerate,
then for any $l$ in $\mathcal{I}_f$ and for any $n \geq 0$, 
\begin{enumerate}
\item $f^{-n} (\partial U^l) = \partial f^{-n} (U^l)$,
\item $f^n ( f^{-n} (\partial U^l) ) \subset \partial U^l$,
and $f^n ( f^{-n} (\partial U^l) ) = \partial U^l$ if $f$ is polynomial,
\item $f^{-n} (S^l) = f^{-n} (\partial U^l) \cap \{ (z,w) \, | \, |p^n(z)| < r_1 \}$,
\item $f^n ( f^{-n} (S^l) ) \subset S^l$, 
and $f^n ( f^{-n} (S^l) ) = S^l$ if $f$ is polynomial,
\item 
$f^{-1} (\partial A_f^l) = \partial A_f^l$,
$f(\partial A_f^l) \subset \partial A_f^l$, and 
$f(\partial A_f^l) = \partial A_f^l$ if $f$ is polynomial,
\item 
$f^{-1} (S_{\infty}^l) = S_{\infty}^l$,
$f(S_{\infty}^l) \subset S_{\infty}^l$,
and $f(S_{\infty}^l) = S_{\infty}^l$ if $f$ is polynomial.
\end{enumerate}
\end{lemma}

We only give an outline of the proof of (1).
The inclusion  $f^{-n} (\partial U) \supset \partial f^{-n} (U)$
follows since $f$ is continuous.
The inverse inclusion $f^{-n} (\partial U) \subset \partial f^{-n} (U)$
follows since $f$ is open map,
which is guaranteed by the non-degenerate assumption.
The condition of $f$ being polynomial can be replaced by
$f$ being surjective on a suitable global region.
For example,
if $f(\mathbb{C}^2) = \mathbb{C}^2$, then 
the equalities in the lemma hold. 

\begin{remark}
We proved  in \cite{ueno} that 
$f \sim f_0$ on $U^l$ as $r \to 0$,
and $f(U^l) \subset U^l$ for any $l$ in $\mathcal{I}_f$.
With a slight change of the proof,
one can show for any $l$ in $\mathcal{I}_f$ that
\[
\overline{f(U^l)} - \{ 0 \} \subset U^l
\text{ and }
\overline{f(U^l)} \cap \partial U^l = \{ 0 \}.
\]
Moreover,
if $f$ is non-degenerate,
then for any $l$ in $\mathcal{I}_f$ and for any $n \geq 0$,
\[
\overline{f^{-n}(U^l)} - f^{-n} (0) \subset f^{-(n+1)}(U^l)
\text{ and }
\overline{f^{-n}(U^l)} \cap \partial f^{-(n+1)}(U^l) =  f^{-n} (0).
\]
\end{remark}

Let $\lim_{n \to \infty} f^{-n} (\partial U^l)$ be the set of 
the accumulation points of $f^{-n} (\partial U^l)$ 
in $\overline{A_p} \times \mathbb{P}^1$ as $n \to \infty$.
The following follows from the monotonicity of $f^{-n} (U^l)$
and Lemma \ref{lem for asy behavior}.

\begin{proposition}
If $f$ is non-degenerate, 
then 
$\displaystyle \lim_{n \to \infty} f^{-n} (\partial U^l) = \partial A_f^l$
for any $l$ in $\mathcal{I}_f$.
\end{proposition}

\begin{proof}
Let $U = U^l$ and 
let us denote $\lim_{n \to \infty}$ by $\lim$ in the proof
for simplicity.

We first show the inclusion 
$\lim f^{-n} (\partial U) \subset \partial A_f$.
Let $x \in \lim f^{-n} (\partial U)$.
Then we may assume that
there exist $x_n \in f^{-n} (\partial U)$
such that $x_n \to x$ as $n \to \infty$.
Lemma \ref{lem for asy behavior} implies that
$f^{-n} (\partial U) = \partial f^{-n} (U) \subset \overline{f^{-n} (U)} \subset \overline{A_f}$.
Since $\overline{A_f}$ is closed,
$x \in \overline{A_f}$.
Let us show that $x \in \partial A_f$ by contradiction. 
If $x \in A_f$,
then there exists $N$ such that $f^N (x) \in U$.
Hence $x \in f^{-N} (U)$ and so 
$x_n \in f^{-N} (U)$ for any large $n$,
which contradicts 
the fact that $\partial f^{-n} (U) \cap f^{-N} (U) = \emptyset$
for any $n > N$.

We next show the opposite direction. 
Let $x \in \partial A_f$.
Then $x \in \overline{A_f}$ and so
we may assume that
there exists $x_n \in f^{-n} (U)$ 
such that $x_n \to x$ as $n \to \infty$.
Any path connecting $x$ and $x_n$
intersects with $\partial f^{-n} (U)$. 
Let $y_n$ be the intersection of $\partial f^{-n} (U)$ and
the segment connecting $x$ and $x_n$.
Then $y_n \to x$ as $n \to \infty$ and so
$x \in \lim \partial f^{-n} (U)$,
which coincides with $\lim f^{-n} (\partial U)$
by Lemma \ref{lem for asy behavior}.
\end{proof}

Since $S^l = \partial U^l \cap \{ |z| < r \}$ and 
$\lim_{n \to \infty} f^{-n} (\partial U^l - S^l) \subset \partial A_p \times \mathbb{C}$, 
the proposition above implies the following. 

\begin{corollary} \label{cor on the boundary for case 2}
If $f$ is non-degenerate,
then $S_{\infty}^l \subset \partial A_f^l$
and $S_{\infty}^l \cap (A_p \times \mathbb{P}^1) = \partial A_f^l \cap (A_p \times \mathbb{P}^1)$
for any $l$ in $\mathcal{I}_f$.
\end{corollary}

We obtain the following from Lemma \ref{main lem on asy behavior when delta < d} 
and the corollary above.

\begin{proposition}
If $f$ is non-degenerate,
then for any $l$ in $\mathcal{I}_f$ and 
for any $(z_0, w_0)$ in $\partial A_f^l \cap \{ (A_p - E_p) \times \mathbb{P}^1 \}$,
\[
\limsup_{(z,w) \in A_f^l \to (z_0, w_0)} G_z^{\alpha} (w) = 0.
\]
\end{proposition}

Let $E_{deg}$ be the union of all the preimages of 
the set of vertical lines in $A_p \times \mathbb{C}$
that are degenerated to points by $f$:
\[
E_{deg} = \bigcup_{L} \Big\{ \bigcup_{n \geq 0} f^{-n} (L) \Big\}, 
\]
where $L$ is a fiber in $A_p \times \mathbb{C}$
that is degenerated to a point by $f$.
Let $E_{deg}^{*} = \pi_z (E_{deg})$
where $\pi_z$ is the projection to the $z$-coordinate.

\begin{remark}
If $\delta < d$, then $f$ maps $\{ w = 0 \}$ to itself.
Hence $f(E_{deg}) \subset \{ w = 0 \}$
and so $E_{deg} \subset A_f \cup E_z$.
\end{remark}

Because $f$ is non-degenerate on 
$A_p \times \mathbb{C} - E_z \cup E_{deg}$,
we obtain the following 
by similar arguments as above.

\begin{proposition} \label{main prop on asy behavior when delta < d}
For any $l$ in $\mathcal{I}_f$ and 
for any $(z_0, w_0)$ in $\partial A_f^l \cap \{ (A_p - E_p \cup E_{deg}^*) \times \mathbb{P}^1 \}$,
\[
\limsup_{(z,w) \in A_f^l \to (z_0, w_0)} G_z^{\alpha} (w) = 0.
\]
\end{proposition}

This proposition implies the following, 
which also follows from 
Theorem \ref{main thm on attr sets for Case 2}.

\begin{corollary}
It follows that
$A_f^{l} = A_f^{l_1}$
for any $l$ in $\mathcal{I}_f$.
\end{corollary}

\begin{proof}
By definition,
$A_f^{l} \subset A_f^{l_1}$ for any $l > l_1$.
It is enough to show that
\[
A_f^{l} - E_{deg} = A_f^{l_1} - E_{deg}
\]
for any $l > l_1$
since it implies that $A_f^{l} \cap E_{deg} = A_f^{l_1} \cap E_{deg}$.
We derive the contradiction,
assuming that the equality above does not hold.
Let $G = G_z^{\alpha} = G_z$ and
$\partial A = \partial A_f^{l} \cap A_f^{l_1} \cap (A_p \times \mathbb{C} - E_{deg})$.
If the equality above does not hold, then $\partial A \neq \emptyset$.
Since $G$ is plurisubharmonic on $A_f^{l_1}$ and
since $A_f^l \cap E_z = \emptyset$,
it follows from the proposition above that 
$G \geq 0$ on $\partial A$. 
On the other hand,
$G \leq 0$ on $A_f^{l_1}$
since $G = \lim_{n \to \infty} d^{-n} \log |w_n|$ 
and $|w_n| \to 0$ as $n \to \infty$ on $A_f^{l_1}$.
Hence 
$G = 0$ on $\partial A$, 
which contradicts to the maximum principle for $G$. 
\end{proof}

\begin{question}
Can we remove the non-degenerate condition?
\end{question}

\begin{question}
Note that $\partial A_f \cap E_z \neq \emptyset$ and,
moreover,
$\partial A_f \cap E_z$ have an interior point in $E_z$
since $A_f = A_0 - E_z$
by Theorem \ref{main thm on attr sets for Case 2}.
Can $\partial A_f \cap \mathbb{C}_z$ have an interior point in $\mathbb{C}_z$
for some $z$ in $A_p - E_p$?
\end{question}

\subsubsection{Asymptotic behavior of $G_z^{\alpha}$ toward $\partial A_f^l$: Limit inferior} 

It follows from the assumption $m_j \geq d > \delta \geq 2$
that $f$ maps $\{ w=0 \}$ to itself.
If $f^{-1} \{ w=0 \} \supsetneq \{ w=0 \}$,
then we expect that
the limit infimum of $G_z^{\alpha}$ toward $\partial A_f^l$ is $- \infty$
since $G_z^{\alpha} = - \infty$ on $E_w - E_z$.
Actually,
we obtain the following
with the assumption that $q$ is polynomial,
which guarantees the surjectivity of the function.

\begin{proposition}\label{prop of liminf when delta < d}
If $q$ is polynomial and 
$f^{-1} \{ w=0 \} \supsetneq \{ w=0 \}$,
then
for any $l$ in $\mathcal{I}_f$ and
for any $z$ in $A_p - E_p \cup E_{deg}^*$,
\[
\liminf_{w \in A_f^l \cap \mathbb{C}_z \to \partial A_f^l \cap \mathbb{P}_z^1} G_z^{\alpha} (w) = - \infty,
\]
where $\pi_z$ is the projection to the $z$-coordinate
and $\mathbb{P}_z^1 = \{ z \} \times \mathbb{P}^1$.
\end{proposition}

\begin{proof}
Fix any $l$ in $\mathcal{I}_f$ and any $z$ in $A_p - E_p \cup E_{deg}^*$. 
It is enough to show that 
$E_w \cap \mathbb{C}_z$ accumulates to $\partial A_f^l \cap \mathbb{P}_z^1$
since $G_z^{\alpha} = - \infty$ on $E_w - E_z$.

Let $f^n (z,w) = (p^n (z),Q_z^n (w))$.
In other words,
let
\[
Q_z^n (w) = q_{z_{n-1}} \circ \cdots \circ q_{z_1} \circ q_z (w),
\]
where $z_j = p^j (z)$.
Hence $Q_z^n$ is the composition of n holomorphic functions on $w$,
which are not constant since $z \notin E_{deg}^*$.

We first check that the number of $(Q_z^n)^{-1} (0)$ tends to infinity as $n \to \infty$.
Let $q(z,w) = w^d h(z,w)$ and
$H = \{ z \in \mathbb{C} \, | \, \exists w \neq 0 \text{ such that } h(z,w)=0 \}$.
Then $H = \{ z \in \mathbb{C} \, | \, q_z^{-1} (0) \supsetneq \{ 0 \} \}$ and
it is enough to show that $H$ is the complement of only finitely many points.
Because $h$ is polynomial,
$\{ z \in \mathbb{C} \, | \, \exists w \text{ such that } h(z,w)=0 \} = \mathbb{C}$.
On the other hand,
$\{ z \in \mathbb{C} \, | \, h(z,0)=0 \}$ is a finite set.
If not,
it follows from the identity theorem that $h(z,0) \equiv 0$, 
which contradicts the minimality of $d$.
Therefore,
$H$ is the complement of only finitely many points.

We next show that  
$(Q_z^n)^{-1} (0) - \{ 0 \}$ converges to $\partial A_f^l \cap \mathbb{P}_z^1$
as $n \to \infty$,
which completes the proof.
Let $w_n$ be a point in $(Q_z^n)^{-1} (0) - \{ 0 \}$ and
$\{ w_{n_k} \}$ a subsequence of $\{ w_n \}$ 
that converges to a point $w_{\infty}$. 
Assuming that $w_{\infty}$ belongs to $A_f^l \cap \mathbb{C}_z$,
we derive the contradiction.
Since $(z,w_{\infty})$ belongs to $A_f^l$,
there is an integer $N$ such that
$f^N$ maps a neighborhood of $(z,w_{\infty})$ to $U^l$.
Because $f \sim f_0$ on $U^l$, 
it follows for each sufficient large $k$ that 
$Q_{z_N}^n (Q_z^{N} (w_{n_k})) = Q_z^{n+N} (w_{n_k}) \neq 0$
for any $n \geq 0$.
Hence $Q_z^{n_k} (w_{n_k}) = Q_z^{n+N} (w_{n_k}) \neq 0$ for $n = n_k - N$,
which contradicts the definition of $w_{n_k}$. 
\end{proof}

\begin{question}
It is known that
if $p$ is polynomial and $p^{-1} (0) \supsetneq \{ 0 \}$,
then $p^{-n} (0) - \{ 0 \}$ converges to $J_p$ densely.
Does the same or similar thing hold for skew products?
More precisely,
if $f$ is polynomial and $f^{-1} \{ w=0 \} \supsetneq \{ w=0 \}$,
then does $[ f^{-n} \{ w=0 \} - \{ w=0 \} ] \cap (A_p \times \mathbb{C})$ 
converge to $\partial A_f \cap (A_p \times \mathbb{P}^1)$ densely?
If this holds,
then we can revise the proposition above as follows: 
for any $l$ in $\mathcal{I}_f$ and
for any $(z_0, w_0)$ in $\partial A_f^l \cap \{ (A_p - E_p \cup E_{deg}^*) \times \mathbb{P}^1 \}$,
\[
\liminf_{(z,w) \in A_f^l \to (z_0, w_0)} G_z^{\alpha} (w) = - \infty.
\]
\end{question}

\begin{question}
Can we remove the condition that the map is polynomial
in the previous proposition and question?
\end{question}

\subsection{$\delta = d$}

Recall that
$\mathcal{I}_f =[ l_1, \infty )$ if $\delta = d$.

We first show the existence and properties of plurisubharmonic functions 
on $A_f^l$ for any $l$ in $\mathcal{I}_f$.
Because $f$ is conjugate to $f_0$ on $U^l$ 
for any $l$ in $\mathcal{I}_f$,
the limit $G_z^{\infty}$ exists on $U^l$. 
More precisely,
as shown below,
\[
G_z^{\infty} (w) 
= \log \left| \phi_2 (z,w) \right| 
= \log \left| w \right| + o(1)
\text{ on } U^l \text{ as } r \to 0.
\]
It extends to $A_f^l$ 
via the equation $G_z^{\infty} \circ f^n = d^n G_z^{\infty} + n \gamma d^{n-1} G_p$,
and $A_f^l = A_0 - E_z$ for any $l$ in $\mathcal{I}_f$
by Theorem \ref{main thm on attr sets for Case 2}.
Consequently,
we obtain the following.

\begin{proposition} \label{prop on G_z^a when delta = d}
If $\delta = d$, then 
$G_z^{\infty}$ is plurisubharmonic on $A_0 - E_z$.
It is pluriharmonic on $A_0 - E$, 
$- \infty$ on $E_w - E_z$,
$\infty$ on $E_z - E_w$ and
not defined on $E_z \cap E_w$.
Moreover,
$G_z^{\infty} (w) = \log \left| w \right| + o(1)$
on $U^{l_1}$ as $r \to 0$.
\end{proposition}

The last estimate in the proposition follows from the lemma below.

\begin{lemma}
If $\delta = d$, then 
$G_z^{\infty} = \log \left| \phi_2 \right|$ on $U^{l_1}$. 
\end{lemma}

\begin{proof}
Since $\gamma_n = n \gamma d^{n-1}$,
\[
\phi_2 (z,w)
= \lim_{n \to \infty} \left( \dfrac{w_n}{z_n^{\gamma_n/d^n}} \right)^{1/d^n}
= \lim_{n \to \infty} \left( \dfrac{w_n}{z_n^{n \gamma/d}} \right)^{1/d^n}
\]
and so
\[
\log | \phi_2 (z,w) | 
= \lim_{n \to \infty} \dfrac{1}{d^n} \log \left| \dfrac{w_n}{z_n^{n \gamma/ d}} \right|
= G_z^{\infty} (w) \text{ on } U^{l_1}.
\]
\end{proof}

The limits $G_z$ and $G_f$ also exist on $U^{l_1}$,
which extend to $A_0$.

\begin{proposition} \label{prop on G_z when delta = d}
If $\delta = d$, then $G_z = - \infty$ on $A_0$.
\end{proposition}

\begin{proof}
It follows that
\begin{align*}
G_z^{\infty} (w)
&= \lim_{n \to \infty} \dfrac{1}{d^n} \log \left| \dfrac{w_n}{z_n^{n \gamma/ d}} \right|
= \lim_{n \to \infty} \dfrac{1}{d^n} \log |w_n|
- \lim_{n \to \infty} \dfrac{n \gamma}{d} \cdot \dfrac{1}{d^n} \log |z_n| \\
&= G_z (w) - \infty \cdot G_p (z) = G_z (w) + \infty.
\end{align*}
Since $G_z^{\infty} < \infty$ on $A_0 - E_z$,
$G_z = - \infty$ on $A_0 - E_z$.
The equality $G_z = - \infty$ on $E_z \cap A_0$ follows from 
the dynamics on the $w$-axis 
as the same as the proof of Proposition \ref{prop on G_z^a when delta < d}.
\end{proof}

\begin{corollary} \label{prop on G_f when delta = d}
If $\delta = d$, then $G_f = G_p$ on $A_0$.
\end{corollary}

We next investigate the asymptotic behavior of $G_z^{\infty}$
as $(z,w)$ in $A_f$ tends to the boundary 
$\partial A_f^l \cap (A_p \times \mathbb{P}^1)$.
Since the arguments are similar to the previous case $\delta < d$,
we omit the proofs except the following lemma.

\begin{lemma}
Let $\delta = d$.
It follows for any $l$ in $\mathcal{I}_f$ and
for  any $z$ in $A_p - E_p$ that
$S_n^l \cap \mathbb{C}_z \neq \emptyset$ for any large $n$, and
$G_z^{\infty} |_{S_n} \to + \infty$ as $n \to \infty$.
\end{lemma}

\begin{proof}
By Proposition \ref{prop on G_z^a when delta = d},
there is a constant $C > 0$ such that
\[
\left| G_z^{\infty} (w) - \log |w| \right| < C
\text{ on } U^l.
\]
Since $|w| = r|z|^{l}$ on $S_n^l$,
\[
\left| G_{p^n(z)}^{\infty} (Q_z^n(w)) - \log r |p^n(z)|^l \right| < C
\text{ on } S_n^l. 
\]
Since $G_z^{\infty} \circ f^n = d^n G_z^{\infty} + n \gamma d^{n-1} G_p$,
\[
\left| G_z^{\infty} (w) + n \dfrac{\gamma}{d} G_p(z) - \dfrac{\log r}{d^{n}} - \dfrac{l}{d^n} \log |p^n(z)| \right|
\leq \dfrac{C}{d^{n}}
\text{ on } S_n^l.
\]
Since $\deg p = \delta = d$ and $G_p(z) \neq - \infty$ 
by the setting,
\[
- n \dfrac{\gamma}{d} G_p(z) + \dfrac{l}{d^n} \log |p^n(z)|
 \to + \infty - l G_p(z) = + \infty
\]
as $n \to \infty$.
\end{proof}

This lemma and 
Corollary \ref{cor on the boundary for case 2}
induce the following.

\begin{proposition} \label{main prop on asy behavior when delta = d}
Let $\delta = d$.
For any $l$ in $\mathcal{I}_f$ and 
for any $(z_0, w_0)$ in $\partial A_f^l \cap \{ (A_p - E_p \cup E_{deg}^*) \times \mathbb{P}^1 \}$,
\[
\limsup_{(z,w) \in A_f^l \to (z_0, w_0)} G_z^{\infty} (w) = \infty.
\]
\end{proposition}

This proposition implies the following, 
which also follows from 
Theorem \ref{main thm on attr sets for Case 2}.

\begin{corollary}
If $\delta = d$, then 
$A_f^{l} = A_f^{l_1}$ for any $l$ in $\mathcal{I}_f$. 
\end{corollary}

Note that  $f$ maps $\{ w=0 \}$ to itself
since $m_j \geq d = \delta \geq 2$.

\begin{proposition}
Let $\delta = d$.
If $q$ is polynomial and 
$f^{-1} \{ w=0 \} \supsetneq \{ w=0 \}$,
then
for any $l$ in $\mathcal{I}_f$ and
for any $z$ in $A_p - E_p \cup E_{deg}^*$,
\[
\liminf_{w \in A_f^l \cap \mathbb{C}_z \to \partial A_f^l \cap \mathbb{P}_z^1} G_z^{\infty} (w) = - \infty.
\]
\end{proposition}

\subsection{$T_{s-1} > \delta > d$}

Recall that
$\alpha > 0$ and $\mathcal{I}_f =[ l_1, \alpha ]$
if $T_{s-1} > \delta > d$.

The dynamics differs whether $d \geq 1$ or $d = 0$;
the B\"{o}ttcher coordinate exists only if $d \geq 1$.
We deal with the case $d \geq 1$ in Section 5.4.1 and
the case $d = 0$ in Section 5.4.2,
respectively.

\subsubsection{$d \geq 1$}

We first show the existence and properties of plurisubharmonic functions 
on $A_f^l$ for any $l$ in $\mathcal{I}_f$.
Because $f$ is conjugate to $f_0$ on $U^l$ 
for any $l$ in $\mathcal{I}_f$ if $d \geq 2$, 
and for any $l$ in $\mathcal{I}_f - \{ \alpha \}$ if $d = 1$, 
the limit $G_z^{\alpha}$ exists on $U^l$. 
More precisely,
\[
G_z^{\alpha} (w) 
= \log \left| \dfrac{\phi_2 (z,w)}{\phi_1 (z)^{\alpha}} \right| 
= \log \left| \dfrac{w}{z^{\alpha}} \right| + o(1)
\text{ on } U^l \text{ as } r \to 0.
\]
It extends to $A_f^l$,  
and $A_f^l = A_0 - E_z$ for any $l$ in $\mathcal{I}_f - \{ \alpha \}$ 
by Theorem \ref{main thm on attr sets for Case 2}.
Consequently,
we obtain the following.

\begin{proposition} \label{prop on G_z^a when T > delta > d}
If $T_{s-1} > \delta > d \geq 1$, then 
$G_z^{\alpha}$ is plurisubharmonic on $A_0 - E_z$.
It is pluriharmonic on $A_0 - E$, 
$- \infty$ on $E_w - E_z$,
$\infty$ on $E_z - E_w$ and
not defined on $E_z \cap E_w$.
Moreover,
$G_z^{\alpha} (w) = \log \left| w/z^{\alpha} \right| + o(1)$
on $U_r^{l_1}$ as $r \to 0$.
\end{proposition}

The limit $G_z$ also exists on $U^{l_1}$.
More precisely,
\[
G_z (w) = \alpha G_p(z) \text{ on } U^{l_1} - \{ w = 0 \}
\]
and $G_z (w) = - \infty$ on $\{ w = 0 \}$,
which extends to $A_0$.

\begin{proposition} \label{prop on G_z when T > delta > d}
If $T_{s-1} > \delta > d \geq 1$, then 
$G_z = \alpha G_p$ on $A_0 - E_w$
and $- \infty$ on $E_w$.
It is plurisubharmonic on $A_0 - (E_w - E_z)$,
pluriharmonic on $A_0 - E$ 
and $- \infty$ on $E \cap A_0$.
\end{proposition}

\begin{proof}
It follows that
\begin{align*}
G_z (w) - \alpha G_p (z)
&= \lim_{n \to \infty} \dfrac{1}{\delta^n} \log |w_n|
- \lim_{n \to \infty} \dfrac{\alpha}{\delta^n} \log |z_n|
= \lim_{n \to \infty} \dfrac{1}{\delta^n} \log \left| \dfrac{w_n}{z_n^{\alpha}} \right| \\
&= \lim_{n \to \infty} \left( \dfrac{d}{\delta} \right)^n \cdot \dfrac{1}{d^n} \log \left| \dfrac{w_n}{z_n^{\alpha}} \right|
= 0 \cdot G_z^{\alpha} (w)
= 0
\end{align*}
on $U^{l_1} - \{ w = 0 \}$,
which extends to $A_0 - E$.
The equality $G_z = - \infty$ on $E_w$ follows from the definition, and
the equality $G_z = - \infty$ on $E_z \cap A_0$ follows from 
the dynamics on the $w$-axis 
as the same as the proof of Proposition \ref{prop on G_z^a when delta < d}.
\end{proof}

\begin{corollary}  \label{cor on G_f^a when T > delta > d}
If $T_{s-1} > \delta > d \geq 1$, then 
$G_f^{\alpha} = \alpha G_p$ on $A_0$.
\end{corollary}

We next investigate the asymptotic behavior of $G_z^{\alpha}$
as $(z,w)$ in $A_f^l$ tends to the boundary 
$\partial A_f^l \cap (A_p \times \mathbb{P}^1)$.
Since the arguments are similar to the case $\delta < d$, 
we omit the proofs except the following lemma.
We also omit to denote the assumption $T_{s-1} > \delta > d \geq 1$
for all the statements below.

\begin{lemma}
Fix any $z$ in $A_p - E_p$.
Then $S_n^l \cap \mathbb{C}_z \neq \emptyset$ for any large $n$, and
\[
G_z^{\alpha} |_{S_n^l} 
\begin{cases} 
\ \to \infty \text{ as } n \to \infty & \text{ if } l_1 \leq l < \alpha, \\
\ \to 0 \text{ as } n \to \infty & \text{ if } l = \alpha \text{ and } d \geq 2, \\
\ \text{is bounded as } n \to \infty & \text{ if } l = \alpha \text{ and } d = 1.
\end{cases} 
\]
\end{lemma}

\begin{proof}
The proof is similar to that of Lemma \ref{main lem on asy behavior when delta < d}
except the assumptions of $\delta$ and $d$, and thus $\alpha$.
By Proposition \ref{prop on G_z^a when T > delta > d},
there is a constant $C > 0$ such that
\[
\left| G_z^{\alpha} (w) - \log \left| \dfrac{w}{z^{\alpha}} \right| \right| < C
\text{ on } U^l.
\]
Since 
$\left| z^{- \alpha} w \right| = r |z|^{l - \alpha}$ on $S^l$,
where $l - \alpha \leq 0$,
\[
\left| \ G_{p^n(z)}^{\alpha}(Q_z^n(w)) - \log r |p^n(z)|^{l - \alpha} \ \right|
\leq C \ \text{ on } S_n^l.
\]
Since $G_{p^n(z)}^{\alpha}(Q_z^n(w)) = d^n G_z^{\alpha}(w)$,
\[
\left| \ G_z^{\alpha}(w) - d^{-n} \log r |p^n(z)|^{l - \alpha} \ \right| \leq d^{-n} C
\ \text{ on } S_n^l.
\]
If $l < \alpha$ and $d \geq 2$, 
then
\[
\dfrac{1}{d^{n}} \log |p^n(z)|^{l - \alpha} 
 = \left( \dfrac{\delta}{d} \right)^n \cdot \dfrac{l - \alpha}{\delta^n} \log |p^n(z)|
 \to \infty \cdot (l - \alpha) G_p(z) = \infty
\]
as $n \to \infty$
since $\deg p = \delta > d$ and $G_p(z) \neq - \infty$.
If $l < \alpha$ and $d = 1$, 
then
\[
\log |p^n(z)|^{l - \alpha} = (l - \alpha) \log |p^n(z)|
\to (l - \alpha) \cdot (- \infty) = \infty 
\]
as $n \to \infty$.
On the other hand,
if $l = \alpha$ and $d \geq 2$,
then
\[
d^{-n} \log r|p^n(z)|^{l - \alpha} = d^{-n} \log r \to 0
\]
and $d^{-n} C \to 0$ as $n \to \infty$.
If $l = \alpha$ and $d = 1$, then
\[ 
\left| \ G_z^{\alpha} (w) - \log r \ \right| \leq C \ \text{ on } S_n^l. 
\]
\end{proof}

This lemma and 
Corollary \ref{cor on the boundary for case 2}
induce the following.

\begin{proposition} \label{main prop on asy behavior when T > delta > d}
For any $(z_0, w_0)$ in $\partial A_f^l \cap \{ (A_p - E_p \cup E_{deg}^*) \times \mathbb{P}^1 \}$,
\begin{equation*}
\limsup_{(z,w) \in A_f^l \to (z_0, w_0)} G_z^{\alpha} (w) = 
\begin{cases} 
\ \infty & \text{ if } l_1 \leq l < \alpha, \\
\ 0 & \text{ if } l = \alpha \text{ and } d \geq 2, 
\end{cases}
\end{equation*}
and $G_z^{\alpha}$ is bounded as $(z,w)$ in $A_f^l$ tends to $(z_0, w_0)$
if $l = \alpha$ and $d = 1$.
\end{proposition}

\begin{corollary} \label{cor on A_f^l when T > delta > d}
It follows that 
$A_f^{\alpha} \subsetneq A_f^{l} = A_f^{l_1}$ for any $l$ in $\mathcal{I}_f - \{ \alpha \}$. 
\end{corollary}

Note that  $f$ maps $\{ w=0 \}$ to itself
since $m_j \geq d \geq 1$.

\begin{proposition}
If $q$ is polynomial and 
$f^{-1} \{ w=0 \} \supsetneq \{ w=0 \}$,
then
for any $l$ in $\mathcal{I}_f$ and 
for any $z$ in $A_p - E_p \cup E_{deg}^*$,
\[
\liminf_{w \in A_f^l \cap \mathbb{C}_z \to \partial A_f^l \cap \mathbb{P}_z^1} G_z^{\alpha} (w) = - \infty.
\]
\end{proposition}

Let $E_w^{\infty}$ be the set of the accumulation points of 
$f^{-n} \{ w=0 \} - \{ w=0 \}$
in $\mathbb{C} \times \mathbb{P}^1$ as $n \to \infty$.
The proof of Proposition \ref{prop of liminf when delta < d}
implies the following.

\begin{proposition}
If $q$ is polynomial and 
$f^{-1} \{ w=0 \} \supsetneq \{ w=0 \}$,
then
\[
E_w^{\infty} \cap \{ (A_p - E_p \cup E_{deg}^*) \times \mathbb{P}^1 \}
 \subset \bigcap_{l_1 \leq l \leq \alpha} \partial A_f^l.
\]
\end{proposition}

Consequently,
$E_w^{\infty}$ is unbounded if $f$ satisfies the assumptions
since $A_f^{\alpha} \subsetneq A_f^{l_1}$.
Since $E_w^{\infty} \cap (A_p \times \mathbb{P}^1)$ is
included in the closure of $A_f^l$ in $\mathbb{C} \times \mathbb{P}^1$,
we obtain the following.

\begin{corollary} \label{cor on the unboundness when T > delta > d}
If $q$ is polynomial and 
$f^{-1} \{ w=0 \} \supsetneq \{ w=0 \}$,
then $A_f^l$ is unbounded for any $l$ in $\mathcal{I}_f$. 
\end{corollary}

\subsubsection{$d = 0$}

If $d=0$, then $f$ can not be conjugate to $f_0$ on any open set,
since $f_0$ is not dominant,
where $f_0 (z,w) = (z^{\delta}, z^{\gamma})$.
However,
the same statements as Lemma \ref{detailed lemma for case 2} hold 
for any $l$ in $\mathcal{I}_f - \{ \alpha \}$,
and we can show that 
$G_z = \alpha G_p$ on $A_0$.

\begin{lemma}
Let $T_{s-1} > \delta > d = 0$.
The following hold for any $l$ in $\mathcal{I}_f - \{ \alpha \}$: 
$f \sim f_0$ on $U_{r_1, r_2}^l$ as $r_1$ and $r_2 \to 0$,
and $f(U_{r_1, r_2}^l) \subset U_{r_1, r_2}^l$ 
for suitable small $r_1$ and $r_2$.
\end{lemma}

The first statement follows from the same argument as in \cite{ueno}, 
and it holds even for $l = \alpha$.
The second statement follows from Theorem \ref{main thm on inv wedges for Case 2}. 

\begin{remark}[Illustration of the second statement in terms of blow-ups] 
Let $\tilde{f}$ be a formal blow-up of $f$ by $\pi (z,c) = (z, z^l c)$.
Then 
$\tilde{f} (z,c) = \left( z^{\delta}, \ z^{\tilde{\gamma}} \left\{ 1 + o(1) \right\} \right)$,
where $\tilde{\gamma} = \gamma - l \delta \geq 0$
since $d = 0$.
If $l \in \mathcal{I}_f - \{ \alpha \}$,
then $\tilde{\gamma} > 0$ and so
the origin is a superattracting fixed point of $\tilde{f}$.
In particular,
$\tilde{f}$ preserves a neighborhood of the origin,
and so $f$ preserves $U_{r_1, r_2}^l$. 
On the other hand,
if $l = \alpha$, then $\tilde{\gamma} = 0$ and so
the origin is neither superattracting nor a fixed point of $\tilde{f}$.
\end{remark}

\begin{proposition} \label{prop on G_z when d = 0}
If $T_{s-1} > \delta > d = 0$, then
$G_z = \alpha G_p$ on $A_0$.
\end{proposition}

\begin{proof}
Fix any $l$ in $\mathcal{I}_f - \{ \alpha \}$. 
Let $U = U_{r_1, r_2}^l$ and
let $(z_n, w_n) = f^n (z,w)$ for any $(z,w)$ in $U$.
The second statement in the lemma above implies that 
$(z_n, w_n)$ also belongs to $U$.
The first statement in the lemma implies that 
\[
C_1 |(z_n)^{\gamma}| < |w_n| < C_2 |(z_n)^{\gamma}|
\]
for some constants $C_1 < 1$ and $C_2 > 1$.
Hence $\delta^{-n} \log |w_n|$ is bounded by
\[
\dfrac{1}{\delta^n} \log C_j + \dfrac{\gamma}{\delta} \cdot \dfrac{1}{\delta^{n-1}} \log |z_n|
\]
from below and above for $j=1$ and $j=2$,
respectively.
Therefore,
$G_z = \alpha G_p$ on $U$ and so on $A_f^l$.
Thanks to Theorem \ref{main thm on attr sets for Case 2},
the equality holds on $A_0$.
\end{proof}

\subsection{$\delta = T_{s-1}$} 

Recall that 
$\delta > d$ and $\mathcal{I}_f = \{ l_1 \} = \{ \alpha \}$
if $\delta = T_{s-1}$.

Because $f$ is conjugate to 
$f_0$ on $U^{\alpha}$ if $d \geq 2$,
the limit $G_z^{\alpha}$ exists on $U^{\alpha}$,
which extends to 
$A_f^{\alpha}$.

\begin{proposition} \label{prop on G_z^a when delta = T}
If $T_{s-1} = \delta > d \geq 2$, then 
$G_z^{\alpha}$ is plurisubharmonic on $A_f^{\alpha}$,
pluriharmonic on $A_f^{\alpha} - E_w$ and 
$- \infty$ on $E_w$.
Moreover,
$G_z^{\alpha} (w) = \log \left| w/z^{\alpha} \right| + o(1)$
on $U_r^{\alpha}$ as $r \to 0$.
\end{proposition}

The asymptotic behavior of $G_z^{\alpha}$ toward $\partial A_f^{\alpha}$ is as follows.

\begin{proposition} \label{main prop on asy behavior when delta = T foe case 2}
If $T_{s-1} = \delta > d \geq 2$, then 
for any $(z_0, w_0)$ in $\partial A_f^{\alpha} \cap \{ (A_p - E_p \cup E_{deg}^*) \times \mathbb{P}^1 \}$,
\[
\limsup_{(z,w) \in A_f^{\alpha} \to (z_0, w_0)} G_z^{\alpha} (w) = 0.
\]
Moreover,
if $q$ is polynomial and 
$f^{-1} \{ w=0 \} \supsetneq \{ w=0 \}$,
then
for any $z$ in $A_p - E_p \cup E_{deg}^*$,
\[
\liminf_{w \in A_f^{\alpha} \cap \mathbb{C}_z \to \partial A_f^{\alpha} \cap \mathbb{P}_z^1} G_z^{\alpha} (w) = - \infty.
\]
\end{proposition}


The dynamics of $f$ differs whether $n_{s-1} > 0$ or $n_{s-1} = 0$.
If $n_{s-1} > 0$,
then Theorem \ref{main thm on attr sets for Case 2}
implies the following inequality and equalities on $A_0$.

\begin{proposition} \label{prop on G_z and G_f^a when delta = T and n > 0}
Let $\delta = T_{s-1}$.
If $n_{s-1} > 0$, then 
$G_z^* \leq \alpha G_p$ on $A_0$ and
$G_z = \alpha G_p$ on $A_0 - A_f^{\alpha}$.
In particular,
$G_f^{\alpha} = \alpha G_p$ on $A_0$.
Moreover,
if $d \geq 2$, then
$G_z = \alpha G_p$ on $A_0 - E_w$
and $- \infty$ on $E_w$.
\end{proposition}

On the other hand,
if $n_{s-1} = 0$, then
we have the same inequality and 
almost the same equalities only on $A_f^{\alpha}$.

\begin{proposition} \label{prop on G_z and G_f^a when delta = T and n = 0}
Let $\delta = T_{s-1}$.
If $n_{s-1} = 0$, then 
$G_z^* \leq \alpha G_p$ on $A_f^{\alpha}$.
In particular,
$G_f^{\alpha} = \alpha G_p$ on $A_f^{\alpha}$.
Moreover,
if $d \geq 2$, then
$G_z = \alpha G_p$ on $A_f^{\alpha} - E_w$
and $- \infty$ on $E_w$.
\end{proposition}

\begin{remark}
If $\delta = T_{s-1}$, then 
$f$ has the other dominant term $(n_{s-1}, m_{s-1})$
than $(n_{s}, m_{s})$,
which belongs to Case 3 or Case 4.
Moreover,
if $n_{s-1} = 0$, then $s = 2$ and
$(n_1, m_1) = (n_{s-1}, m_{s-1}) = (0, \delta)$,
which belongs to Case 3.
Therefore,
we can apply the results in Case 3 or Case 4. 
For example,
variants of Corollary \ref{cor of G_f^a for Case 3} and
Corollary \ref{cor of G_f^a for Case 4}
imply that $G_f^{\alpha} = \alpha G_p$ on $A_0$ if $n_{s-1} > 0$ and
that $G_f^{\alpha}$ is plurisubharmonic and continuous
on $A_0 - E_{deg}$ if $n_{s-1} = 0$.
\end{remark}


\section{Case 3}

We deal with Case 3 in this section.
Let $s > 1$, 
\[
T_1 \leq \delta
\text{ and }
(\gamma, d) = (n_1, m_1).
\]
Note that $\gamma \geq 0$ and 
$\delta \geq T_1 \geq d \geq 1$ by the setting.

Although we observe that the blow-up $\tilde{f}$ has the required properties
in Section 6.1,
there are no results on invariant wedges and
the unions of all the preimages of wedges for Case 3.
We exhibit results on the existence and properties of plurisubharmonic functions
and on the asymptotic behavior of the functions toward the boundaries of the unions
in Section 6.2,
which imply all the main results stated in Section 2.2.
Although we omit to prove the results on the existence and properties
since they are similar to Case 2,
we prove the results on the asymptotic behavior 
in Section 6.3.
One may refer Table \ref{table for case 3} 
in Section 2.2 for a comparison chart of the results.

Two examples are given in Section 6.4.
For any polynomial $h$,
we can construct
a polynomial skew product that is
semiconjugate to a polynomial product
of the form $(z^{\delta}, h(w))$.
Dynamics of such a map is rather easy to understand.
More precisely,
all the limits of the map are expresses by the counterparts of $h$ and,
in particular,
the boundary of of the set where $G_z^{\alpha, +}$ is not pluriharmonic 
is expresses by the Julia set of $h$.


\subsection{Observation of Blow-ups}

Let us observe that the blow-up $\tilde{f}$ is superattracting at the origin and
degenerates the $t$-axis for a suitable weight $l$.
If $l$ belongs to $\mathcal{I}_f$,
then $\tilde{q} (t,w) \sim t^{\gamma} w^{\tilde{d}}$
on a neighborhood of the origin
and so
\[
\tilde{f} (t,w) = 
\left( t^{\delta - l^{-1} \gamma} w^{l^{-1} ( \delta - \tilde{d} )} \{ 1 + \zeta (t,w) \},
\ t^{\gamma} w^{\tilde{d}} \{ 1 + \eta (t,w) \} \right),
\]
where
$\tilde{d} = l^{-1} \gamma + d$.
Moreover,
if $\delta > T_1$ and $l > \alpha$,
then $\delta - \tilde{d} > 0$ and so
we have the following. 

\begin{proposition}\label{}
If $\delta > T_1$,
then $\tilde{f}$ is superattracting at the origin and
degenerates the $t$-axis 
for any $l$ in $\mathcal{I}_f - \{ \alpha \}$.
\end{proposition}

Although $\tilde{f}$ degenerates the $t$-axis,
we do not know whether $A_f^l$ contains $A_0 - E_w$.

\subsection{Plurisubharmonic functions}

Let us first review $\mathcal{I}_f$ and related quantities,
which depends on whether $\delta > T_1$ and 
whether $\gamma > 0$; 
see the table below for a summary.
If $\delta > T_1$ and $\gamma > 0$, 
then $\delta > d$, $\alpha > 0$ and $\mathcal{I}_f = [ \alpha, l_2 ]$. 
If $\delta > T_1$ and $\gamma = 0$, 
then $\delta > d$, $\alpha = 0$ and $\mathcal{I}_f = ( 0, l_2 ]$.
Moreover, $d \geq 2$
because of the assumption for $f$ being superattracting at the origin.
If $\delta = T_1$ and $\gamma > 0$,
then $\delta > d$, $\alpha > 0$ and $\mathcal{I}_f = \{ \alpha \} = \{ l_2 \}$.
If $\delta = T_1$ and $\gamma = 0$,
then $\delta = d$, $\alpha = l_2$ and $\mathcal{I}_f = (0, l_2]$;
recall that we redefined $\alpha$ to be $l_2$
in Section 2.2,
whereas $\gamma/(\delta - d)$ is not defined for this case.

\begin{table}[htb]
\caption{Informations on $\mathcal{I}_f$ and related quantities for Case 3} 
\begin{tabular}{|c|c|c|} \hline
   & $\gamma > 0$ & $\gamma = 0$   \\ \hline
$\delta > T_1$ & $\mathcal{I}_f = [ \alpha, l_2 ]$ & $\mathcal{I}_f = ( 0, l_2 ]$   \\ 
              & $\delta > d \geq 1$ and $\alpha > 0$ & $\delta > d \geq 2$ and $\alpha = 0$  \\ \hline
$\delta = T_1$ & $\mathcal{I}_f = \{ \alpha \} = \{ l_2 \}$ & $\mathcal{I}_f = ( 0, l_2 ]$   \\ 
  & $\delta > d \geq 1$ and $\alpha > 0$ & $\delta = d \geq 2$ and $\alpha = l_2 > 0$   \\ \hline
\end{tabular}
\end{table}

In particular,
the dynamics differs more or less
whether $\gamma > 0$ or $\gamma = 0$.
If $\gamma = 0$,
then $f_0$ is just a product.
More precisely,
if $\gamma = 0$ and $\delta > d$,
then $\alpha = 0$ and so
the properties of $G_z^{\alpha}$ are different from the other cases.
On the other hand,
if $\gamma = 0$ and $\delta = d$,
then $f_0 (z,w) = (z^{\delta}, w^{\delta})$ and so
the properties of $G_z$ and the magnitude relation of $A_f^{l}$
for any $l$ in $\mathcal{I}_f$ are different from the other cases.

Recall that $f$ is conjugate to the monomial map $f_0$ on $U^l$
for any $l$ in $\mathcal{I}_f$ if $d \geq 2$
and for any $l$ in $\mathcal{I}_f- \{ \alpha \}$ 
if $d = 1$ and $\delta > T_1$ 
as stated in Lemma \ref{detailed lemma for case 3}.
Therefore,
$G_z^{\alpha}$ exists on $U^l - \{ z = 0 \}$
for any $l$ as above
if $\gamma > 0$ or $\delta = d$.
More precisely,
\[
G_z^{\alpha} (w) 
= \log \left| \dfrac{\phi_2}{\phi_1^{\alpha}} \right| 
= \log \left| \dfrac{w}{z^{\alpha}} \right| + o(1)
\text{ on } U^{l} - \{ z = 0 \} \text{ as } r \to 0,
\]
which extends to a pluriharmonic function on $A_f^{l} - E_z$.
On the other hand,
if $\gamma = 0$ and $\delta > d$,
then $\alpha = 0$ and
$G_z^{\alpha}$ exists on $U^l$ for any $l$ in $\mathcal{I}_f$.
More precisely,
\[
G_z^{\alpha} (w) 
= \log \left| \phi_2 \right| 
= \log \left| w \right| + o(1)
\text{ on } U^{l} \text{ as } r \to 0,
\]
which extends to a pluriharmonic function on $A_f^{l}$.

\begin{proposition} \label{Case 3: existence of G_z^alpha}
The following hold for any $l$ in $\mathcal{I}_f$ if $d \geq 2$
and for any $l$ in $\mathcal{I}_f- \{ \alpha \}$ if $d = 1$ and $\delta > T_1$:
If $\gamma > 0$ or $\delta = d$,
then $\alpha > 0$ and $G_z^{\alpha}$ is pluriharmonic on $A_f^{l} - E_z$. 
It is $\infty$ on $E_z - E_w$ and
not defined on $E_z \cap E_w$.
Moreover,
$G_z^{\alpha} (w) = \log \left| w/z^{\alpha} \right| + o(1)$
on $U^{l} - \{ z = 0 \}$ as $r \to 0$.
If $\gamma = 0$ and $\delta > d$,
then $\alpha = 0$ and $G_z^{\alpha}$ is pluriharmonic on $A_f^{l}$. 
Moreover,
$G_z^{\alpha} (w) = \log \left| w \right| + o(1)$
on $U^{l}$ as $r \to 0$.
\end{proposition}

This proposition implies Proposition \ref{main result for case 3: G_z^a}.
Recall that $A_f = A_f^{l_2}$ and that
$A_f^{l_2}$ is the largest wedge among $A_f^l$ for any $l$ in $\mathcal{I}_f$.

Moreover,
we have the following propositions 
on the asymptotic behavior of $G_z^{\alpha}$ toward $\partial A_f^l$,
which are proved in the next subsection.

\begin{proposition} \label{Case 3: asy of G_z^alpha on boundary^out}
Let $d \geq 2$ or $\delta > T_1$.
Then for any $l$ in the proposition above and
for any $(z_0, w_0)$ in 
$\{ \partial A_f^l \cap \partial A_0 \} \cap \{ (A_p - E_p \cup E_{deg}^*) \times \mathbb{P}^1 \}$,
\[
\limsup_{(z,w) \in A_f^l \to (z_0, w_0)} G_z^{\alpha} (w) =
\begin{cases}
\infty & \text{if } \gamma > 0 \, (\text{and } \delta > d), \\
0 & \text{if } \gamma = 0 \text{ and } \delta > d, \\
- \alpha G_p (z) & \text{if } \gamma = 0 \text{ and } \delta = d.
\end{cases}
\]
The same convergence to $0$ holds for any $z_0$ in $A_p - E_{deg}^*$
if $\gamma = 0$ and $\delta > d$.
\end{proposition}

\begin{proposition} \label{Case 3: asy of G_z^alpha on boundary^in}
Let $\delta > d$. 
For any $(z_0, w_0)$ in $\partial A_f^l \cap A_0 - E_z \cup E_{deg}$,
\[
\lim_{(z,w) \in A_f^l \to (z_0, w_0)} G_z^{\alpha} (w) = 
\begin{cases} 
\ 0 & \text{ if } \gamma > 0, d \geq 2 \text{ and } l = \alpha, \\
\ - \infty & \text{ if } \alpha < l \leq l_2,
\end{cases}
\]
and $G_z^{\alpha}$ is bounded as $(z,w)$ in $A_f^{\alpha}$ tends to $(z_0, w_0)$
if $\gamma > 0$, $d = 1$ and $\delta > T_1$. 
The same convergence to $- \infty$ holds 
for any $(z_0, w_0)$ in $\partial A_f^l \cap A_0 - E_{deg}$
if $\gamma = 0$. 

On the other hand, 
if $\delta = d$, then
for any $l$ in $\mathcal{I}_f$ and 
for any $(z_0, w_0)$ in $\partial A_f^l \cap A_0 - E_z \cup E_{deg}$,
\[
\lim_{(z,w) \in A_f^l \to (z_0, w_0)} G_z^{\alpha}  (w) = (l - \alpha) G_p(z),
\]
where $\mathcal{I}_f = (0, l_2]$ and $\alpha = l_2$.
\end{proposition}

As a consequence of Propositions 
\ref{Case 3: existence of G_z^alpha}, 
\ref{Case 3: asy of G_z^alpha on boundary^out} and 
\ref{Case 3: asy of G_z^alpha on boundary^in},
we obtain Theorem \ref{main thm on G_z^a,+ for case 3} 
and Proposition \ref{main result for case 3: A_f^l}.

The limit $G_z$ also exists on $U$ 
if $d \geq 2$ or if $d = 1$ and $\delta > T_1$, 
where $U = U^{l_2}$.
More precisely,
\begin{eqnarray*}
&& G_z (w) = \alpha G_p (z) \text{ on } U
\text{ if } \delta > d, 
\text{ and} \\
&& G_z (w) = \log \left| w \right| + o(1)
\text{ on } U^{} \text{ as } r \to 0
\text{ if } \delta = d, 
\end{eqnarray*}
which extends to a pluriharmonic function on $A_f^{}$.

\begin{proposition}  \label{Case 3: existence of G_z}
If $d \geq 2$ or $\delta > T_1$, 
then $G_z$ is pluriharmonic on $A_f^{}$. 
Moreover,
$G_z = \alpha G_p$ on $A_f^{}$
if $\delta > d$, 
and 
$G_z (w) = \log \left| w \right| + o(1)$
on $U^{}$ as $r \to 0$
if $\delta = d$. 
\end{proposition}

We remark that
if $\gamma = 0$ and $\delta > d$,
then $G_z = 0$ on $A_f$. 

By arguments similar to the proof of 
Proposition \ref{Case 3: asy of G_z^alpha on boundary^in},
we obtain the following.


\begin{proposition} \label{Case 3: asy of G_z on boundary^in}
If $\delta = d$, then 
for any $l$ in $\mathcal{I}_f$ and
for any $(z_0, w_0)$ in $\partial A_f^l \cap A_0 - E_{deg}$,
\[
\lim_{(z,w) \in A_f^l \to (z_0, w_0)} G_z (w) = l G_p (z).
\]
\end{proposition}

As a consequence of Propositions \ref{Case 3: existence of G_z} 
and \ref{Case 3: asy of G_z on boundary^in} 
and the remark below,
we obtain the following,
which,
together with Proposition \ref{Case 3: existence of G_z},
implies Theorem \ref{main thm on G_z and G_f^a for case 3}. 

\begin{corollary}\label{cor of G_f^a for Case 3}
If $\delta > d$, then
$G_f^{\alpha} = \alpha G_p$ on $A_0$,
which is plurisubharmonic on $A_0$ and
pluriharmonic on $A_0 - E_z$. 
If $\delta = d$, then
$G_f^{\alpha} > \alpha G_p$ on $A_f^{\alpha}$ and
$G_f^{\alpha} = \alpha G_p$ on $A_0 - A_f^{\alpha}$,
which is plurisubharmonic and continuous on $A_0 - E_{deg}$ and
pluriharmonic on $A_0 - \partial A_f^{\alpha} \cup E_{deg}$.
\end{corollary}

\begin{remark} \label{remark on G_f^a for case 3}
If $d = 1$ and $\delta = T_1$, then $m_2 = 0$,
where $(n_2, m_2)$ is the other dominant term of $q$.
Hence $f$ is also in Case 2, and
Proposition \ref{prop on G_z when d = 0} 
implies that
$G_f^{\alpha} = \alpha G_p$ on $A_0$.
\end{remark}

We end this subsection with a proposition that relates to the discontinuity of $G_z$.

\begin{proposition} 
If $\delta > T_1$, then 
$G_z^* < \alpha G_p$ on $A_0 - A_f^{}$.
\end{proposition}

\begin{proof}
For any $l$ in $\mathcal{I}_f - \{ \alpha \}$
and for any $(z,w)$ in $A_0 - A_f^l$,
it follows 
that $|w_n| \leq r^{-l} |z_n|^l$ for any $n \geq 1$,
where $(z_n,w_n) = f^n(z,w)$.
Hence $G_z^* (w) \leq l G_p (z) < \alpha G_p (z)$ on $A_0 - A_f^l$
for any $l$ in $\mathcal{I}_f - \{ \alpha \}$.
In particular,
$G_z^* (w) < \alpha G_p (z)$ on $A_0 - A_f^{l_2}$.
\end{proof}

Therefore,
even if the limit $G_z$ exists on $A_0$,
it is not continuous on $\partial A_f^{} \cap A_0$.

\subsection{Asymptotic behavior of $G_z^{\alpha}$ toward $\partial A_f^l$}

In this subsection
we prove Propositions \ref{Case 3: asy of G_z^alpha on boundary^out} and 
\ref{Case 3: asy of G_z^alpha on boundary^in}
by arguments similar to Case 2.

We separate the boundary $\partial U^l$
into two surfaces $S^{out}$ and $S^{in}$:
let 
\[
S^{out} = \{ |z|^l < r^l |w|, |w| = r \} 
\text{ and } 
S^{in} = \{ |z|^l = r^l |w|, |w| < r \}. 
\]
Note that $S^{out} = \{ |z| < r^{1+1/l}, |w| = r \}$
and $S^{out} \cup S^{in} = \partial U^l \cap \{ |z| < r^{1+1/l} \}$.
Let $S^{out}_n = f^{-n} (S^{out})$ 
and $S^{in}_n = f^{-n} (S^{in})$. 

\begin{lemma} \label{Case 3: asy of G_z^alpha on s_n^out}
Let $d \geq 2$ or $\delta > T_1$.
Then for any $l$ in Proposition \ref{Case 3: existence of G_z^alpha} and
for any $z$ in $A_p - E_p$, 
\[
G_z^{\alpha} |_{S^{out}_n} \to 
\begin{cases}
\infty \text{ as } n \to \infty & \text{if } \gamma > 0 \, (\text{and } \delta > d), \\ 
0 \text{ as } n \to \infty & \text{if } \gamma = 0 \text{ and } \delta > d, \\
- \alpha G_p \text{ as } n \to \infty & \text{if } \gamma = 0 \text{ and } \delta = d.
\end{cases}
\]  
The same convergence to $0$ holds for any $z$ in $A_p$ if $\gamma = 0$ and $\delta > d$.
\end{lemma}

\begin{proof}
By Proposition \ref{Case 3: existence of G_z^alpha},
there is a constant $C > 0$ such that
\[
\left| G_z^{\alpha} (w) - \log \left| \dfrac{w}{z^{\alpha}} \right| \right| < C
\text{ on } U - \{ z = 0 \}.
\]
Since $|w| = r$ on $S^{out}$,
$\left| G_z^{\alpha} (w) - \log r |z|^{- \alpha} \right| \leq C$ on $S^{out} - \{ z = 0 \}$
and so
\[
\left| \ G_{p^n(z)}^{\alpha}(Q_z^n(w)) - \log r |p^n(z)|^{- \alpha} \ \right|
\leq C \ \text{ on } S_n^{out} - f^{-n} \{ z = 0 \}.
\]
Since $G_{p^n(z)}^{\alpha}(Q_z^n(w)) = d^n G_z^{\alpha}(w)$,
\[
\left| \ G_z^{\alpha}(w) - d^{-n} \log r |p^n(z)|^{- \alpha} \ \right| \leq d^{-n} C
\ \text{ on } S_n^{out} - f^{-n} \{ z = 0 \}.
\]
If $\gamma > 0$, then $\alpha > 0$ since  $\delta > d$, and
\[
\dfrac{- \alpha}{d^{n}} \log |p^n(z)| 
 = \left( \dfrac{\delta}{d} \right)^n \cdot \dfrac{- \alpha}{\delta^n} \log |p^n(z)|
 \to \infty \cdot (- \alpha) G_p(z) = \infty
\]
as $n \to \infty$.
If $\gamma = 0$ and $\delta = d$, then $\alpha > 0$ and
\[
\dfrac{- \alpha}{d^{n}} \log |p^n(z)|
 = \dfrac{- \alpha}{\delta^n} \log |p^n(z)|
 \to - \alpha G_p(z)
\] 
as $n \to \infty$.

If $\gamma = 0$ and $\delta > d$, then $\alpha = 0$ and
there is a constant $C > 0$ such that
\[
\left| G_z^{\alpha} (w) - \log \left| w \right| \right| < C
\text{ on } U 
\]
by Proposition \ref{Case 3: existence of G_z^alpha}.
The same arguments as above imply that
\[
\left| \ G_z^{\alpha}(w) - d^{-n} \log r \ \right| \leq d^{-n} C
\ \text{ on } S_n^{out}.
\]
Clearly,
$d^{-n} \log r \to 0$ and $d^{-n} C \to 0$ as $n \to \infty$
since $d \geq 2$.
\end{proof}

\begin{lemma}  \label{Case 3: asy of G_z^alpha on s_n^in}
Fix any $z$ in $A_p - E_p$.
Let $\delta > d$.
Then 
\[
G_z^{\alpha} |_{S^{in}_n} 
\begin{cases} 
\ \to 0 \text{ as } n \to \infty & \text{ if } \gamma > 0, d \geq 2 \text{ and } l = \alpha, \\
\ \text{is bounded } \text{ as } n \to \infty & \text{ if } \gamma > 0, d = 1, \delta > T_1 \text{ and } l = \alpha, \\
\ \to - \infty \text{ as } n \to \infty & \text{ if } \alpha < l \leq l_2.
\end{cases} 
\]
On the other hand,
if $\delta = d$, then
\[
G_z^{\alpha} |_{S^{in}_n} \to (l - \alpha) G_p
\text{ as } n \to \infty
\]
for any $0 < l \leq l_2$, where $\alpha = l_2$.

The same convergence to $- \infty$ holds 
for any $z$ in $A_p$ 
if $\gamma = 0$ and $\delta > d$.
\end{lemma}

\begin{proof}
By Proposition \ref{Case 3: existence of G_z^alpha},
there is a constant $C > 0$ such that
\[
\left| G_z^{\alpha} (w) - \log \left| \dfrac{w}{z^{\alpha}} \right| \right| < C
\text{ on } U^l - \{ z = 0 \}.  
\]
Since $\left| z^{- \alpha} w \right| = r |z|^{l - \alpha}$ on $S^{in}$,
$\left| G_z^{\alpha} (w) - \log r |z|^{l - \alpha} \right| \leq C$ on $S^{in} - \{ 0 \}$
and so
\[
\left| \ G_{p^n(z)}^{\alpha}(Q_z^n(w)) - \log r |p^n(z)|^{l - \alpha} \ \right|
\leq C \ \text{ on } S_n^{in} - f^{-n} (0).
\]
Since $G_{p^n(z)}^{\alpha}(Q_z^n(w)) = d^n G_z^{\alpha}(w)$,
\[
\left| \ G_z^{\alpha}(w) - d^{-n} \log r |p^n(z)|^{l - \alpha} \ \right| \leq d^{-n} C
\ \text{ on } S_n^{in} - f^{-n} (0).
\]

Let  $\delta > d$.
If $\gamma > 0$, $d \geq 2$ and $l = \alpha$,
then
\[
d^{-n} \log r|p^n(z)|^{l - \alpha} = d^{-n} \log r \to 0
\]
and $d^{-n} C \to 0$ 
as $n \to \infty$.
If $\gamma > 0$,  $d = 1$ and $l = \alpha$, then
\[ 
\left| \ G_z^{\alpha} (w) - \log r \ \right| \leq C \ \text{ on } S_n^{in} - f^{-n} (0). 
\]
If $\alpha < l \leq l_2$, 
then
\[
\dfrac{l - \alpha}{d^{n}} \log |p^n(z)|^{} 
 = \left( \dfrac{\delta}{d} \right)^n \cdot \dfrac{l - \alpha}{\delta^n} \log |p^n(z)|
 \to + \infty \cdot (l - \alpha) G_p(z) = - \infty
\]
as $n \to \infty$.

On the other hand,
if $\delta = d$ and $0 < l \leq l_2$, then
\[
\dfrac{1}{d^{n}} \log |p^n(z)|^{l - \alpha} 
 = \dfrac{l - \alpha}{\delta^n} \log |p^n(z)|
 \to (l - \alpha) G_p(z)
\]
as $n \to \infty$, where $\alpha = l_2$.

If $\gamma = 0$ and $\delta > d$, then $\alpha = 0$ and
the same inequalities hold not only on $S_n^{in} - f^{-n} (0)$ but also on $S_n^{in}$.
Hence 
$G_z^{\alpha} |_{S^{in}_n} \fallingdotseq d^{-n} \log r |p^n(z)|^l \to - \infty$ as $n \to \infty$
for any $z$ in $A_p$ and for any $0 < l \leq l_2$. 
\end{proof}

Let $\lim_{n \to \infty} f^{-n} (\partial U^l)$ be the set of 
the accumulation points of $f^{-n} (\partial U^l)$ 
in $\overline{A_p} \times \mathbb{P}^1$ as $n \to \infty$,
as the same as Case 2.
Then the following equality follows from the same argument as Case 2.

\begin{proposition} \label{Case 3: lim of f^-n (partial U)}
If $f$ is non-degenerate, then
$\displaystyle \lim_{n \to \infty} f^{-n} (\partial U^l) = \partial A_f^l$
for any $l$ in Proposition \ref{Case 3: existence of G_z^alpha}.
\end{proposition}


Let $S^{out}_{\infty}$ and $S^{in}_{\infty}$ 
be the sets of 
the accumulation points of $S_n^{out}$ and $S_n^{in}$
in $\overline{A_p} \times \mathbb{P}^1$ as $n \to \infty$,
respectively.
We investigate that whether the equalities 
$S^{out}_{\infty} = \partial A_f^l \cap \partial A_0$ and 
$S^{in}_{\infty} = \partial A_f^l \cap A_0$ hold or not. 
Let $r_l = r^{1+1/l}$,
\[
W^l = \{ |z| < r_l, |w| \leq r^{-l} |z|^l \}
\text{ and }
U^+ = U^l \cup W^l = \{ |z| < r_l, |w| < r \}.
\]
Let $W_n^l = f^{-n} (W^l)$ and $U^+_n = f^{-n} (U^+)$.
Note that $U^+_n$ is a neighborhood of the origin,
$\partial U^+_n \cap \{ |z| < r_l \} = S^{out}_n$
and $\partial W_n^l \cap \{ |z| < r_l \} = S^{in}_n$.

Applying the same argument as Case 2 for $U^+_n$, 
we obtain the following.

\begin{proposition} \label{Case 3: lim of S^out_n}
If $f$ is non-degenerate, then
$S^{out}_{\infty} \subset \partial A_f^l \cap \partial A_0$ and
$S^{out}_{\infty}  \cap (A_p \times \mathbb{P}^1) 
= \partial A_f^l \cap \partial A_0 \cap (A_p \times \mathbb{P}^1)$
for any $l$ in Proposition \ref{Case 3: existence of G_z^alpha}.
\end{proposition}

Combining this proposition with  
Proposition \ref{Case 3: asy of G_z^alpha on s_n^out},
we obtain Proposition \ref{Case 3: asy of G_z^alpha on boundary^out}
in the previous subsection.


We next show Proposition \ref{Case 3: asy of G_z^alpha on boundary^in}.
Note that 
$W_n^l \cap U^+$ is a decreasing sequence and
\[
(\cap_{n \geq 1} W_n^l) \cap U^+ 
= U^+ - \cup_{n \geq 1} f^{-n} (U^l)
= U^+ - A_f^l,
\]
which is not empty and closed in $U^+$.
Since $S^{in}_n = \partial W_n^l \cap \{ |z| < r_l \}$,
it follows that $S^{in}_{\infty} \cap U^+ \subset \partial A_f^l \cap U^+$. 
Moreover, 
we have the equality since $\partial A_f^l \cap U^+ \subset W^l$ and
$f$ is non-degenerate on $W^l$.


\begin{proposition}
It follows that
$S^{in}_{\infty} \cap U^+ = \partial A_f^l \cap U^+$
for any $l$ in Proposition \ref{Case 3: existence of G_z^alpha}.
\end{proposition}

In addition,
Propositions \ref{Case 3: lim of f^-n (partial U)} and 
\ref{Case 3: lim of S^out_n} imply the following global statement.

\begin{proposition}
If $f$ is non-degenerate, then
$S^{in}_{\infty} \supset \partial A_f^l \cap A_0$
for any $l$ in Proposition \ref{Case 3: existence of G_z^alpha}.
\end{proposition}

Combining this proposition with  
Proposition \ref{Case 3: asy of G_z^alpha on s_n^in},
we obtain Proposition \ref{Case 3: asy of G_z^alpha on boundary^in}.
Note that $G_z^{\alpha}$ is monotonic increasing on
$A_f^l \cap W_n^l \cap \mathbb{C}_z$
for any $z$ in $A_p - E_p \cup E_{deg}^*$ and for any $n \geq 0$
because of the maximum principle.
Hence we can take the limit instead of the limit inferior in 
Proposition \ref{Case 3: asy of G_z^alpha on boundary^in}.

\begin{question}
We display some questions.
\begin{enumerate}
\item 
In general,
$S^{out}_n \cap S^{in}_n = \emptyset$ and
$S^{out}_{\infty} \cap S^{in}_{\infty} = \emptyset$?
\item 
If $f$ is non-degenerate, then
$S^{in}_{\infty} \cap (A_p \times \mathbb{P}^1)
= \partial A_f \cap A_0 \cap (A_p \times \mathbb{P}^1)$?
\item 
Can we remove the non-degenerate condition
for all the statements and the question above?
\end{enumerate}
\end{question}

\subsection{Examples}

For any polynomial $h$ of degree $d$ and
for any integers $\delta$ and $\alpha$
such that $\delta \geq d$,
we can construct a polynomial skew product $f$
that is semiconjugate to the polynomial product $g(z,w) = (z^{\delta}, h(w))$
and such that the $\alpha$ determined by $f$ coincides with 
the given integer $\alpha$.  
The dynamics of such a map is rather easy to study.
For example,
the B\"{o}ttcher coordinates and plurisubharmonic functions for $f$
are expressed by those for $h$.
In particular,
the set where $G_z^{\alpha,+}$ is not pluriharmonic 
is expressed by the Julia set of $h$.

We divide a family of such maps 
into two cases: the case $\delta > d$ and the case $\delta = d$.
If $\delta > d$,
then $f$ degenerate the $w$-axis and $A_0$ is unbounded.
If $\delta = d$,
then $f$ is non-degenerate and $A_0$ is bounded.
Let $h$ be a monic polynomial of degree $d \geq 2$.
More precisely, let
$h(w) = w^d + b_{d-1} w^{d-1} + \cdots + b_m w^m$,
where $b_m \neq 0$.

\begin{example}[degenerate type] \label{example of deg type}
Let $\alpha$ be an integer, $\delta > d$ and 
\[
f(z,w) = \left( z^{\delta}, z^{\alpha \delta} h \left( \frac{w}{z^{\alpha}} \right) \right).
\]
Then $f$ is semiconjugate to 
$g(z,w) = (z^{\delta}, h(w))$
by $\pi (z,w) = (z, z^{\alpha} w)$: $f \circ \pi = \pi \circ g$.

Since 
$q(z,w) = z^{\alpha(\delta - d)} w^d + b_{d-1} z^{\alpha(\delta - (d-1))} w^{d-1} + \cdots + b_m z^{\alpha(\delta - m)} w^m$,
the Newton polygon of $q$ has two vertices:
$(n_1, m_1) = (\alpha(\delta - d), d)$ and $(n_2, m_2) = (\alpha(\delta - m), m)$.
The vertices belong to Cases 3 and 2, 
respectively. 
Let $U_1 = \{ |z|^{\alpha} < r^{\alpha} |w|, |w| < r \}$ and $U_2 =\{ |z| < r, |w| < r|z|^{\alpha} \}$ for small $r > 0$.
Let $A_1$ and $A_2$ be the unions of all the preimages of $U_1$ and $U_2$, respectively.
Since
\[
f^n(z,w) = \left( z^{\delta^n}, z^{\alpha \delta^n} h^n \left( \frac{w}{z^{\alpha}} \right) \right),
\]
we have the following equalities:
\[
\phi (z,w) = \left( z, z^{\alpha} \varphi_h^{\infty} \left( \dfrac{w}{z^{\alpha}} \right) \right)
\text{ on } U_1
\text{ and } 
\phi (z,w) = \left( z, z^{\alpha} \varphi_h^{0} \left( \dfrac{w}{z^{\alpha}} \right) \right)
\text{ on } U_2
\]
if $m \geq 2$,
where $\displaystyle \varphi_h^{\infty} (w) = \lim_{n \to \infty}  \sqrt[d^n]{h^n (w)}$
and $\displaystyle \varphi_h^{0} (w) = \lim_{n \to \infty}  \sqrt[m^n]{h^n (w)}$,
\[
G_z^{\alpha} (w) = G_h^{\infty} \left( \dfrac{w}{z^{\alpha}} \right) > 0
\text{ on } A_1 - \{ z = 0 \}
\text{ and } 
G_z^{\alpha} (w) = G_h^{0} \left( \dfrac{w}{z^{\alpha}} \right) < 0
\text{ on } A_2
\]
if $m \geq 2$,
where $\displaystyle G_h^{\infty} (w) = \lim_{n \to \infty} d^{-n} \log |h^n(w)|$
and $\displaystyle G_h^{0} (w) = \lim_{n \to \infty} m^{-n} \log |h^n(w)|$,
\[
G_z^{\alpha,+} (w) 
= G_h^{\infty, +} \left( \dfrac{w}{z^{\alpha}} \right)
=
\begin{cases}
G_h^{\infty} \left( \dfrac{w}{z^{\alpha}} \right) & \text{on } A_1 - \{ z = 0 \}, \\
0 & \text{on } A_0 - A_1,
\end{cases}
\]
where $\displaystyle G_h^{\infty, +} (w) = \lim_{n \to \infty} d^{-n} \log^{+} \left| h^n (w) \right|$,
\[
\partial A_1 \cap \left\{ (A_p - E_p) \times \mathbb{C} \right\}
= \bigcup_{0<|z|<1} \{ z \} \times \partial \left\{ w \, | \, G_z^{\alpha,+} (w) > 0 \right\}
\]
\[
= \bigcup_{0<|z|<1} \{ z \} \times \partial \left\{ w \, | \, G_z^{\alpha,+} (w) = 0 \right\}
= \bigcup_{0<|z|<1} \{ z \} \times z^{\alpha} J_h,
\]
where $J_h$ is the Julia set of $h$.
Note that 
$\partial A_1 \cap (A_p \times \mathbb{C}) = \partial A_2 \cap (A_p \times \mathbb{C})$
if $m \geq 2$
because the boundaries of the attracting basins of $\infty$ and $0$ for $h$
coincide.
Moreover,
\[
G_z(w) = \alpha \log |z|
\text{ on } A_1
\text{ and } 
G_z(w) = 
\begin{cases}
\alpha \log |z| & \text{ on } A_2 - E_w, \\
- \infty & \text{ on } E_w
\end{cases}
\]
if $m \geq 2$.
Therefore,
\[
G_z(w) = 
\begin{cases}
\alpha \log |z| & \text{ on } A_0 - E_w, \\
- \infty & \text{ on } E_w
\end{cases}
\]
if $m \geq 2$,
and $G_f^{\alpha} (z, w) = \alpha \log |z|$ on $A_0$,
where $A_0 = \{ |z| < 1 \}$.
\end{example}

\begin{example}[non-degenerate type] \label{example of non-deg type}
Let $\alpha$ be an integer and
\[
f(z,w) = \left( z^d, z^{\alpha d} \left( \frac{w}{z^{\alpha}} \right) \right).
\]
Then $f$ is semiconjugate to 
$g(z,w) = (z^d, h(w))$
by $\pi (z,w) = (z, z^{\alpha} w)$: $f \circ \pi = \pi \circ g$.

Since $q(z,w) = w^d + b_{d-1} z^{\alpha} w^{d-1} + \cdots + b_m z^{\alpha (d - m)} w^m$,
the Newton polygon of $q$ has two vertices:
$(n_1, m_1) = (0, d)$ and $(n_2, m_2) = (\alpha (d - m), m)$.
Since
\[
f^n(z,w) = \left( z^{d^n}, z^{\alpha d^n} h^n \left( \frac{w}{z^{\alpha}} \right) \right),
\]
the same equalities as the previous example hold
except the following:
\[
G_z(w) = \alpha \log |z| + G_h^{\infty} \left( \dfrac{w}{z^{\alpha}} \right) > \alpha \log |z|
\text{ on } A_1 - \{ z = 0 \}
\]
and so
\[
G_f^{\alpha} (z, w) =
\begin{cases}
\log |w| & \text{ on } \{ z = 0 \}, \\
\alpha \log |z| + G_h^{\infty} \left( \dfrac{w}{z^{\alpha}} \right) & \text{ on } A_1 - \{ z = 0 \}, \\
\alpha \log |z| & \text{ on } A_0 - A_1,
\end{cases}
\]
where
\[
A_0 = \{ (0, w) : |w| < 1 \} 
\cup \Big( \bigcup_{0<|z|<1} \left\{ G_h \left( \dfrac{w}{z^{\alpha}} \right) < - \alpha \log |z| \right\}  \Big).
\]
\end{example}

The same construction works even if $d = 1$.
Such maps are used in Example 7.5 in \cite{ueno}
to show that we can not remove the condition
$\delta \neq T_k$ for any $k$
in Lemma \ref{main lemma for previous reselts}
if $d = 1$.
See also Example 5.2 in \cite{u2}.

Even if $\alpha$ is rational,
one can construct similar examples,
although
it is more complicated and the desired polynomial $h$ is limited.
Proposition 3.9 in \cite{u-sym1} and
Proposition 3.1 in \cite{u-fiberwise}
give necessary and sufficient conditions for 
the non-degenerate polynomial skew product $(z^d, q(z,w))$,
where $q(z,w) = w^d + O_z(w^{d-1})$,
to be semiconjugate to $(z^d, q(1,w))$ by $\pi (z,w) = (z^r, z^s w)$
for some positive integers $r$ and $s$.
More generally,
Lemma 5.1 in \cite{u-sym2}
gives necessary and sufficient conditions for 
the rational skew product $(z^{\delta}, q(z,w))$,
where the degree of $q$ with respect to $w$ is $d$, 
to be semiconjugate to $(z^{\delta}, q(1,w))$ by $\pi (z,w) = (z^r, z^s w)$
for some integers $r$ and $s$.
Green functions of such maps are studied in \cite{u-fiberwise} and
such maps are characterized by symmetries of Julia sets
as stated in \cite{u-sym1} and \cite{u-sym2}.
One can find statements similar to the examples above
in Examples 5.2 and 5.3 in \cite{u-weight1},
Examples 4.3 and 5.2 in \cite{u-weight2},
Section 10 in \cite{u1}
and Example 4.4 in \cite{u2}.

\section{Case 4}

We deal with Case 4 in the last section.
Let $s > 2$, 
\[
T_k \leq \delta \leq T_{k-1}
\text{ and }
(\gamma, d) = (n_k, m_k)
\]
for some $2 \leq k \leq s-1$.
Note that
$\gamma > 0$ and $\delta > d \geq 1$ by the setting. 

We first exhibit several results on invariant wedges and
the unions of all the preimages of the wedges,
which imply Theorems \ref{main thm on inv wedges for Case 4}
and \ref{main thm on attr sets for Case 4},
and describe the properties of the blow-up $\tilde{f}_1$ 
in Section 7.1.
We then exhibit a few results on the existence and properties of
plurisubharmonic functions in Section 7.2.
We finally investigate the asymptotic behavior of $G_z^{\alpha}$ 
toward the boundaries of the unions
in Section 7.3,
which is separated into three subcases.
One may refer Table \ref{table for case 4} 
in Section 2.3 for a comparison chart of the results.


\subsection{Invariant wedges and Attracting sets}

We exhibit several results 
that imply 
Theorems \ref{main thm on inv wedges for Case 4}
and \ref{main thm on attr sets for Case 4}
in Section 7.1.1.
Since the statements and proofs are almost the same as the case $\delta > d$ in Case 2,
we omit the proofs.

Note that
the dynamics of the first blow-up $\tilde{f}_1$ is 
almost the same as the case $\delta > d$ in Case 2.

\begin{proposition}\label{}
The blow-up $\tilde{f}_1$ is superattracting at the origin for any $0 < l < \alpha$,
even for $l = \alpha$ if $\delta > T_k$ and $d \geq 2$ 
or if $\delta = T_k$ and $m_{k+1} \geq 2$, and 
degenerates the $c$-axis for any $0 < l < \alpha$ unless $(n_1, m_1) = (0, \delta)$.
\end{proposition}

Therefore,
one may expect that 
Theorems \ref{main thm on inv wedges for Case 4}
and \ref{main thm on attr sets for Case 4} hold.
This proposition corresponds to 
Proposition \ref{properties of blow-up for Case 2} in Case 2.
See Section 7.1.2 for a detail,
in which 
the shape of the Newton polygon of $\tilde{q}_1$ 
and properties of $\tilde{f}_1$ are summarized.


\subsubsection{Exhibition of results} 

We can apply the same arguments for Case 2
to the dynamics of $f$ on $U_{r_1, r_2}^{l, +}$ and $U^{l, +}$,
where
$U_{r_1, r_2}^{l, +} = \{ |z| < r_1, |w| < r_2 |z|^{l} \}$ and 
$U^{l, +} = \{ |z| < r, |w| < r |z|^{l} \}$.

The following lemma corresponds to Lemma \ref{lem on D for Case 2} in Case 2.

\begin{lemma}
If $(n_1, m_1) \neq (0, \delta)$,
then $D > \delta$ for any $0 < l < \alpha$,
and $D = \delta$ for $l = \alpha$.
On the other hand,
if $(n_1, m_1) = (0, \delta)$,
then $D = \delta$ for any $0 < l \leq \alpha$.
\end{lemma}

The following lemma and theorem correspond to 
Lemma \ref{lem on inv wedges for Case 2} and
Theorem \ref{thm on inv wedges for Case 2},
respectively,
which imply Theorem \ref{main thm on inv wedges for Case 4}.

\begin{lemma} 
For any $l > 0$ and
for any small $r_1$ and $r_2$,
there exists a constant $C > 0$ such that
$|q(z,w)| \leq C |z|^{l D}$ on $U^{l, +}_{r_1, r_2}$.
\end{lemma}

\begin{theorem}\label{}
The following holds for any $0< l < \alpha$:
$f(U_{r_1, r_2}^{l, +}) \subset U_{r_1, r_2}^{l, +}$
for suitable small $r_1$ and $r_2$.
More precisely,
$r_1$ and $r_2$ have to satisfy the inequality
$C r_1^{l (D - \delta)} < r_2$ or
$2 C r_1^{l (D^* - \delta)} < r_2$
for some constants $C$ and $D^*$
if $D > \delta$ or $D = \delta$.
\end{theorem}

Let $V^l = \{ 0 < |z| < r, \, r |z|^{l} \leq |w| < r_3 \}$. 
The following lemma and proposition correspond to 
Lemma \ref{lem on attr sets for Case 2} and
Proposition \ref{prop on attr sets for Case 2},
respectively,
which imply Theorem \ref{main thm on attr sets for Case 4}.

\begin{lemma}
For any $l > 0$ and
for any small $r$ and $r_3$,
there exists a constant $C > 0$ such that
$|q(z,w)| \leq C |w|^D$ on $V^l$.
\end{lemma}

\begin{proposition}
If $(n_1, m_1) \neq (0, \delta)$, then
the following holds for any $0 < l < \alpha$:
for any $(z,w)$ in $V^l$ with suitable small $r$ and $r_3$,
there exists an integer $n$ such that
$f^n(z,w)$ belongs to $U^{l,+}$. 
\end{proposition}

\subsubsection{Shapes of Newton polygons} 

The shape of the Newton polygon of $\tilde{q}_1$ 
and properties of $\tilde{f}_1$ for Case 4
are more or less similar to those for Case 2.
More precisely,
the magnitude relation of $0$, $\tilde{n}_1$, $\tilde{n}_2$, $\cdots$, 
$\tilde{n}_k$ are the same as Case 2,
where $\tilde{n}_k = \tilde{\gamma}$.
On the other hand,
we have the extra vertices $(\tilde{n}_{k+1}, m_{k+1})$, $\cdots$, $(\tilde{n}_s, m_s)$ for Case 4,
and so 
$N(\tilde{q}_1)$ always has other vertices than $(\tilde{\gamma}, d)$. 

We only give tables of summaries of  
the shape of $N(\tilde{q}_1)$ and properties of $\tilde{f}_1$.
Recall that 
$\mathcal{I}_f^1 = [ l_1, \alpha ]$ if $T_k < \delta < T_{k-1}$, 
$\mathcal{I}_f^1 = [ l_1, \alpha )$ if $\delta = T_{k}$, and
$\mathcal{I}_f^1 = \{ l_1 \} = \{ \alpha \}$ if $\delta = T_{k-1}$.
Tables \ref{shape when T_k < delta < T_k-1}, 
\ref{shape when delta = T_k-1} and 
\ref{shape for special case when delta = T_k-1} below for Case 4 correspond to
Tables \ref{shape when d < delta < T}, 
\ref{shape when delta = T} and 
\ref{shape for special case} for Case 2, respectively.
Table \ref{shape when delta = T_k} resembles to 
Table \ref{shape when T_k < delta < T_k-1},
whereas the case $l = \alpha$ is different.
Note that $(\gamma, d) = (n_2, m_2)$ 
for Table \ref{shape for special case when delta = T_k-1}
since $(n_1, m_1) = (0, \delta)$.

\begin{table}[htb]
\caption{Shape of $N(\tilde{q}_1)$ when $T_k < \delta < T_{k-1}$} \label{shape when T_k < delta < T_k-1}
\begin{tabular}{|c|c|c|c|c|} \hline
$l$   & $0 < l < l_1$ & $l_1$ & $l_1 < l < \alpha$ & $\alpha$  \\ \hline
$\tilde{\gamma}$ & $0 < \tilde{n}_{k-1} < \tilde{\gamma}$ & $0 < \tilde{\gamma}  = \tilde{n}_{k-1} < \tilde{n}_j$ & $0 < \tilde{\gamma} < \tilde{n}_j$ & $0 = \tilde{\gamma} < \tilde{n}_j$   \\ 
              & $0 < \tilde{n}_j$ for any $j$ & for $j \neq k, k-1$ & for $j \neq k$ & for $j \neq k$  \\ \hline
$\tilde{f}_1$  & SA and Deg & SA and Deg & SA and Deg & SA if $d \geq 2$  \\ \hline
\end{tabular}
\end{table}

\begin{table}[htb]
\caption{Shape of $N(\tilde{q}_1)$ when $\delta = T_k$} \label{shape when delta = T_k}
\begin{tabular}{|c|c|c|c|c|} \hline
$l$   & $0 < l < l_1$ & $l_1$ & $l_1 < l < \alpha$ & $\alpha$  \\ \hline
$\tilde{\gamma}$ & $0 < \tilde{n}_{k-1} < \tilde{\gamma}$ & $0 < \tilde{\gamma}  = \tilde{n}_{k-1} < \tilde{n}_j$ & $0 < \tilde{\gamma} < \tilde{n}_j$ & $0 = \tilde{\gamma} = \tilde{n}_{k+1} < \tilde{n}_j$   \\ 
              & $0 < \tilde{n}_j$ for any $j$ & for $j \neq k, k-1$ & for $j \neq k$ & for $j \neq k, k+1$  \\ \hline
$\tilde{f}_1$  & SA and Deg & SA and Deg & SA and Deg & SA if $m_{k+1} \geq 2$  \\ \hline
\end{tabular}
\end{table}

\begin{table}[htb]
\caption{Shape of $N(\tilde{q}_1)$ when $\delta = T_{k-1}$ and $(n_1, m_1) \neq (0, \delta)$}  \label{shape when delta = T_k-1}
\begin{tabular}{|c|c|c|} \hline
$l$   & $0 < l < l_1 = \alpha$ & $l_1 = \alpha$   \\ \hline
$\tilde{\gamma}$ & $0 < \tilde{n}_{k-1} < \tilde{\gamma}$ & $0 = \tilde{\gamma} = \tilde{n}_{k-1} < \tilde{n}_j$   \\ 
              & $0 < \tilde{n}_j$ for any $j$ & for $j \neq k, k-1$  \\ \hline
$\tilde{f}_1$  & SA and Deg & SA if $d \geq 2$  \\ \hline
\end{tabular}
\end{table}

\begin{table}[htb]
\caption{Shape of $N(\tilde{q}_1)$ when $(n_1, m_1) = (0, \delta)$}  \label{shape for special case when delta = T_k-1}
\begin{tabular}{|c|c|c|} \hline
$l$   & $0 < l < l_1 = \alpha$ & $l_1 = \alpha$   \\ \hline
$\tilde{\gamma}$ & $0 = \tilde{n}_1 < \tilde{\gamma} < \tilde{n}_j $ & $0 = \tilde{n}_1 = \tilde{\gamma} < \tilde{n}_j$   \\ 
              & for $j \geq 3$ & for $j \geq 3$  \\ \hline
$\tilde{f}_1$  & SA & SA if $d \geq 2$  \\ \hline
\end{tabular}
\end{table}

\newpage
\subsection{Plurisubharmonic functions}


Because $f$ is conjugate to the monomial map $f_0$ on $U^{l_{(1)},l_{(2)}}$
for any $l_{(1)}$ in $\mathcal{I}_f^1$ and $l_{(2)}$ in $\mathcal{I}_f^2$ if $d \geq 2$, and
for any $l_{(1)}$ in $\mathcal{I}_f^1 - \{ \alpha \}$ and 
$l_{(2)}$ in $\mathcal{I}_f^2 - \{ \alpha - l_{(1)} \}$ if $d = 1$ and $T_k < \delta < T_{k-1}$,
the limit $G_z^{\alpha}$ exists on $U^{l_{(1)},l_{(2)}}$,
which extends to a pluriharmonic function on $A_f^{l_{(1)},l_{(2)}}$.

\begin{proposition}\label{Case 4: existence of G_z^alpha}
For any $l_{(1)}$ in $\mathcal{I}_f^1$ and $l_{(2)}$ in $\mathcal{I}_f^2$ if $d \geq 2$, and
for any $l_{(1)}$ in $\mathcal{I}_f^1 - \{ \alpha \}$ and 
$l_{(2)}$ in $\mathcal{I}_f^2 - \{ \alpha - l_{(1)} \}$ if $d = 1$ and $T_k < \delta < T_{k-1}$,
the limit $G_z^{\alpha}$ is pluriharmonic on $A_f^{l_{(1)},l_{(2)}}$. 
Moreover,
$G_z^{\alpha} (w) = \log \left| w/z^{\alpha} \right| + o(1)$
on $U^{l_{(1)},l_{(2)}}$ as $r \to 0$.
\end{proposition}

Recall that $A_f = A_f^{l_1, l_2}$ and
that $A_f^{l_1, l_2}$ is the largest wedge among $A_f^{l_{(1)},l_{(2)}}$ 
for any $l_{(1)}$ in $\mathcal{I}_f^1$ and $l_{(2)}$ in $\mathcal{I}_f^2$.
This proposition implies
Proposition \ref{main prop on G_z^a for Case 4}.


Combining Theorem \ref{main thm on attr sets for Case 4} and
Proposition \ref{main prop on G_z^a for Case 4} with
Propositions \ref{asy behavior when T_k < delta  T_k-1} 
and \ref{asy behavior when delta = T_k},
results on 
the asymptotic behavior of $G_z^{\alpha}$ 
exhibited in the next subsection, 
we obtain Theorem \ref{main thm on G_z^a,+ for case 4}. 


The limit $G_z$ also exists on $U$,
where $U = U^{l_1, l_2}$,
which extends to $A_f$.

\begin{proposition}
If $d \geq 2$ or $T_k < \delta < T_{k-1}$, 
then $G_z = \alpha G_p$ on $A_f$. 
\end{proposition}

Moreover,
it follows from Theorems \ref{main thm on inv wedges for Case 4}
and \ref{main thm on attr sets for Case 4}
that $G_z^* \leq \alpha G_p$ on $A_0$ if $(n_1, m_1) \neq (0, \delta)$.
Therefore, we have the following. 

\begin{corollary}\label{cor of G_f^a for Case 4}
If $(n_1, m_1) \neq (0, \delta)$,
then $G_f^{\alpha} = \alpha G_p$ on $A_0$.
\end{corollary}

The proposition and corollary above imply 
Theorem \ref{main thm on G_z and G_f^a for Case 4}.

\newpage
\subsection{Asymptotic behavior of $G_z^{\alpha}$ toward $\partial A_f^{l_{(1)},l_{(2)}}$}

In this subsection 
we exhibit results on the asymptotic behavior of $G_z^{\alpha}$ 
for the cases $T_k < \delta < T_{k-1}$, $\delta = T_k$ and $\delta = T_{k-1}$
in Sections 7.3.1, 7.3.2 and 7.3.3, respectively.
In particular,
Propositions \ref{asy behavior when T_k < delta  T_k-1} and
\ref{asy behavior when delta = T_k} below imply
Theorem \ref{main thm on G_z^a,+ for case 4}, and
Corollaries \ref{cor on A_f^l when T_k < delta  T_k-1} and
\ref{cor on A_f^l when delta = T_k} below imply 
Proposition \ref{prop on A_f^l for case 4}.
We omit all the proofs since they are almost the same as the previous cases.

\subsubsection{$T_k < \delta < T_{k-1}$}

Recall that
$\mathcal{I}_f = [l_1, \alpha] \times [\alpha, l_1 + l_2] - \{ (\alpha, \alpha) \}$
if $T_k < \delta < T_{k-1}$
and that
$U^{l_{(1)},l_{(2)}} = \{ |z|^{l_{(1)} + l_{(2)}} < r^{l_{(2)}} |w|, |w| < r|z|^{l_{(1)}} \}$.

We first point out that $f$ has two disjoint invariant wedges in $U^{l_1,l_2}$.

\begin{proposition}
It follows for sufficiently small $r > 0$ that
\begin{enumerate}
\item $f(U^{l_1, \alpha - l_1}) \subset U^{l_1, \alpha - l_1}$ and 
        $f(U^{\alpha, l_1 + l_2 - \alpha}) \subset U^{\alpha, l_1 + l_2 - \alpha}$,
\item $U^{l_1, \alpha - l_1} \cap U^{\alpha, l_1 + l_2 - \alpha} = \emptyset$ and
        $U^{l_1, \alpha - l_1} \cup U^{\alpha, l_1 + l_2 - \alpha} \subset U^{l_1,l_2}$.
\end{enumerate}
\end{proposition}

Similar to Case 3,
we separate the boundary $\partial U^{l_{(1)},l_{(2)}}$
into two surfaces: let
$S^{out} = \{ |z|^{l_{(1)} + l_{(2)}} < r^{l_{(2)}} |w|, |w| = r|z|^{l_{(1)}} \}$ and
$S^{in} = \{ |z|^{l_{(1)} + l_{(2)}} = r^{l_{(2)}} |w|, |w| < r|z|^{l_{(1)}} \}$.
Let $S^{out}_n = f^{-n} (S^{out})$ and $S^{in}_n = f^{-n} (S^{in})$.

\begin{lemma}
Fix any $z$ in $A_p - E_p$ and
let $l_1 \leq l_{(1)} < l_{(1)} + l_{(2)} \leq l_1 + l_2$.
Then
\[
G_z^{\alpha} |_{S_n^{out}} 
\begin{cases} 
\ \to \infty \text{ as } n \to \infty & \text{ if } l_{(1)} < \alpha, \\
\ \to 0 \text{ as } n \to \infty & \text{ if } l_{(1)} = \alpha \text{ and } d \geq 2, \\
\ \text{is bounded as } n \to \infty & \text{ if } l_{(1)} = \alpha \text{ and } d = 1,
\end{cases} 
\]
and
\[
G_z^{\alpha} |_{S_n^{in}} 
\begin{cases} 
\ \to 0 \text{ as } n \to \infty & \text{ if } l_{(1)} + l_{(2)} = \alpha \text{ and } d \geq 2, \\
\ \text{is bounded as } n \to \infty & \text{ if } l_{(1)} + l_{(2)} = \alpha \text{ and } d = 1, \\
\ \to - \infty \text{ as } n \to \infty & \text{ if } l_{(1)} + l_{(2)} > \alpha.
\end{cases} 
\]
\end{lemma}

In particular,
the following hold under the condition $d \geq 2$:
\begin{alignat*}{2} 
& G_z^{\alpha} |_{S^{out}_n} \to \infty \text{ and } G^{\alpha} |_{S^{in}_n} \to - \infty
\text{ as } n \to \infty \ & \text{ if } l_{(1)} < \alpha < l_{(1)} + l_{(2)}, \\
& G_z^{\alpha} |_{S^{out}_n} \to \infty \text{ and } G^{\alpha} |_{S^{in}_n} \to 0 
\text{ as } n \to \infty & \text{ if } l_{(1)} < \alpha = l_{(1)} + l_{(2)}, \\
& G_z^{\alpha} |_{S^{out}_n} \to 0 \ \text{ and } G^{\alpha} |_{S^{in}_n} \to - \infty
\text{ as } n \to \infty & \text{ if } l_{(1)} = \alpha < l_{(1)} + l_{(2)}.
\end{alignat*}

Let $S^{out}_{\infty}$ and $S^{in}_{\infty}$ 
be the sets of 
the accumulation points of $S_n^{out}$ and $S_n^{in}$
in $\overline{A_p} \times \mathbb{P}^1$ as $n \to \infty$,
respectively.
Let $U^{l_{(1)},l_{(2)},+} = \{ |z| < r^{1+1/l_{(2)}}, |w| < r|z|^{l_{(1)}} \}$ 
and $A_f^{l_{(1)},l_{(2)},+}$ 
the union of all the preimages of $U^{l_{(1)},l_{(2)},+}$.

\begin{lemma} \label{lem on asy behaviour for case 4}
If $f$ is non-degenerate,  
then 
for any $l_{(1)}$ and $l_{(2)}$ in Proposition \ref{Case 4: existence of G_z^alpha},
\begin{enumerate}
\item 
$\displaystyle \lim_{n \to \infty} f^{-n} (\partial U^{l_{(1)},l_{(2)}}) = S^{out}_{\infty} \cup S^{in}_{\infty} = \partial A_f^{l_{(1)},l_{(2)}}$,
\item 
$S^{out}_{\infty} \subset \partial A_f^{l_{(1)},l_{(2)},+}  \subset \partial A_f^{l_{(1)},l_{(2)}}$,
\item
$S^{out}_{\infty} \cap (A_p \times \mathbb{P}^1) = \partial A_f^{l_{(1)},l_{(2)},+}  \cap (A_p \times \mathbb{P}^1)$,
\item 
$\partial A_f^{l_{(1)},l_{(2)}} - \partial A_f^{l_{(1)},l_{(2)},+}  \subset S^{in}_{\infty} \subset \partial A_f^{l_{(1)},l_{(2)}}$,
\item
$S^{in}_{\infty} \cap U^{l_{(1)},l_{(2)},+}  = \partial A_f^{l_{(1)},l_{(2)}} \cap U^{l_{(1)},l_{(2)},+}$.
\end{enumerate}
\end{lemma} 

Note that this lemma holds even if we replace
$U^{l_{(1)},l_{(2)},+} $ and $A_f^{l_{(1)},l_{(2)},+}$ to
$U^{l_{(1)},+} $ and $A_f^{l_{(1)},+}$.
Hence
the two lemmas above imply the following.

\begin{proposition} \label{asy behavior when T_k < delta  T_k-1} 
For any $(z_0, w_0)$ in $\partial A_f^{l_{(1)},+} \cap \{ (A_p - E_p \cup E_{deg}^*) \times \mathbb{P}^1 \}$,
\[
\limsup_{(z,w) \in A_f^{l_{(1)},l_{(2)}} \to (z_0, w_0)} G_z^{\alpha} (w) = 
\begin{cases} 
\ \infty & \text{ if } l_1 \leq l_{(1)} < \alpha, \\
\ 0 & \text{ if } l_{(1)} = \alpha \text{ and } d \geq 2,
\end{cases}
\]
and $G_z^{\alpha}$ is bounded as $(z,w)$ in $A_f^{l_{(1)},l_{(2)}}$ tends to $(z_0, w_0)$
if $l_{(1)} = \alpha$ and $d = 1$.

For any $(z_0, w_0)$ in $\partial A_f^{l_{(1)},l_{(2)}} \cap A_0 - \partial A_f^{l_{(1)},+} \cup E_z \cup E_{deg}$,
\[
\lim_{(z,w) \in A_f^{l_{(1)},l_{(2)}} \to (z_0, w_0)} G_z^{\alpha} (w) = 
\begin{cases} 
\ - \infty & \text{ if } \alpha < l_{(1)} + l_{(2)} \leq l_1 + l_2, \\
\ 0 & \text{ if } \alpha = l_{(1)} + l_{(2)} \text{ and } d \geq 2, 
\end{cases}
\]
and $G_z^{\alpha}$ is bounded as $(z,w)$ in $A_f^{l_{(1)},l_{(2)}}$ tends to $(z_0, w_0)$
if $l_{(1)} + l_{(2)} = \alpha$ and $d = 1$.
\end{proposition}

Note that we can replace
$\partial A_f^{l_{(1)},l_{(2)}} \cap A_0 - \partial A_f^{l_{(1)},+} \cup E_z \cup E_{deg}$ in the proposition to
$\{ \partial A_f^{l_{(1)},l_{(2)}} - \partial A_f^{l_{(1)},+} \} \cap \{ (A_p - E_p \cup E_{deg}^*)  \times \mathbb{P}^1 \}$.


\begin{corollary} \label{cor on A_f^l when T_k < delta  T_k-1}
It follows that
$A_f^{l_{(1)}, l_{(2)}} = A_f^{l_1, l_2}$ 
for any $l_1 \leq l_{(1)} < \alpha$
and $\alpha < l_{(1)} + l_{(2)} \leq l_1 + l_2$.
Moreover,
$A_f^{l_1, \alpha - l_1} \subsetneq A_f^{l_1, l_2}$,
$A_f^{\alpha, l_1 + l_2 - \alpha} \subsetneq A_f^{l_1, l_2}$ and
$A_f^{l_1, \alpha - l_1} \cap A_f^{\alpha, l_1 + l_2 - \alpha} = \emptyset$.
\end{corollary}


\subsubsection{$\delta = T_k$} 

Recall that $l_1 < \alpha$ and 
$\mathcal{I}_f = [l_1, \alpha) \times \{ \alpha \}$
if $\delta = T_k$.

Let $S^{out} = \{ |z|^{l_{} + \alpha} < r^{\alpha} |w|, |w| = r|z|^{l_{}} \}$ 
and $S^{in} = \{ |z|^{l_{} + \alpha} = r^{\alpha} |w|, |w| < r|z|^{l_{}} \}$.

\begin{lemma}
Let $d \geq 2$.
It follows for any $l$ in $\mathcal{I}_f^1$ and
for any $z$ in $A_p - E_p$
that $G_z^{\alpha} |_{S^{out}_n} \to \infty$
and $G_z^{\alpha} |_{S^{in}_n} \to 0$ as $n \to \infty$.
\end{lemma}

This lemma and Lemma \ref{lem on asy behaviour for case 4} imply the following.

\begin{proposition} \label{asy behavior when delta = T_k}
Let $d \geq 2$. 
For any $l$ in $\mathcal{I}_f^1$ and 
for any $(z_0, w_0)$ in $\partial A_f^{l_{},+} \cap \{ (A_p - E_p \cup E_{deg}^*) \times \mathbb{P}^1 \}$,
\[
\limsup_{(z,w) \in A_f^{l_{}, \alpha - l} \to (z_0, w_0)} G_z^{\alpha} (w) = \infty. 
\]
For any $l$ in $\mathcal{I}_f^1$ and 
for any $(z_0, w_0)$ in $\partial A_f^{l_{}, \alpha - l} \cap A_0 - \partial A_f^{l_{},+} \cup E_z \cup E_{deg}$,
\[
\lim_{(z,w) \in A_f^{l_{}, \alpha - l} \to (z_0, w_0)} G_z^{\alpha} (w) = 0. 
\]
\end{proposition}


\begin{corollary} \label{cor on A_f^l when delta = T_k}
If $d \geq 2$,
then $A_f^{l_{}, \alpha - l_{}} = A_f^{l_1, \alpha - l_1}$ 
for any $l$ in $\mathcal{I}_f^1$. 
\end{corollary}


\subsubsection{$\delta = T_{k-1}$} 

Recall that 
$\mathcal{I}_f^1 = \{ \alpha \}$ and
$\mathcal{I}_f^2 = (0, l_2]$ if $\delta = T_{k-1}$.

Let $S^{out} = \{ |z|^{\alpha + l_{}} < r^{l_{}} |w|, |w| = r|z|^{\alpha} \}$
and $S^{in} = \{ |z|^{\alpha + l_{}} = r^{l_{}} |w|, |w| < r|z|^{\alpha} \}$.

\begin{lemma}
Let $d \geq 2$.
It follows for any $l$ in $\mathcal{I}_f^2$ and
for any $z$ in $A_p - E_p$
that $G_z^{\alpha} |_{S^{out}_n} \to 0$
and $G_z^{\alpha} |_{S^{in}_n} \to - \infty$ as $n \to \infty$.
\end{lemma}

This lemma and Lemma \ref{lem on asy behaviour for case 4} imply the following.

\begin{proposition} \label{asy behavior when delta = T_k-1}
Let $d \geq 2$.
For any $(z_0, w_0)$ in $\partial A_f^{\alpha, +} \cap \{ (A_p - E_p \cup E_{deg}^*) \times \mathbb{P}^1 \}$,
\[
\limsup_{(z,w) \in A_f^{\alpha, +} \to (z_0, w_0)} G_z^{\alpha} (w) = 0. 
\]
For any $l$ in $\mathcal{I}_f^2$ and 
for any $(z_0, w_0)$ in $\partial A_f^{\alpha, l - \alpha} \cap A_0 - \partial A_f^{\alpha, +} \cup E_z \cup E_{deg}$,
\[
\lim_{(z,w) \in A_f^{\alpha, l - \alpha} \to (z_0, w_0)} G_z^{\alpha} (w) = - \infty. 
\]
\end{proposition}


\begin{corollary}
If $d \geq 2$,
then $A_f^{\alpha, l_{} - \alpha} = A_f^{\alpha, l_2 - \alpha}$ 
for any $l$ in $\mathcal{I}_f^2$. 
\end{corollary}

\begin{remark}
If $\delta = T_{k-1}$, then
$f$ has the other dominant term $(n_{k-1}, m_{k-1})$, 
which belongs to Case 3 or Case 4.
Therefore,
variants of Theorems \ref{main thm on G_z^a,+ for case 3} and
\ref{main thm on G_z^a,+ for case 4} imply that
$G_z^{\alpha, +}$ is plurisubharmonic on $A_0 - E_z \cup E_{deg}$.
Moreover,
variants of Corollaries \ref{cor of G_f^a for Case 3} and
\ref{cor of G_f^a for Case 4} imply that
$G_f^{\alpha} = \alpha G_p$ on $A_0$ if $\delta > m_{k-1}$, and
$G_f^{\alpha}$ is plurisubharmonic and
continuous on $A_0 - E_{deg}$
if $\delta = m_{k-1}$.
\end{remark}


\end{document}